\documentclass{cimart}

\usepackage{longtable}
\usepackage{multicol}
\usepackage[table]{xcolor}

\newcommand{\ad}[1]{\mathrm{ad}_{#1}}
\newcommand{\bt}{\mathfrak{b}}
\newcommand{\Cp}{\mathbb{C}}
\newcommand{\Fp}{\mathbb{F}}
\newcommand{\gt}{\mathfrak{g}}
\newcommand{\Np}{\mathbb{N}}
\newcommand{\nt}{\mathfrak{n}}
\newcommand{\Ou}{\mathcal{O}}
\newcommand{\wt}{\widetilde}

\DeclareMathOperator{\rk}{rk}
\DeclareMathOperator{\supp}{supp}

\title{Classification of coadjoint orbits for the maximal unipotent subgroup in the simple group of type $F_4$}

\authors{Matvey Surkov}

\authorinfo{Samara National Research University and Saint Petersburg University, Russia}{surkovmatveya@yandex.ru}

\abstract{%
   Let $N$ be the maximal unipotent subgroup of the simple algebraic group of type $\Phi$. It naturally acts on the space $\nt^*$ dual to the Lie algebra $\nt$ of $N$, and this action is called the coadjoint action. Such orbits play a key role in A.A. Kirillov's orbit method. In this work, we classify the orbits of this action in the case of $\Phi=F_4$ in terms of supports of canonical forms. This means that we will present a set $S$ of linear forms from $\nt^*$ such that for any coadjoint orbit there exists a unique form from $S$ belonging to that orbit. The set of canonical forms will be explicitly described in terms of supports. The support of a form $\lambda\in\nt^*$ is the set of positive roots $\alpha\in\Phi^+$ such that $\lambda(e_\alpha)\neq 0$.
    }



\VOLUME{34}
\YEAR{2026}
\ISSUE{1}
\NUMBER{4}
\DOI{https://doi.org/10.46298/cm.16015}
\licence{Creative Commons Attribution-ShareAlike 4.0}

\editinfo{July 8, 2025}{October 30, 2025}{Ivan Kaygorodov}

\acknowledgments{This work was supported by the Ministry of Science and Higher Education of the Russian Federation (agreement no. 075--15--2022--287). I thank M.~V.~Ignatev and A.~N.~Panov for helpful comments and support during the writing of this paper.}

\begin{document}

\section{Introduction and the main result}

Let $\gt$ be a simple finite-dimensional Lie algebra over a finite field $\Fp_q$ of sufficiently large characteristic $p$, $\bt$ be a Borel subalgebra of $\gt$, $\Phi$ be the root system of $\gt$, $\Phi^+$ be the set of positive roots corresponding to $\bt$, $\nt$ be the nilradical of $\bt$, $N=\exp(\nt)$ be the corresponding unipotent algebraic group, and $\nt^*$ be the dual space of $\nt$. The group $N$ is also a maximal unipotent subgroup of the simple algebraic group $G$ of type $\Phi$, or equivalently a Sylow $p$-subgroup of $G$. The group $N$ acts on $\nt$ by the adjoint action; the dual action of $N$ on the space $\nt^*$ is called the coadjoint action; we will denote the result of this action by $g\cdot\lambda$ for $g\in N$, $\lambda\in\nt^*$. According to the orbit method introduced by A.A. Kirillov in 1962, coadjoint orbits play the key role in the representation theory of the group $N$. Almost all coadjoint orbits that have been studied well so far are so-called orbits associated with rook placements. They are discussed in detail in the works \cite{Andre95, AndreNeto06,  Kostant12, Kostant13, Panov08, IgnatyevPanov09, Ignatev09, Ignatev11, Ignatyev12, IgnatyevShevchenko21, IgnatyevSurkov23}.

Explicitly, the coadjoint action is given by:
\begin{equation*}
\begin{split}
(\exp(x)\cdot\lambda)(y)=\lambda(\exp(-\ad{x})(y))=\lambda(y) - \lambda([x, y]) + \dfrac {1}{2!}\lambda([x, [x, y]]) - \ldots,\\
x,y\in\nt,~\lambda\in\nt^*.
\end{split}
\end{equation*}

Let $\{e_\alpha, \alpha\in\Phi^+\}$ be a basis of $\nt$ consisting of root vectors, and let $\{e_\alpha^*, \alpha\in\Phi^+\}$ be the dual basis in $\nt^*$. The support of a linear form $\lambda\in\nt^*$ is defined as the subset of positive roots given by $$\supp(\lambda)=\{\alpha\in\Phi^+\colon\lambda(e_\alpha)\neq0\}.$$

Our goal is to describe the coadjoint orbits for the case of $\Phi=F_4$. The description will be given in terms of supports of canonical forms. A list of supports will be provided with the following properties.
\begin{itemize}
    \item For any two linear forms $\lambda_1\neq\lambda_2$ whose supports belong to the given list (even if the supports coincide), the orbits $N\cdot\lambda_1$ and $N\cdot\lambda_2$ are different.
    \item For any linear form $\lambda\in\nt^*$, there exists a form $\lambda'$ with a support from the given list such that $N\cdot\lambda = N\cdot\lambda'$.
\end{itemize}
In other words, each coadjoint orbit contains a unique linear form whose support belongs to the given list.

The set of simple roots of $F_4$ can be identified with the following subset of $\mathbb{R}^4$:
\begin{equation*}
\Delta=\{\alpha_1=\varepsilon_2-\varepsilon_3,~\alpha_2=\varepsilon_3-\varepsilon_4,~\alpha_3=\varepsilon_4,~\alpha_4=\frac{1}{2}(\varepsilon_1-\varepsilon_2-\varepsilon_3-\varepsilon_4)\}. \end{equation*}
Here $\{\varepsilon_i\}_{i=1}^4$ is the standard basis of $\mathbb{R}^4$ (endowed with the standard scalar product).

Recall that there exists a natural partial order on $\Phi^+$ defined as follows: $\alpha>\beta$ if $\alpha-\beta$ is a sum of positive roots. Let us fix a linear lexicographic order $\succ$ on the positive roots of $F_4$. We say that $\alpha = a_1\alpha_1+a_2\alpha_2+a_3\alpha_3+a_4\alpha_4\succ\beta = b_1\alpha_1+b_2\alpha_2+b_3\alpha_3+b_4\alpha_4$ if $a_1>b_1$ or $a_1=b_1, a_2>b_2$ or $a_1=b_1, a_2=b_2, a_3>b_3$ or $a_1=b_1, a_2=b_2, a_3=b_3, a_4>b_4$. Thus, the roots will be ordered as follows:
\begin{equation*}
\begin{split}
\Phi^+=\{&\alpha_4,~\alpha_3,~\alpha_3+\alpha_4,~\alpha_2,~\alpha_2+\alpha_3,~\alpha_2+\alpha_3+\alpha_4,~\alpha_2+2\alpha_3,~\alpha_2+2\alpha_3+\alpha_4,\\
&\alpha_2+2\alpha_3+2\alpha_4,~\alpha_1,~\alpha_1+\alpha_2,~\alpha_1+\alpha_2+\alpha_3,~\alpha_1+\alpha_2+\alpha_3+\alpha_4,~\alpha_1+\alpha_2+2\alpha_3,\\
&\alpha_1+\alpha_2+2\alpha_3+\alpha_4,~\alpha_1+\alpha_2+2\alpha_3+2\alpha_4,~\alpha_1+2\alpha_2+2\alpha_3,~\alpha_1+2\alpha_2+2\alpha_3+\alpha_4,\\
&\alpha_1+2\alpha_2+2\alpha_3+2\alpha_4,~\alpha_1+2\alpha_2+3\alpha_3+\alpha_4,~\alpha_1+2\alpha_2+3\alpha_3+2\alpha_4,\\
&\alpha_1+2\alpha_2+4\alpha_3+2\alpha_4,~\alpha_1+3\alpha_2+4\alpha_3+2\alpha_4,~2\alpha_1+3\alpha_2+4\alpha_3+2\alpha_4\}.
\end{split}
\end{equation*}
Clearly, $\alpha>\beta$ implies $\alpha\succ\beta$. Further, if we talk about order on positive roots, we will always mean the lexicographic order as above, unless otherwise stated.

Let $\lambda\in\nt^*$, $\gamma\in\Phi^+$. Consider the following matrices.
\begin{equation*}
\begin{split}
    A_{\lambda,\gamma}:=\left(\lambda([e_\alpha, e_\beta])\right)_{\alpha\in\Phi^+, \alpha\succeq\gamma, \beta\in\Phi^+}\cdot\\
    B_{\lambda,\gamma}:=\left(\lambda([e_\alpha, e_\beta])\right)_{\alpha\in\Phi^+, \alpha\succ\gamma, \beta\in\Phi^+}.
\end{split}
\end{equation*}
The columns of the matrices are numbered by all positive roots, and the rows are numbered by certain positive roots. The matrix $A_{\lambda,\gamma}$ is the matrix $B_{\lambda,\gamma}$ with one more row. Further, let $S$ be the set of $\lambda\in\nt^*$ for which the following condition holds: for any $\gamma\in \supp(\lambda)$,
\begin{equation}\label{rk}
\rk A_{\lambda,\gamma} = \rk B_{\lambda,\gamma}.
\end{equation}
Finally, we can formulate the main result of this work.
\begin{theorem}
    \label{list}
For any coadjoint orbit\textup{,} there exists a unique linear form $\lambda\in S$ lying on this orbit.
\end{theorem}

This theorem generalizes the result of the work \cite{IgnatyevSurkov23}.  This article discusses the so-called orbits associated with rook placements. By definition, a subset $D$ of $\Phi^+$ is called a non-singular rook placement if $\alpha-\beta\notin \Phi^+$ for all distinct $\alpha,\beta\in D$. It follows that $(\alpha,\beta)\leq0$ for all distinct $\alpha,\beta\in D$. Given a rook placement $D$ and a map $\xi\colon D\to\Cp^{\times}$, we put
\begin{equation*}
f_{D,\xi}=\sum_{\alpha\in D}\xi(\alpha)e_{\alpha}^*\in\nt^*.
\end{equation*}
We say that the coadjoint orbit $\Omega_{D,\xi}$ of the linear form $f_{D,\xi}$ is associated with the rook placement $D$ and the map $\xi$. Thus, $f_{D,\xi}$ is a linear form with support $D$.

In \cite{IgnatyevSurkov23} it is proved that if $\Phi=F_4$, $D$ is an orthogonal non-singular rook placement and $\xi_1$, $\xi_2$ are distinct maps from $D$ to $\Cp^{\times}$ then $\Omega_{D,\xi_1}\neq\Omega_{D,\xi_2}$. Moreover, if $D$ is a non-singular rook placement, then $\lambda = f_{D,\xi}$ belongs to $S$. Indeed, if $\gamma\in D$, then $\gamma+\beta\notin D$ for all $\beta\in\Phi^+$ and $\lambda([e_\gamma,e_\beta])=0$ for all $\beta\in\Phi^+$. Hence, in the matrix $A_{\lambda,\gamma}$ the row corresponding to $\gamma$ is zero, and condition (\ref{rk}) is true for $\gamma\in D$. Therefore, distinct linear forms with the same support $D$ lie in different coadjoint orbits.

A result similar to Theorem \ref{list} holds for the root system $\Phi=G_2$ and field $\Cp$. In the same work \cite{IgnatyevSurkov23}, the equations defining the basic subvarieties $\Ou_{D,\xi}$ are given. By definition, the basic subvariety $\Ou_{D,\xi}$ corresponding to a rook placement $D$ and a map $\xi\colon D\to\Cp^{\times}$ is $$\Ou_{D,\xi}=\sum_{\alpha\in D}\Omega_{\{\alpha\},\xi_{\alpha}},$$ where $\xi_{\alpha}$ is the restriction of $\xi$ to $\{\alpha\}$. That work also proves André's stratification for $G_2$:
$$\nt^*=\bigsqcup_{D,\xi}\Ou_{D,\xi},$$ where the union is taken over all non-singular rook placements $D$ and all maps $\xi\colon D\to\Cp^{\times}$.

In this case $\Phi^+=\{ \alpha, \beta, \alpha+\beta, 2\alpha+\beta, 3\alpha+\beta, 3\alpha+2\beta \}$. Let us introduce the notation of structural constants as in the work \cite{IgnatyevSurkov23}:
\begin{gather*}
    e_\alpha, e_\beta] = c_1\cdot e_{\alpha+\beta}, \quad
    [e_\alpha, e_{\alpha+\beta}] = c_2\cdot e_{2\alpha+\beta}, \quad
    [e_\alpha, e_{2\alpha+\beta}] = c_3\cdot e_{3\alpha+\beta}, \\
    [e_{3\alpha+\beta}, e_\beta] = c_4\cdot e_{3\alpha+2\beta} \quad \text{and} \quad
    [e_{\alpha+\beta}, e_{2\alpha+\beta}] = c_5\cdot e_{3\alpha+2\beta}.
\end{gather*}

For almost every rook placement $D$, we have $\Ou_{D,\xi} = \Omega_{D,\xi}$. The only exception is the case $D=\{ \alpha+\beta, 3\alpha+\beta\}$. In this case $\Ou_{D,\xi}$ is the union of orbits $\Omega_{D,\xi}$ and $\Omega_{\{ \beta, \alpha+\beta, 3\alpha+\beta\},\xi'}$ where $\xi'\colon {\beta, \alpha+\beta, 3\alpha+\beta} \to \Cp^{\times}$ is such that its restriction to $D$ coincides with $\xi$. Let us write down the equations defining the orbits for $D = {\alpha+\beta, 3\alpha+\beta}$. A linear form $\lambda = \sum_{\gamma \in \Phi^+} \lambda_{\gamma} e_{\gamma}^*$ lies in $\Omega_{D,\xi}$ if and only if it satisfies the following system:
\begin{equation*}
    \begin{split}
        6c_3^2\lambda_{3\alpha+\beta}^2\lambda_\beta-6c_1c_3\lambda_{2\alpha+\beta}\lambda_{3\alpha+\beta}\lambda_{\alpha+\beta}+5c_1c_2\lambda_{2\alpha+\beta}^3={ }&{ }0,\\
        2c_3\lambda_{\alpha+\beta}\lambda_{3\alpha+\beta}-c_2\lambda_{2\alpha+\beta}^2={ }&{ }2c_3\xi(\alpha+\beta)\xi(3\alpha+\beta),\\
        \lambda_{3\alpha+\beta}={ }&{ }\xi(3\alpha+\beta),\\
        \lambda_{3\alpha+2\beta}={ }&{ }0.
    \end{split}
\end{equation*}
For a linear form $\lambda\in\Omega_{\{ \beta, \alpha+\beta, 3\alpha+\beta\},\xi'}$ we have the following system of
equations.
\begin{equation*}
    \begin{split}
        6c_3^2\lambda_{3\alpha+\beta}^2\lambda_\beta-6c_1c_3\lambda_{2\alpha+\beta}\lambda_{3\alpha+\beta}\lambda_{\alpha+\beta}+5c_1c_2\lambda_{2\alpha+\beta}^3={ }&{ } 6c_3^2(\xi'(3\alpha+\beta))^2\xi'(\beta),\\
        2c_3\lambda_{\alpha+\beta}\lambda_{3\alpha+\beta}-c_2\lambda_{2\alpha+\beta}^2={ }&{ }2c_3\xi'(\alpha+\beta)\xi'(3\alpha+\beta),\\
        \lambda_{3\alpha+\beta}={ }&{ }\xi'(3\alpha+\beta),\\
        \lambda_{3\alpha+2\beta}={ }&{ }0.
    \end{split}
\end{equation*}

In the case $\Phi=G_2$ the set $S$ consists of linear forms with supports from this list:
\begin{equation*}
\begin{split}
\emptyset, \{\alpha\}, \{\beta\}, \{\alpha,\beta\}, \{\alpha+\beta\}, \{2\alpha+\beta\}, \{2\alpha+\beta, \beta\}, \{3\alpha+\beta\},\\ \{3\alpha+\beta, \alpha+\beta\}, \{3\alpha+\beta, \alpha+\beta, \beta\},
\{3\alpha+\beta, \beta\}, \{3\alpha+2\beta\}, \{3\alpha+2\beta, \alpha\}.
\end{split}
\end{equation*}

The orbits of these forms correspond precisely to all the orbits whose equations are provided in \cite{IgnatyevSurkov23} and in the present work.

\section{Description of the set of canonical forms}

In this section, we will obtain an explicit description of the set $S$ to prove Theorem \ref{list} and to demonstrate its use. At first glance, it might seem that this set does not admit a nice description. In fact, this is not the case. The set $S$ can be described very neatly in terms of supports. A list of 880 supports is presented in the Appendix.

\begin{proposition}
    \label{S}
The set $S$ consists of all linear forms with supports from the list and forms with one of the three supports
\begin{equation*}
\begin{split}
    &\{ \alpha_1+\alpha_2+2\alpha_3,~\alpha_1+\alpha_2,~\alpha_2+2\alpha_3+\alpha_4, ~\alpha_2+\alpha_3+\alpha_4,~\alpha_2 \},\\
    &\{ \alpha_1+\alpha_2+2\alpha_3,~\alpha_1+\alpha_2,~\alpha_2+2\alpha_3+\alpha_4, ~\alpha_2+\alpha_3+\alpha_4,~\alpha_2,~\alpha_4 \},\\
    &\{ \alpha_1+\alpha_2+2\alpha_3,~\alpha_1+\alpha_2,~\alpha_2+2\alpha_3+\alpha_4, ~\alpha_2+\alpha_3+\alpha_4,~\alpha_4 \},
\end{split}
\end{equation*}
and of the forms with one of the three supports listed below, provided their coordinates satisfy the equation
\begin{equation*}
\lambda_{\alpha_1+\alpha_2}\lambda^2_{\alpha_2+2\alpha_3+\alpha_4} = \lambda_{\alpha_1+\alpha_2+2\alpha_3}\lambda^2_{\alpha_2+\alpha_3+\alpha_4}.
\end{equation*}
Here $\lambda = \sum_{\alpha\in\Phi^+} \lambda_\alpha e_\alpha^*$, and $\lambda_\alpha$ are the coordinates with respect to the dual basis.
\end{proposition}

This section is devoted to the proof of this proposition.

In what follows, we will use the following table. It is a $24\times24$ table whose rows and columns are indexed by the positive roots. In the cell at the intersection of row $\alpha$ and column $\beta$, we write $\beta-\alpha$ if $\beta-\alpha$ can be expressed as a sum of positive roots; otherwise, the cell is left blank. The main diagonal contains zeros. It is easy to see that all cells below the diagonal are empty. Below we show the upper-left corner of this table, corresponding to those positive roots whose simple root expansion does not contain $\alpha_1$. A cell at position $(\alpha,\beta)$ is highlighted in orange if the difference $\beta-\alpha$ is itself a positive root. Almost all examples will involve specifically this part of the table. Four columns are highlighted in blue, as they will appear most frequently in the examples.

\begin{center}
    \textbf{Table 1}
\end{center}
{\tiny

\newcolumntype{c}{>{\columncolor[HTML]{b8e4ff}} l}

\begin{longtable}{||l||c|l|c|c|l|l|c|l|l||}
\hline
$e\backslash e^*$&$\alpha_4$&$\alpha_3$&$\alpha_3+\alpha_4$&$\alpha_2$&$\alpha_2+\alpha_3$&
$\alpha_2+\alpha_3+\alpha_4$&$\alpha_2+2\alpha_3$&$\alpha_2+2\alpha_3+\alpha_4$&$\alpha_2+2\alpha_3+2\alpha_4$\\
\hline \hline
$\alpha_4$&$\cellcolor[HTML]{C0C0C0}0$& &$\cellcolor[HTML]{ffa500}\alpha_3$& & &$\cellcolor[HTML]{ffa500}\alpha_2+\alpha_3$& &$\cellcolor[HTML]{ffa500}\alpha_2+2\alpha_3$&$\cellcolor[HTML]{ffa500}\alpha_2+2\alpha_3+\alpha_4$\\
\hline
$\alpha_3$& &$\cellcolor[HTML]{C0C0C0}0$&$\cellcolor[HTML]{ffa500}\alpha_4$& &$\cellcolor[HTML]{ffa500}\alpha_2$&$\alpha_2+\alpha_4$&$\cellcolor[HTML]{ffa500}\alpha_2+\alpha_3$&$\cellcolor[HTML]{ffa500}\alpha_2+\alpha_3+\alpha_4$&$\alpha_2+\alpha_3+2\alpha_4$\\
\hline
$\alpha_3+\alpha_4$& & & $\cellcolor[HTML]{C0C0C0}0$& & &$\cellcolor[HTML]{ffa500}\alpha_2$& &$\cellcolor[HTML]{ffa500}\alpha_2+\alpha_3$&$\cellcolor[HTML]{ffa500}\alpha_2+\alpha_3+\alpha_4$\\
\hline
$\alpha_2$& & & &$\cellcolor[HTML]{C0C0C0}0$&$\cellcolor[HTML]{ffa500}\alpha_3$&$\cellcolor[HTML]{ffa500}\alpha_3+\alpha_4$&$2\alpha_3$&$2\alpha_3+\alpha_4$&$2\alpha_3+2\alpha_4$\\
\hline
$\alpha_2+\alpha_3$& & & & &$\cellcolor[HTML]{C0C0C0}0$&$\cellcolor[HTML]{ffa500}\alpha_4$&$\cellcolor[HTML]{ffa500}\alpha_3$&$\cellcolor[HTML]{ffa500}\alpha_3+\alpha_4$&$\alpha_3+2\alpha_4$\\
\hline
$\alpha_2+\alpha_3+\alpha_4$& & & & & &$\cellcolor[HTML]{C0C0C0}0$& &$\cellcolor[HTML]{ffa500}\alpha_3$&$\cellcolor[HTML]{ffa500}\alpha_3+\alpha_4$\\
\hline
$\alpha_2+2\alpha_3$& & & & & & &$\cellcolor[HTML]{C0C0C0}0$&$\cellcolor[HTML]{ffa500}\alpha_4$&$2\alpha_4$\\
\hline
$\alpha_2+2\alpha_3+\alpha_4$& & & & & & & &$\cellcolor[HTML]{C0C0C0}0$&$\cellcolor[HTML]{ffa500}\alpha_4$\\
\hline
$\alpha_2+2\alpha_3+2\alpha_4$& & & & & & & & &$\cellcolor[HTML]{C0C0C0}0$\\
\hline
\end{longtable}
}

To describe the set $S$, we wrote a program, which is included in the Appendix. The program outputs supports that satisfy a sufficient condition for being included in the list, and supports that satisfy a necessary condition; these conditions will be described below.

First, let us reformulate condition (\ref{rk}). The algebra $\nt$ acts on $\nt^*$ as follows: for $x,y\in\nt$ and $\lambda\in\nt^*$, define $(x\cdot\lambda)(y) := -\lambda([x,y])$. This action is called tangent.

Let us expand $x\in\nt$ in the chosen basis: $$x = \sum_{\alpha\in\Phi^+}x_\alpha e_\alpha.$$ Let $\lambda\in\nt^*$, $\beta\in\Phi^+$. Consider the coordinates of $x\cdot\lambda$ in the dual basis: $$(x\cdot\lambda)(e_\beta) = -\lambda([\sum_{\alpha\in\Phi^+}x_\alpha e_\alpha, e_\beta]) = -\sum_{\alpha\in\Phi^+}\lambda([e_\alpha,e_\beta]) x_\alpha.$$ We can view these coordinates as linear combinations of variables $x_\alpha, \alpha\in\Phi^+$, with constant coefficients $-\lambda([e_\alpha,e_\beta])$. By definition, $-A_{\lambda, \gamma}$ is the matrix of linear combinations $\{(x\cdot\lambda)(e_\alpha)\}_{\alpha\succeq\gamma}$, and $-B_{\lambda, \gamma}$ is the matrix of linear combinations $\{(x\cdot\lambda)(e_\alpha)\}_{\alpha\succ\gamma}$. Clearly, $\rk A_{\lambda, \gamma} = \rk B_{\lambda, \gamma}$ if and only if the row numbered $\gamma$ of the matrix $A_{\lambda, \gamma}$ is a linear combination of all its other rows (i.e., the rows of the matrix $B_{\lambda, \gamma}$). In other words, $(x\cdot\lambda)(e_\gamma)$ can be expressed as a linear combination of $\{(x\cdot\lambda)(e_\alpha)\}_{\alpha\succ\gamma}$.

\begin{example}
    \label{ex}
Let $D=\supp(\lambda)=\{\alpha_4,\alpha_3+\alpha_4,\alpha_2,\alpha_2+2\alpha_3\}$. We will show that such a $\lambda$ lies in $S$. To visualize what the orbit of the tangent action looks like, let us refer to the table. The highlighted blue columns correspond to the roots in $D$. The linear combination $(x\cdot\lambda)(e_\beta)$ contains only those $x_\alpha$ for which $\lambda([e_\alpha,e_\beta])\neq0$, which is equivalent to $\alpha+\beta\in D$. It can be seen from the construction of the table that $(x\cdot\lambda)(e_\beta)$ contains only $x_\alpha$ such that the orange root $\alpha$ is in the row corresponding to $\beta$ in the highlighted column. For example, the combination $(x\cdot\lambda)(e_{\alpha_3})$ contains only $x_{\alpha_4}$ and $x_{\alpha_2+\alpha_3}$.

Now let us check condition (\ref{rk}) for all roots from $D$. As you can see from the table, $(x\cdot\lambda)(e_{\alpha_2+2\alpha_3}) = 0$. Naturally, this will always be the case for the maximal root in the support. Note that $(x\cdot\lambda)(e_{\alpha_2+\alpha_3})$ contains only $x_{\alpha_3}$. Further, $(x\cdot\lambda)(e_{\alpha_2}) = 0$, because there are no orange cells in the blue columns in the row $\alpha_2$; $(x\cdot\lambda)(e_{\alpha_3+\alpha_4}) = 0$ for similar reasons. Finally, $(x\cdot\lambda)(e_{\alpha_4})$ contains only $x_{\alpha_3}$, which means that it is a scalar multiple of $(x\cdot\lambda)(e_{\alpha_2+\alpha_3})$.

It is clear that we did not need to know the coefficients for $x_{\alpha_3}$ in either $(x\cdot\lambda)(e_{\alpha_2+\alpha_3})$ or $(x\cdot\lambda)(e_{\alpha_4})$. This is simply because $x_{\alpha_3}$ as a variable can be expressed in terms of $(x\cdot\lambda)(e_{\alpha_2+\alpha_3})$, and everywhere above (i.e., in $(x\cdot\lambda)(e_{\alpha_2})$, $(x\cdot\lambda)(e_{\alpha_3+\alpha_4})$, $(x\cdot\lambda)(e_{\alpha_3})$, $(x\cdot\lambda)(e_{\alpha_4})$) $x_{\alpha_3}$ can be expressed as $(x\cdot\lambda)(e_{\alpha_2+\alpha_3})$ multiplied by some coefficient. The Sufficient condition below is based on this idea.
\end{example}

Consider the tangent action $\nt\curvearrowright\nt^*$. Let $\gamma\in\Phi^+$, $\lambda\in\nt^*$, $D = \supp(\lambda)$. Starting from the support $D$ and the root $\gamma$, we will construct a set $\zeta_{D,\gamma}$ of positive roots. We will add roots to $\zeta_{D,\gamma}$ one by one. We add a root $\delta$ to $\zeta_{D,\gamma}$ if there exists a root $\beta\succ\gamma$ such that $\lambda([e_\delta, e_\beta])\neq0$ and if $\lambda([e_\alpha, e_\beta])\neq0$ for some other $\alpha\in\Phi^+$, then $\alpha\in\zeta_{D,\gamma}$. Using this procedure, we will add to $\zeta_{D,\gamma}$ the maximum possible number of roots. Note that $\zeta_{D,\gamma}$ does not depend on the specific values of the coordinates of $\lambda$ in the dual basis, but only on its support $D$. The set $\zeta_{D,\gamma}$ also does not depend on the order in which we add the roots. Indeed, define sets $R_n$ for $n\in\Np$. A root $\delta$ lies in $R_n$ if there exists a root $\beta\succ\gamma$ such that $\lambda([e_\delta, e_\beta])\neq0$ and if $\lambda([e_\alpha, e_\beta])\neq0$ for some other $\alpha\in\Phi^+$, then $\alpha\in R_i$ for some $i<n$. Clearly, we must add all roots from $\bigcup_{n\in\Np} R_n$ to $\zeta_{D,\gamma}$, and we cannot add any other roots.

\begin{proposition}\textup{(Sufficient condition)}\label{suf}
Let $\lambda\in\mathfrak{n}^*$, $D = \operatorname{supp}(\lambda)$. Suppose that for every root $\gamma\in D$, the following condition holds: if $\lambda([e_\alpha, e_\gamma])\neq0$, then $\alpha\in\zeta_{D,\gamma}$. Then $\lambda\in S$\end{proposition}
\begin{proof}
Fix a root $\gamma\in D$. Consider the orbit of the infinitesimal action $\mathfrak{n}\cdot\lambda$. Let $\mu\in\mathfrak{n}\cdot\lambda$, $\mu = x\cdot\lambda$, where $x = \sum_{\alpha\in\Phi^+}x_\alpha e_\alpha$. Here $\mu$ is given, while the values $x_\alpha$ are unknown.

We will show by induction on $|\zeta_{D,\gamma}|$ that if $\delta\in\zeta_{D,\gamma}$, then $x_\delta$ can be expressed as a linear combination of $\{\mu(e_\beta)\}_{\beta\succ\gamma}$. Consider the step of adding $\delta$ to $\zeta_{D,\gamma}$. Suppose we have already proved this for all roots that are already in $\zeta_{D,\gamma}$. There exists a root $\beta\succ\gamma$ such that $\delta$ is the unique root not in $\zeta_{D,\gamma}$ satisfying $\lambda([e_\delta, e_\beta])\neq0$. Consider $\mu(e_\beta) = -\sum_{\alpha\in\Phi^+}\lambda([e_\alpha, e_\beta])x_\alpha$. This can be rewritten as: $$\lambda([e_\delta, e_\beta])x_\delta = -\mu(e_\beta)-\sum_{\alpha\neq\delta}\lambda([e_\alpha, e_\beta])x_\alpha.$$ We claim that the right-hand side is a linear combination of $\{\mu(e_{\beta'})\}_{\beta'\succ\gamma}$. Indeed, if the coefficient of $x\alpha$ on the right-hand side is non-zero (i.e., $\lambda([e_\alpha, e_\beta])\neq 0$), then $\alpha\in\zeta_{D,\gamma}$. By the inductive hypothesis, $x_{\alpha}$ can be expressed as a linear combination of $\{\mu(e_{\beta'})\}_{\beta'\succ\gamma}$. Since $\lambda([e_\delta, e_\beta])\neq 0$, it follows that $x_\delta$ can also be expressed as a linear combination of $\{\mu(e_{\beta'})\}_{\beta'\succ\gamma}$.

Now consider $\mu(e_\gamma) = -\sum_{\alpha\in\Phi^+}\lambda([e_\alpha, e_\gamma])x_\alpha$. By the hypothesis of the proposition, if $\lambda([e_\alpha,e_\gamma])\neq 0$, then $x_\alpha$ can be expressed as a linear combination of $\{\mu(e_{\beta'})\}_{\beta'\succ\gamma}$. Consequently, the same holds for $\mu(e_\gamma)$, which is equivalent to condition (\ref{rk}) for the root $\gamma$. The result follows.
\end{proof}

Note that if we take another form $\lambda'$ with support $D$, then the conditions of Proposition~\ref{suf} will also be satisfied for it. Therefore, for a given support $D$, either all forms with this support satisfy the conditions of Proposition \ref{suf}, or none do. Hence, it suffices to check the conditions for a single form with a given support, for example, for the form with all coordinates equal to $1$ at the basis vectors $e_\alpha^*$ with $\alpha\in D$.

\begin{proposition}\textup{(Necessary condition)}\label{nec}
Let $\lambda\in\nt^*$\textup, $D = \supp(\lambda)$\textup, $\gamma\in D$. Suppose that there exist positive roots $\delta_1, \delta_2, \dots \delta_k, \varepsilon_1, \varepsilon_2, \dots \varepsilon_k$ \textup(not necessarily distinct\textup) satisfying the conditions\textup:
\begin{itemize}
    \item[\textup1.] $\delta_1=\gamma$ and $\delta_i\succ\gamma$ for all $i=\overline{2,k}$\textup;
    \item[\textup2.] $\lambda([e_{\delta_i}, e_\beta])\neq0$ and $\beta\notin\zeta_{D,\gamma} \iff \beta = \varepsilon_{i-1}$ or $\beta = \varepsilon_i$ for all $i=\overline{1,k-1}$\textup;
    \item[\textup3.] $\lambda([e_\alpha, e_{\varepsilon_i}])\neq0$ and $\alpha\succeq\gamma \iff \alpha = \delta_i$ or $\alpha = \delta_{i+1}$ for all $i=\overline{1,k}$\textup;
    \item[\textup4.] $\lambda([e_{\delta_k}, e_{\varepsilon_k}])\neq0$ and $\varepsilon_k\notin\zeta_{D,\gamma}$.
\end{itemize}
Then $\lambda\notin S$.
\end{proposition}
\begin{proof}
It is enough to show that $(x\cdot\lambda)(e_\gamma)$ cannot be expressed as a linear combination of $\{(x\cdot\lambda)(e_\beta)\}_{\beta\succ\gamma}$. In other words, consider the orbit of the tangent action $\nt\cdot\lambda$, and let $\mu\in\nt\cdot\lambda$. It is enough to check that if all values $\mu(e_\beta)$ are known for $\beta\succ\gamma$, then the value of $\mu(e_\gamma)$ can be arbitrary. We will construct $\mu'\in\nt\cdot\lambda$ such that $\mu'(e_\beta) = \mu(e_\beta)$ for all $\beta\succ\gamma$ and $\mu'(e_\gamma) = A$, where $A$ is an arbitrary given value. Let $\mu' = x\cdot\lambda$, $x = \sum_{\alpha\in\Phi^+}x_\alpha e_\alpha$. It follows from the proof of Proposition \ref{suf} that if $\delta\in\zeta_{D,\gamma}$, then $x_\delta$ is uniquely defined by the values $\{\mu(e_\beta)\}_{\beta\succ\gamma}$. We define all other values $x_\alpha$, except $x_{\varepsilon_i}$, to be the same as for $\mu$. We will determine the values $x_{\varepsilon_i}$ by induction on $i=\overline{1,k-1}$.

\textbf{Base case}. Let $i=1$, $A = \sum_{\alpha\in\Phi^+}\lambda([e_{\delta_1}, e_\alpha])x_\alpha$. Clearly, if $\lambda([e_{\delta_1}, e_\alpha])\neq 0$, then either $\alpha=\varepsilon_1$ or $\alpha\in\zeta_{D,\gamma}$. Hence, in the right-hand side, all values $x_\alpha$ with nonzero $\lambda([e_{\delta_1}, e_\alpha])$  except possibly $x_{\varepsilon_1}$, are already defined. Since $\lambda([e_{\delta_1}, e_{\varepsilon_1}])\neq 0$, the value of $x_{\varepsilon_1}$ is uniquely determined as well. After this, the only expression $\mu'(e_\beta), \beta\succeq\gamma$ that may still depend on $x_{\varepsilon_1}$ is $\mu'(e_{\delta_2})$. If $k=1$, then nothing else depends on $x_{\varepsilon_1}$.

\textbf{Inductive step}. Suppose $x_{\varepsilon_{i-1}}$ was found. Consider $\mu(e_{\delta_i}) = \sum_{\alpha\in\Phi^+}\lambda([e_{\delta_i}, e_\alpha])x_\alpha$. On the right-hand side, only $x_{\varepsilon_{i-1}}, x_{\varepsilon_i}$ and $x_\alpha$, for which $\alpha\in\zeta_{D,\gamma}$, can appear with non-zero coefficients. The value $x_{\varepsilon_{i-1}}$ was determined in the previous step. Thus, on the right-hand side, only $x_{\varepsilon_i}$ may be undefined, and it definitely appears. From this equation, the value $x_{\varepsilon_1}$ is uniquely determined. The only $\mu'(e_\beta), \beta\succ\gamma$ that may still depend on $x_{\varepsilon_i}$ is $\mu'(e_{\delta_{i+1}})$.

\textbf{Case $i=k$}. Consider $\mu(e_{\delta_k}) = \sum_{\alpha\in\Phi^+}\lambda([e_{\delta_k}, e_\alpha])x_\alpha$. The right-hand side definitely contains $x_{\varepsilon_k}$ and possibly $x_{\varepsilon_{k-1}}$. The value $x_{\varepsilon_{k-1}}$ was determined in the previous step (or does not appear if $k=1$). All other $x_\alpha$ are also defined. Therefore, the value $x_{\varepsilon_k}$ is uniquely determined. Moreover, no other $\mu'(e_\beta), \beta\succeq\gamma$ depends on $x_{\varepsilon_k}$, and its value no longer affects anything else.

Thus, we have computed all coefficients $x_\alpha$ and constructed the required $\mu'$.
\end{proof}

Also note that the conditions of Proposition \ref{nec} do not depend on the specific coordinates of $\lambda$, but only on its support. Therefore, it is sufficient to check these conditions for just one form with a given support; for example, we may assume that all its nonzero coordinates are equal to $1$.

The conditions of Propositions \ref{suf} and \ref{nec} are formulated so that they can be checked algorithmically. Below we present algorithms for checking the necessary and sufficient conditions for a support to belong to the set $S$.

\textbf{Algorithm}

An incomplete enumeration of supports is performed. For each support, the necessary (respectively, sufficient) condition is checked, and the next support to be checked is then generated based on the result. The enumeration relies on the following observation. Let $D$ be a support, and let $\alpha$ be a positive root not in $D$ that is smaller than all roots in $D$. If $D$ does not satisfy the necessary (respectively, sufficient) condition, then $D\cup{\alpha}$ also fails to satisfy the necessary (respectively, sufficient) condition. Indeed, the necessary (respectively, sufficient) condition for a support fails when there exists a root $\gamma$ in the support satisfying certain conditions. These conditions depend only on those $\beta \in D$ such that $\beta \succeq \gamma$. Neither the set $\zeta_{D,\gamma}$ nor the values $\lambda([e_\alpha, e_\beta])$ (where $\alpha$ or $\beta \succeq \gamma$) depend on the roots in the support that are smaller than $\gamma$.

The enumeration proceeds according to the following rules. First, the support of the element ${2\alpha_1+3\alpha_2+4\alpha_3+2\alpha_4}$ is checked against the necessary (respectively, sufficient) condition. Next, suppose we have checked a support $D$ against the necessary (respectively, sufficient) condition. Let $\delta$ be the minimal root in $D$. First, let us consider the case if $\delta\neq\alpha_4$. If $D$ passes the check, we add to $D$ the positive root immediately preceding $\delta$ (with respect to the order $\succ$). If $D$ fails the check, we remove $\delta$ from $D$ and add the positive root immediately preceding $\delta$.
Then we check the new support against the necessary (respectively, sufficient) condition. Second, if $\delta=\alpha_4$, then we remove $\delta$, remove the new minimal root $\beta$ in the support (if it exists) and add the root closest to $\beta$ from below. If $D = \{\alpha_4\}$, then the enumeration ends here. Finally, we include the empty set in the list of supports, since it trivially satisfies the necessary (respectively, sufficient) condition. In this way, we enumerate all supports that could potentially satisfy the necessary (respectively, sufficient) condition, while avoiding checks for supports that certainly do not satisfy these conditions.

Note that if $D$ is the support of a form $\lambda$, then the condition $\lambda([e_\alpha, e_\beta])\neq0$ is equivalent to $\alpha+\beta\in D$. Checking this condition is easy to implement. As mentioned earlier, the necessary (respectively, sufficient) condition must be checked for all roots in the support $D$. It is clear that it suffices to check it only for the root $\gamma$ that was most recently added to the support, because for all other roots the condition has already been verified. The first step is to construct the set $\zeta_{D,\gamma}$. This can be done as follows. We run a loop over rows indexed by roots $\alpha \succ \gamma$. If in a row $\alpha$ there is exactly one non-zero element $\lambda([e_\alpha, e_\delta])$ such that $\delta \notin \zeta_{D,\gamma}$, then we add $\delta$ to $\zeta_{D,\gamma}$. After completing a full pass through the rows, if the size of $\zeta_{D,\gamma}$ has increased, we repeat the loop. The process stops when the cardinality of $\zeta_{D,\gamma}$ stabilizes. This yields the set $\zeta_{D,\gamma}$.

Checking the sufficient condition for the support $D$ is straightforward. We iterate over the row indexed by $\gamma$ and verify the condition from Proposition \ref{suf}.

The necessary condition is checked as follows. We describe the $i$-th step of the verification. We examine the row indexed by $\delta_i$ (recall that $\delta_1 = \gamma$). If we encounter a non-zero element $\lambda([e_{\delta_i}, e_{\varepsilon}])$ such that $\varepsilon \notin \zeta_{D,\gamma}$ and $\varepsilon \neq \varepsilon_{i-1}$ (where $\varepsilon_0$ is undefined), we inspect the column indexed by $\varepsilon$, starting from row $\gamma$ upwards. If there are no other non-zero elements in this segment of the column, the necessary condition is not satisfied. If this is not the case, but exactly one such element $\lambda([e_{\delta_i}, e_{\varepsilon}])$ with $\varepsilon \notin \zeta_{D,\gamma}$ is found in the row (ignoring a possible $\lambda([e_{\delta_i}, e_{\varepsilon_{i-1}}])$), and if, in the corresponding segment of the column, there is exactly one other non-zero element of the form $\lambda([e_{\delta}, e_{\varepsilon}])$, then we set $\varepsilon_i = \varepsilon$ and $\delta_{i+1} = \delta$. We then proceed to check the condition for the next column, indexed by $\delta_{i+1}$.

In practice, our program checks the necessary condition only up to the second step. If the condition has not been violated by that point, the program considers the check passed. This simplification does not significantly affect the final result.

The results of the check are as follows: 878 supports satisfy the sufficient condition, while 911 satisfy the necessary condition. Therefore, 33 supports remain to be checked manually. They are listed below:
\begin{enumerate}\label{tablesupp}
    \item $\alpha_1+2\alpha_2+4\alpha_3+2\alpha_4,~\alpha_1+2\alpha_2+2\alpha_3+2\alpha_4,~\alpha_1+2\alpha_2+2\alpha_3,~\alpha_1+\alpha_2$
    \item $\alpha_1+2\alpha_2+4\alpha_3+2\alpha_4,~\alpha_1+2\alpha_2+2\alpha_3+2\alpha_4,~\alpha_1+2\alpha_2+2\alpha_3,~\alpha_1+\alpha_2,~\alpha_1$
    \item $\alpha_1+2\alpha_2+4\alpha_3+2\alpha_4,~\alpha_1+2\alpha_2+2\alpha_3+2\alpha_4,~\alpha_1+2\alpha_2+2\alpha_3,~\alpha_1,~\alpha_2$
    \item $\alpha_1+2\alpha_2+4\alpha_3+2\alpha_4,~\alpha_1+2\alpha_2+2\alpha_3+2\alpha_4,~\alpha_1+2\alpha_2+2\alpha_3,~\alpha_2$
    \item $\alpha_1+2\alpha_2+2\alpha_3+2\alpha_4,~\alpha_1+2\alpha_2+2\alpha_3,~\alpha_1+\alpha_2+2\alpha_3+\alpha_4,~\alpha_1+\alpha_2+2\alpha_3,~\alpha_2+2\alpha_3,~\alpha_3+\alpha_4$
    \item $\alpha_1+2\alpha_2+2\alpha_3+2\alpha_4,~\alpha_1+\alpha_2+2\alpha_3+\alpha_4,~\alpha_1+\alpha_2+2\alpha_3,~\alpha_2+2\alpha_3,~\alpha_2+\alpha_3$
    \item $\alpha_1+2\alpha_2+2\alpha_3,~\alpha_1+\alpha_2+2\alpha_3+\alpha_4,~\alpha_2+\alpha_3+\alpha_4,~\alpha_3+\alpha_4$
    \item $\alpha_1+2\alpha_2+2\alpha_3,~\alpha_1+\alpha_2+2\alpha_3+\alpha_4,~\alpha_2+\alpha_3+\alpha_4,~\alpha_3+\alpha_4,~\alpha_4$
    \item $\alpha_1+2\alpha_2+2\alpha_3,~\alpha_1+\alpha_2+2\alpha_3+\alpha_4,~\alpha_2+\alpha_3+\alpha_4,~\alpha_4$
    \item $\alpha_1+2\alpha_2+2\alpha_3,~\alpha_1+\alpha_2+\alpha_3+\alpha_4,~\alpha_2+2\alpha_3+\alpha_4,~\alpha_3+\alpha_4$
    \item $\alpha_1+2\alpha_2+2\alpha_3,~\alpha_1+\alpha_2+\alpha_3+\alpha_4,~\alpha_2+2\alpha_3+\alpha_4,~\alpha_3+\alpha_4,~\alpha_3$
    \item $\alpha_1+2\alpha_2+2\alpha_3,~\alpha_1+\alpha_2+\alpha_3+\alpha_4,~\alpha_2+2\alpha_3+\alpha_4,~\alpha_3+\alpha_4,~\alpha_3,~\alpha_4$
    \item $\alpha_1+2\alpha_2+2\alpha_3,~\alpha_1+\alpha_2+\alpha_3+\alpha_4,~\alpha_2+2\alpha_3+\alpha_4,~\alpha_3+\alpha_4,~\alpha_4$
    \item $\alpha_1+2\alpha_2+2\alpha_3,~\alpha_1+\alpha_2+\alpha_3+\alpha_4,~\alpha_2+2\alpha_3+\alpha_4,~\alpha_3,~\alpha_4$
    \item $\alpha_1+2\alpha_2+2\alpha_3,~\alpha_1+\alpha_2+\alpha_3+\alpha_4,~\alpha_2+2\alpha_3+\alpha_4,~\alpha_4$
    \item $\alpha_1+\alpha_2+2\alpha_3+2\alpha_4,~\alpha_1+\alpha_2,~\alpha_2+\alpha_3+\alpha_4,~\alpha_2+\alpha_3,~\alpha_3$
    \item $\alpha_1+\alpha_2+2\alpha_3,~\alpha_1+\alpha_2+\alpha_3+\alpha_4,~\alpha_2+2\alpha_3+\alpha_4,~\alpha_2+\alpha_3+\alpha_4$
    \item $\alpha_1+\alpha_2+2\alpha_3,~\alpha_1+\alpha_2+\alpha_3+\alpha_4,~\alpha_2+2\alpha_3+\alpha_4,~\alpha_2+\alpha_3+\alpha_4,~\alpha_2+\alpha_3$
    \item $\alpha_1+\alpha_2+2\alpha_3,~\alpha_1+\alpha_2+\alpha_3+\alpha_4,~\alpha_2+2\alpha_3+\alpha_4,~\alpha_2+\alpha_3+\alpha_4,~\alpha_2+\alpha_3,~\alpha_2$
    \item $\alpha_1+\alpha_2+2\alpha_3,~\alpha_1+\alpha_2+\alpha_3+\alpha_4,~\alpha_2+2\alpha_3+\alpha_4,~\alpha_2+\alpha_3+\alpha_4,~\alpha_2+\alpha_3,~\alpha_2,~\alpha_4$
    \item $\alpha_1+\alpha_2+2\alpha_3,~\alpha_1+\alpha_2+\alpha_3+\alpha_4,~\alpha_2+2\alpha_3+\alpha_4,~\alpha_2+\alpha_3+\alpha_4,~\alpha_2+\alpha_3,~\alpha_4$
    \item $\alpha_1+\alpha_2+2\alpha_3,~\alpha_1+\alpha_2+\alpha_3+\alpha_4,~\alpha_2+2\alpha_3+\alpha_4,~\alpha_2+\alpha_3+\alpha_4,~\alpha_2$
    \item $\alpha_1+\alpha_2+2\alpha_3,~\alpha_1+\alpha_2+\alpha_3+\alpha_4,~\alpha_2+2\alpha_3+\alpha_4,~\alpha_2+\alpha_3+\alpha_4,~\alpha_2,~\alpha_4$
    \item $\alpha_1+\alpha_2+2\alpha_3,~\alpha_1+\alpha_2+\alpha_3+\alpha_4,~\alpha_2+2\alpha_3+\alpha_4,~\alpha_2+\alpha_3+\alpha_4,~\alpha_4$
    \item $\alpha_1+\alpha_2+2\alpha_3,~\alpha_1+\alpha_2+\alpha_3+\alpha_4,~\alpha_2+2\alpha_3+\alpha_4,~\alpha_2+\alpha_3,~\alpha_2$
    \item $\alpha_1+\alpha_2+2\alpha_3,~\alpha_1+\alpha_2+\alpha_3+\alpha_4,~\alpha_2+2\alpha_3+\alpha_4,~\alpha_2+\alpha_3,~\alpha_2,~\alpha_4$
    \item $\alpha_1+\alpha_2+2\alpha_3,~\alpha_1+\alpha_2+\alpha_3+\alpha_4,~\alpha_2+2\alpha_3+\alpha_4,~\alpha_2+\alpha_3,~\alpha_4$
    \item $\alpha_1+\alpha_2+2\alpha_3,~\alpha_1+\alpha_2+\alpha_3+\alpha_4,~\alpha_2+2\alpha_3+\alpha_4,~\alpha_2,~\alpha_4$
    \item $\alpha_1+\alpha_2+2\alpha_3,~\alpha_1+\alpha_2+\alpha_3+\alpha_4,~\alpha_2+2\alpha_3+\alpha_4,~\alpha_4$
    \item $\alpha_1+\alpha_2+2\alpha_3,~\alpha_1+\alpha_2,~\alpha_2+2\alpha_3+\alpha_4,~\alpha_2+\alpha_3+\alpha_4,~\alpha_2$
    \item $\alpha_1+\alpha_2+2\alpha_3,~\alpha_1+\alpha_2,~\alpha_2+2\alpha_3+\alpha_4,~\alpha_2+\alpha_3+\alpha_4,~\alpha_2,~\alpha_4$
    \item $\alpha_1+\alpha_2+2\alpha_3,~\alpha_1+\alpha_2,~\alpha_2+2\alpha_3+\alpha_4,~\alpha_2+\alpha_3+\alpha_4,~\alpha_4$
    \item $\alpha_1+\alpha_2,~\alpha_2+2\alpha_3,~\alpha_2+\alpha_3+\alpha_4,~\alpha_3+\alpha_4,~\alpha_4$
\end{enumerate}

Linear forms with support 5 do not satisfy the necessary condition, but checking this requires three steps, whereas our algorithm checks only the first two steps (see above). It is clear that linear forms with supports numbered 2, 6, 12, 14, 16, 18, 19, 20 and 21 do not satisfy condition (\ref{rk}); that is, such forms do not belong to $S$. This follows immediately from Table 1, similar to Example~\ref{ex}. In the remaining cases, it is necessary to write out the coordinates of the infinitesimal action explicitly. It turns out that the linear forms with supports numbered 25 and 33 lie in $S$ without any restrictions on their coordinates. Linear forms with supports numbered 30, 31, and 32 lie in $S$ if and only if
\begin{center}
    $\lambda_{\alpha_1+\alpha_2}\lambda^2_{\alpha_2+2\alpha_3+\alpha_4} = \lambda_{\alpha_1+\alpha_2+2\alpha_3}\lambda^2_{\alpha_2+\alpha_3+\alpha_4}$.
\end{center} Thus, we obtain 883 supports whose corresponding forms belong to $S$, three of which require an additional condition on the coordinates. This completes the proof of Proposition~\ref{S}.

\begin{example}
    \label{exam1}
Consider a form $\lambda$ with support 33. We will show that $\lambda\in S$. Let $\mu\in\nt\cdot\lambda$, $\mu=x\cdot\lambda,~x = \sum\nolimits_{\alpha\in\Phi^+}x_\alpha e_\alpha,~\lambda=\sum\nolimits_{\alpha\in\Phi^+}\lambda_\alpha e_\alpha^*,~\mu=\sum\nolimits_{\alpha\in\Phi^+}\mu_\alpha e_\alpha^*$. Let us write down explicitly the coordinates of interest. We denote the structure constants in the relation $C e_{\alpha+\beta}=[e_\alpha,e_\beta]$ by $C_{\alpha,\beta}$. Their values can be found in \cite{KolesnikovPolovinkina23}. Then
\begin{equation*}
\begin{split}
    \mu_{\alpha_1+\alpha_2}&=\mu_{\alpha_2+2\alpha_3}=\mu_{\alpha_2+\alpha_3+\alpha_4}=0;\\
    \mu_{\alpha_1}&=-\lambda([x,e_{\alpha_1}])=-\lambda([x_{\alpha_2}e_{\alpha_2},e_{\alpha_1}])=-C_{\alpha_2,\alpha_1}x_{\alpha_2}\lambda(e_{\alpha_1+\alpha_2})=-x_{\alpha_2}\lambda_{\alpha_1+\alpha_2};\\
    \mu_{\alpha_2+\alpha_3}&=-\lambda([x,e_{\alpha_2+\alpha_3}])=-\lambda([x_{\alpha_4}e_{\alpha_4}+x_{\alpha_3}e_{\alpha_3},e_{\alpha_2+\alpha_3}])=\\
    &=-C_{\alpha_4,\alpha_2+\alpha_3}x_{\alpha_4}\lambda(e_{\alpha_2+\alpha_3+\alpha_4})-C_{\alpha_3,\alpha_2+\alpha_3}x_{\alpha_3}\lambda(e_{\alpha_2+2\alpha_3})= \\
    &=-x_{\alpha_4}\lambda_{\alpha_2+\alpha_3+\alpha_4}+2x_{\alpha_3}\lambda_{\alpha_2+2\alpha_3};\\
    \mu_{\alpha_3+\alpha_4}&=-\lambda([x,e_{\alpha_3+\alpha_4}])=-\lambda([x_{\alpha_2}e_{\alpha_2},e_{\alpha_3+\alpha_4}])=-C_{\alpha_2,\alpha_3+\alpha_4}x_{\alpha_2}\lambda(e_{\alpha_2+\alpha_3+\alpha_4})\\
    &=x_{\alpha_2}\lambda_{\alpha_2+\alpha_3+\alpha_4}=-\dfrac{\lambda_{\alpha_2+\alpha_3+\alpha_4}}{\lambda_{\alpha_1+\alpha_2}}\mu_{\alpha_1};\\
    \mu_{\alpha_3}&=-\lambda([x,e_{\alpha_3}])=-\lambda([x_{\alpha_4}e_{\alpha_4}+x_{\alpha_2+\alpha_3}e_{\alpha_2+\alpha_3},e_{\alpha_3}])\\
    &=-C_{\alpha_4,\alpha_3}x_{\alpha_4}\lambda(e_{\alpha_3+\alpha_4})-C_{\alpha_2+\alpha_3,\alpha_3}x_{\alpha_2+\alpha_3}\lambda(e_{\alpha_2+2\alpha_3})=\\
    &-x_{\alpha_4}\lambda_{\alpha_3+\alpha_4}-2x_{\alpha_2+\alpha_3}\lambda_{\alpha_2+2\alpha_3};\\
    \mu_{\alpha_4}&=-\lambda([x,e_{\alpha_4}])=-\lambda([x_{\alpha_3}e_{\alpha_3}+x_{\alpha_2+\alpha_3}e_{\alpha_2+\alpha_3},e_{\alpha_4}])\\
    &=-C_{\alpha_3,\alpha_4}x_{\alpha_3}\lambda(e_{\alpha_3+\alpha_4})-C_{\alpha_2+\alpha_3,\alpha_4}x_{\alpha_2+\alpha_3}\lambda(e_{\alpha_2+\alpha_3+\alpha_4})= \\
    &=x_{\alpha_3}\lambda_{\alpha_3+\alpha_4}+x_{\alpha_2+\alpha_3}\lambda_{\alpha_2+\alpha_3+\alpha_4}=\dfrac{\lambda_{\alpha_3+\alpha_4}}{2\lambda_{\alpha_2+2\alpha_3}}\mu_{\alpha_2+\alpha_3}-\dfrac{\lambda_{\alpha_2+\alpha_3+\alpha_4}}{2\lambda_{\alpha_2+2\alpha_3}}\mu_{\alpha_3}.
\end{split}
\end{equation*}

Let us also consider support number 25. For all roots except $\alpha_2$, condition~(\ref{rk}) obviously holds. The coordinate $\mu_{\alpha_2}$ equals $-x_{\alpha_3}\lambda_{\alpha_2+\alpha_3}$. Therefore, if $x_{\alpha_3}$ can be expressed as a linear combination of $\{\mu_{\beta}\}_{\beta\succ\alpha_2}$, then $\alpha_2$ satisfies condition (\ref{rk}). This is indeed the case:
\begin{equation*}
\begin{split}
    \mu_{\alpha_1+\alpha_2+\alpha_3}&=-x_{\alpha_4}\lambda_{\alpha_1+\alpha_2+\alpha_3+\alpha_4}+2x_{\alpha_3}\lambda_{\alpha_1+\alpha_2+2\alpha_3};\\
    \mu_{\alpha_2+2\alpha_3}&=-x_{\alpha_4}\lambda_{\alpha_2+2\alpha_3+\alpha_4}+x_{\alpha_1}\lambda_{\alpha_1+\alpha_2+2\alpha_3};\\
    \mu_{\alpha_2+\alpha_3+\alpha_4}&=x_{\alpha_3}\lambda_{\alpha_2+2\alpha_3+\alpha_4}+x_{\alpha_1}\lambda_{\alpha_1+\alpha_2+\alpha_3+\alpha_4};\\
    x_{\alpha_3}&=\dfrac{\mu_{\alpha_2+\alpha_3+\alpha_4}\lambda_{\alpha_1+\alpha_2+2\alpha_3}-\mu_{\alpha_2+2\alpha_3}\lambda_{\alpha_1+\alpha_2+\alpha_3+\alpha_4}+\mu_{\alpha_1+\alpha_2+\alpha_3}\lambda_{\alpha_2+2\alpha_3+\alpha_4}}{3\lambda_{\alpha_2+2\alpha_3+\alpha_4}\lambda_{\alpha_1+\alpha_2+2\alpha_3}}.
\end{split}
\end{equation*}

Finally, consider forms with supports 30, 31, and 32. Suppose the support contains the roots $\alpha_1+\alpha_2+2\alpha_3,~\alpha_1+\alpha_2,~\alpha_2+2\alpha_3+\alpha_4,~\alpha_2+\alpha_3+\alpha_4$. Let us check whether condition~(\ref{rk}) holds for the root $\alpha_2$. Note that
\begin{equation*}
\begin{split}
    \mu_{\alpha_2+2\alpha_3}&=-x_{\alpha_4}\lambda_{\alpha_2+2\alpha_3+\alpha_4}+x_{\alpha_1}\lambda_{\alpha_1+\alpha_2+2\alpha_3};\\
    \mu_{\alpha_2+\alpha_3}&=-x_{\alpha_4}\lambda_{\alpha_2+\alpha_3+\alpha_4}+x_{\alpha_3+\alpha_4}\lambda_{\alpha_2+2\alpha_3+\alpha_4};\\
    \mu_{\alpha_2}&=-x_{\alpha_3+\alpha_4}\lambda_{\alpha_2+\alpha_3+\alpha_4}+x_{\alpha_1}\lambda_{\alpha_1+\alpha_2}.
\end{split}
\end{equation*}
We see that if $\beta\succ\alpha_2$ and the variable $x_{\alpha_3+\alpha_4}$ appears in $\mu_\beta$, then $\beta=\alpha_2+\alpha_3$. Similarly, if $\beta\succ\alpha_2$ and $x_{\alpha_1}$ appears in $\mu_\beta$, then $\beta=\alpha_2+2\alpha_3$. Hence, the coordinate $\mu_{\alpha_2}$ must be a linear combination of $\mu_{\alpha_2+2\alpha_3}$ and $\mu_{\alpha_2+\alpha_3}$. It is easy to verify that this holds if and only if
\begin{center}
$\lambda_{\alpha_1+\alpha_2}\lambda^2_{\alpha_2+2\alpha_3+\alpha_4} = \lambda_{\alpha_1+\alpha_2+2\alpha_3}\lambda^2_{\alpha_2+\alpha_3+\alpha_4}$.
\end{center}

Now, let us check condition (\ref{rk}) for the root $\alpha_4$. Note that
\begin{equation*}
\begin{split}
    \mu_{\alpha_1}&=-x_{\alpha_2}\lambda_{\alpha_1+\alpha_2}-x_{\alpha_2+2\alpha_3}\lambda_{\alpha_1+\alpha_2+2\alpha_3};\\
    \mu_{\alpha_3+\alpha_4}&=x_{\alpha_2}\lambda_{\alpha_2+\alpha_3+\alpha_4}-x_{\alpha_2+\alpha_3}\lambda_{\alpha_2+2\alpha_3+\alpha_4};\\
    \mu_{\alpha_4}&=x_{\alpha_2+\alpha_3}\lambda_{\alpha_2+\alpha_3+\alpha_4}+x_{\alpha_2+2\alpha_3}\lambda_{\alpha_2+2\alpha_3+\alpha_4}.
\end{split}
\end{equation*}
By an argument analogous to the one above, we find that the coordinate $\mu_{\alpha_4}$ must be a linear combination of $\mu_{\alpha_1}$ and $\mu_{\alpha_3+\alpha_4}$. It turns out that this condition is also equivalent to
\begin{center}
$\lambda_{\alpha_1+\alpha_2}\lambda^2_{\alpha_2+2\alpha_3+\alpha_4} = \lambda_{\alpha_1+\alpha_2+2\alpha_3}\lambda^2_{\alpha_2+\alpha_3+\alpha_4}$.
\end{center}
\end{example}

\section{Classification of orbits}

In this section, we will prove Theorem \ref{list}, thereby completing the classification of coadjoint orbits for type $F_4$. For the next stage, we will need the concept of a singular root.
\begin{definition}
Let $\alpha,\gamma\in\Phi^+$. The root $\alpha$ is called singular to $\gamma$ if $\gamma-\alpha\in\Phi^+$. The set of all roots singular to $\gamma$ is denoted by $S_\gamma$.
\end{definition}

It is obvious that the singular roots in $S_\gamma$ are divided into pairs of the form $\alpha+\beta=\gamma$.

From now on, unless stated otherwise, everything will take place over the algebraic closure $\overline{\Fp}$. We introduce two $\overline{\Fp}[x_\alpha,\alpha\in\Phi^+]$-modules:
\begin{equation*}
\begin{split}
M &= \overline{\Fp}[x_\alpha,\alpha\in\Phi^+]\otimes_{\Fp}\nt,\\
M^* &= \overline{\Fp}[x_\alpha,\alpha\in\Phi^+]\otimes_{\Fp}\nt^*.
\end{split}
\end{equation*}
Here $\overline{\Fp}[x_\alpha,\alpha\in\Phi^+]$ denotes the ring of polynomials over $\overline\Fp$ in the independent variables $x_\alpha, \alpha\in\Phi^+$. We define a pairing $M^*\times M\to\overline{\Fp}[x_\alpha,\alpha\in\Phi^+]$ by
\begin{equation*}
\langle f\otimes\lambda,g\otimes x\rangle\mapsto \lambda(x)fg.
\end{equation*}
In what follows, we will write $\lambda(x)$ instead of $\langle\lambda,x\rangle$ for $\lambda\in M^*, x\in M$.
Note that $\nt$ is a $\Fp$-subspace of $M$, and $\nt^*$ is a $\Fp$-subspace of $M^*$. Therefore, two other pairings, $\nt^*\times M\to\overline{\Fp}[x_\alpha,\alpha\in\Phi^+]$ and $M^*\times \nt\to\overline{\Fp}[x_\alpha,\alpha\in\Phi^+]$, are automatically induced. We introduce a Lie algebra structure on $M$ in the obvious way: for $f,g\in\overline{\Fp}[x_\alpha,\alpha\in\Phi^+]$ and $x,y\in\nt$, set
\begin{equation*}
[f\otimes x, g\otimes y] = fg\otimes[x,y].
\end{equation*}

Since $\nt$ can be identified with upper-triangular matrices with entries in $\Fp$, we can identify $M$ with upper-triangular matrices whose entries lie in $\overline{\Fp}[x_\alpha, \alpha\in\Phi^+]$. Consequently, we can define $\exp(x)$ for $x\in M$. Since the characteristic of $\Fp$ is sufficiently large, $\exp(x)$ is well-defined for all $x\in M$. Let $N_M$ be the group consisting of the matrices $\{\exp(x), x\in M\}$. The actions $N_M\curvearrowright M$ and $N_M\curvearrowright M^*$ are defined naturally. An explicit description of the action $N_M\curvearrowright M^*$ is as follows: for $x\in M$ and $\lambda \in M^*$,
\begin{equation*}
\exp(x)\cdot\lambda = \sum_{\alpha\in\Phi^+}\lambda\left(\sum_{i=0}^{\infty}\dfrac{(-1)^i}{i!}\underbrace{[x,[x,\dots,[x,e_\alpha]\dots]]}_{\text{$i$ commutators}}\right)\otimes e_\alpha^*.
\end{equation*}

Denote by $\chi$ the element $\chi = \sum_{\alpha\in\Phi^+}x_\alpha \otimes e_\alpha\in M$. For a form $\lambda\in\nt^*$, we will consider the element $\exp(\chi)\cdot\lambda\in M^*$. Let us also set $\overline{\nt}=\overline{\Fp}\otimes_{\Fp}\nt$, $\overline{N}=\overline{\Fp}\otimes_{\Fp}N$. Note that if $\wt{x} = \sum_{\alpha\in\Phi^+}\wt{x}_\alpha e_\alpha\in \overline\nt$ with $\wt{x}_\alpha\in\overline\Fp$, then $\exp(\wt{x})\cdot\lambda = (\exp(\chi)\cdot\lambda)(\wt{x}_\alpha,\alpha\in\Phi^+)$.

Now we formulate one of the most important theorems of this work. Let $y_\beta,\beta\in\Phi^+$ be independent variables.

\begin{theorem}
    \label{fun}
Let $\lambda\in S$\textup, $D = \supp(\lambda)$\textup, and let $\gamma\in\Phi^+$ satisfy condition \textup{(\ref{rk})}. Then there exists a rational function $F_{D,\gamma}\in\overline{\Fp}(y_\beta,\beta\succ\gamma)$ such that
\begin{equation}\label{FD}
\mu(e_\gamma) = \lambda(e_\gamma) + F_{D,\gamma}(\mu(e_\beta), \beta\succ\gamma),
\end{equation}
where $\mu=\exp(\chi)\cdot\lambda\in M^*$.
\end{theorem}

\begin{example}
    \label{exN}
Consider the same support as in Example \ref{ex}. Let
\begin{center}
    $D=\supp(\lambda)=\{\alpha_4,\alpha_3+\alpha_4,\alpha_2,\alpha_2+2\alpha_3\}$.
\end{center} We will show that for this $\lambda$ and for all roots $\gamma$ satisfying condition (\ref{rk}), formula (\ref{FD}) holds. Let $\mu=\exp(\chi)\cdot\lambda,~\lambda=\sum_{\alpha\in\Phi^+}\lambda_\alpha e_\alpha^*,~\mu=\sum_{\alpha\in\Phi^+}\mu_\alpha\otimes e_\alpha^*$. Let us write down explicitly the coordinates of interest. As in Example \ref{exam1}, we denote the structural constants in the relation $C e_{\alpha+\beta}=[e_\alpha,e_\beta]$ by $C_{\alpha,\beta}$. Their values can be found in \cite{KolesnikovPolovinkina23}. We have
\begin{equation*}
\begin{split}
    \mu_{\alpha_2+2\alpha_3}&=\lambda(e_{\alpha_2+2\alpha_3})=\lambda_{\alpha_2+2\alpha_3};\\
    \mu_{\alpha_2+\alpha_3+\alpha_4}&=0;\\
    \mu_{\alpha_2+\alpha_3}&=-\lambda([x,e_{\alpha_2+\alpha_3}])=-
    \lambda([x_{\alpha_3}e_{\alpha_3},e_{\alpha_2+\alpha_3}])=-C_{\alpha_3,\alpha_2+\alpha_3}x_{\alpha_3}
    \lambda(e_{\alpha_2+2\alpha_3})=\\
    &=2x_{\alpha_3}\lambda_{\alpha_2+2\alpha_3};\\
    \mu_{\alpha_2}&=\lambda(e_{\alpha_2})+\dfrac{1}{2}\lambda([x,[x,e_{\alpha_2}]])=\lambda_{\alpha_2}+\dfrac{1}{2}\lambda([x_{\alpha_3}e_{\alpha_3},[x_{\alpha_3}e_{\alpha_3},e_{\alpha_2}]])=\\
    &=\lambda_{\alpha_2}+\dfrac{1}{2}C_{\alpha_3,\alpha_2+\alpha_3}C_{\alpha_3,\alpha_2}x_{\alpha_3}^2\lambda(e_{\alpha_2+2\alpha_3})=\lambda_{\alpha_2}-x_{\alpha_3}^2\lambda_{\alpha_2+2\alpha_3}=\\
    &=\lambda_{\alpha_2}-\dfrac{\mu_{\alpha_2+2\alpha_3}^2}{4\mu_{\alpha_2+2\alpha_3}};\\
    \mu_{\alpha_3+\alpha_4}&=\lambda(e_{\alpha_3+\alpha_4})=\lambda_{\alpha_3+\alpha_4};\\
    \mu_{\alpha_4}&=\lambda(e_{\alpha_4})-\lambda([x,e_{\alpha_3+\alpha_4}])=\lambda_{\alpha_4}-\lambda([x_{\alpha_3}e_{\alpha_3},e_{\alpha_3+\alpha_4}])=\\
    &=\lambda_{\alpha_4}-C_{\alpha_3,\alpha_4}x_{\alpha_3}\lambda(e_{\alpha_3+\alpha_4})=\lambda_{\alpha_4}+x_{\alpha_3}\lambda_{\alpha_3+\alpha_4}=\lambda_{\alpha_4}+\dfrac{\mu_{\alpha_2+\alpha_3}\mu_{\alpha_3+\alpha_4}}{2\mu_{\alpha_2+2\alpha_3}}.
\end{split}
\end{equation*}

In fact, the validity of formula (\ref{FD}) in this example could have been understood without writing out the coordinates explicitly. It is enough to examine Table 1. Since $\mu(e_{\alpha_2+\alpha_3})$ involves only $x_{\alpha_3}$ and $\lambda(e_{\alpha_2+2\alpha_3})$, and since $\lambda(e_{\alpha_2+2\alpha_3}) = \mu(e_{\alpha_2+2\alpha_3})$, the variable $x_{\alpha_3}$ can be expressed as a rational function of $\mu(e_{\alpha_2+\alpha_3})$ and $\mu(e_{\alpha_2+2\alpha_3})$. Moreover, since $\mu(e_{\alpha_2})-\lambda(e_{\alpha_2})$ depends only on $x_{\alpha_3}$ and $\lambda(e_{\alpha_2+2\alpha_3})$, it can be expressed as a rational function of $\mu(e_{\alpha_2+\alpha_3})$ and $\mu(e_{\alpha_2+2\alpha_3})$. Furthermore, $\mu(e_{\alpha_3+\alpha_4})=\lambda(e_{\alpha_3+\alpha_4})$. As $\mu(e_{\alpha_4})-\lambda(e_{\alpha_4})$ involves only $x_{\alpha_3}$ and $\lambda(e_{\alpha_3+\alpha_4})$, it can be expressed as a rational function of $\mu(e_{\alpha_2+\alpha_3})$ and $\mu(e_{\alpha_3+\alpha_4})$.

\end{example}

Now let us proceed to the proof of Theorem \ref{fun}. First, we prove it in the special case of forms whose support consists of a single root. Any such form obviously belongs to $S$.

\begin{lemma}
    \label{uno}
Theorem \textup{\ref{fun}} holds for any $\lambda\in\nt^*$ with $|\supp(\lambda)|=1$.
\end{lemma}
\begin{proof}
    Denote $D=\supp(\lambda)=\{\delta\}$. Note that in this case, condition (\ref{rk}) depends only on $D$, not on the specific $\lambda$, and it holds for those $\gamma$ that are not singular to $\delta$. Let $\mu = \exp(\chi)\cdot\lambda$. We need to prove that for any $\gamma \notin S_\delta$, the coordinate $\mu(e_\gamma)$ is a function of the coordinates ${\mu(e_\beta), \beta \succ \gamma}$. Observe that if $\gamma \succ \delta$, then $\mu(e_\gamma) = 0$, and if $\gamma = \delta$, then $\mu(e_\delta) = \lambda(e_\delta)$. For these two cases, the condition is obviously satisfied.  It remains to prove it for $\gamma \prec \delta$.

For any $\beta \in S_\delta$, the following equality holds:
\begin{equation*}
\begin{split}
\mu(e_\beta) =& -\lambda([x,e_\beta]) + \dfrac{1}{2!}\lambda([x, [x, e_\beta]]) - \dfrac{1}{3!}\lambda([x,[x,[x,e_\beta]]]) + \ldots\\ =&-\lambda([x_{\delta-\beta}e_{\delta-\beta},e_\beta]) + F(x) = -\lambda([e_{\delta-\beta},e_\beta]) x_{\delta-\beta} + F(x) = C\lambda(e_\delta) x_{\delta-\beta} + F(x),
\end{split}
\end{equation*}
where $C \neq 0$ and $F(x)$ is a polynomial in variables $x_\alpha$ with $\alpha < \delta - \beta$. Consequently,
\begin{equation}\label{xb}
x_{\delta-\beta} = \dfrac{\mu(e_\beta) - F(x)}{C\lambda(e_\delta)}
\end{equation}
Let $\beta_1, \beta_2, \ldots, \beta_{2k}$ be the singular roots to $\delta$, listed in decreasing order with respect to lexicographical order. They come in pairs such that $\beta_i + \beta_{2k-i+1} = \delta$. Now take $\gamma \notin S_\delta$ with $\delta \succ \gamma$ and consider $\mu(e_\gamma)$. Suppose, for contradiction, that there is no rational function $H$ over $\overline{\Fp}$ in the independent variables $t_\beta, \beta \succ \gamma$ such that $\mu(e_\gamma) = H(\mu(e_\beta), \beta \succ \gamma)$.

The coordinate $\mu(e_\gamma)$ is a polynomial in the $x_\beta$ with coefficients from $\overline{\Fp}$. It does not contain the variable $x_{\beta_{2k-i+1}}$ if $\gamma \succ \beta_i$. Let $\beta_{2k-j+1}$ be the largest root in $S_\delta$ such that $x_{\beta_{2k-j+1}}$ appears in $\mu(e_\gamma)$. Using formula (\ref{xb}), we express $x_{\beta_{2k-j+1}}$ in terms of $\mu(e_{\beta_j})$ and ${x_\alpha, \alpha < \beta_{2k-j+1}}$. Thus, $\mu(e_\gamma)$ becomes a polynomial that no longer involves $x_{\beta_{2k-j+1}}$. We perform a similar elimination successively for $\beta_{2k-j+2}, \beta_{2k-j+3}$, and so on, until $\beta_{2k}$. Note that at each step, a variable $x_{\beta_i}$ that was initially absent or already removed cannot reappear, because we proceed in decreasing order of the singular roots. As a result, $\mu(e_\gamma)$ is represented as a polynomial that does not contain $x_\beta, \beta \in S_\delta$, i.e., \[\mu(e_\gamma) = G(\mu(e_{\beta_i}), i=1,\ldots,j;~x_\alpha, \alpha \notin S_\delta)\] for some polynomial $G \in \overline{\Fp}[t_{\beta_i}, i=1,\ldots,j;~x_\alpha, \alpha \notin S_\delta]$. Observe that $\beta_i \succ \gamma$ for all $i=1,\ldots,j$.

By our assumption, $\mu(e_\gamma)$ cannot be expressed solely in terms of the subsequent coordinates. This implies that $G$ is not identically constant and actually depends on at least one variable $x_\alpha$. Then there exists a non-empty Zariski open subset $U \subseteq \overline{\Fp}^j$ such that for any $c = (c_1, c_2, \ldots, c_j) \in U$, the specialized polynomial $G(c_i, i = 1,\ldots,j;~x_\alpha, \alpha \notin S_\delta)$ is non-constant. Fix such a tuple $c_1, c_2, \ldots, c_j \in \overline{\Fp}$. Choose an index $\varepsilon \notin S_\delta$ and constants $c_\alpha \in \overline{\Fp}, \alpha \notin S_\delta, \alpha \neq \varepsilon$, so that the one-variable polynomial \[g(x_\varepsilon) = G(c_i, i = 1,\ldots,j;~x_\varepsilon, c_\alpha, \alpha \notin S_\delta, \alpha \neq \varepsilon)\] is non-constant. The field is algebraically closed, for any chosen $c_0 \in \overline{\Fp}$ there exists a $c_\varepsilon \in \overline{\Fp}$ such that $g(c_\varepsilon) = c_0$. Fix $c_0$ and the corresponding $c_\varepsilon$. Now, choose arbitrary values $c_{j+1}, \ldots, c_{2k} \in \overline{\Fp}$. Substituting the values $c_\varepsilon, c_\alpha, \alpha \notin S_\delta$ into formula~(\ref{xb}), we can determine values for $x_{\beta_{2k}}, x_{\beta_{2k-1}}, \ldots, x_{\beta_1}$ sequentially, in such a way that $\mu(e_{\beta_1}) = c_1$, $\mu(e_{\beta_2}) = c_2$, $\ldots$ and $\mu(e_{\beta_{2k}}) = c_{2k}$. For this assignment, we obtain
\[
    \mu(e_\gamma) = G(c_i, i = 1, \ldots, j;~c_\varepsilon, c_\alpha, \alpha \notin S_\delta) = c_0.
\]

Therefore, there exists a non-empty Zariski open subset $V\subseteq \overline{\Fp}^{2k+1}$ such that for any $(c_0, c_1, \ldots, c_{2k}) \in V$, we can find an element $\overline\mu \in \overline{N}\cdot\lambda$ such that $\overline\mu(e_\gamma) = c_1$, $\overline\mu(e_{\beta_1}) = c_2$, $\ldots$, $\overline\mu(e_{\beta_{2k}}) c_{2k}$. The orbit $\overline{N}\cdot\lambda$ is an algebraic variety. Our construction yields a surjection from $\overline{N}\cdot\lambda$ to $\overline{\Fp}^{2k+1}$, which implies $\dim \overline{N}\cdot\lambda \ge 2k+1$. However, \[\dim \overline{N}\cdot\lambda = \dim \overline{\nt}\cdot\lambda = \dim \nt\cdot\lambda = |S_\delta| = 2k.\] This contradiction completes the proof.
\end{proof}

We introduce further notation. Let $\alpha,\beta\in\Phi^+$. Denote $T_{\alpha,\beta}(x):=(\exp(\chi).e_\beta^*)(e_\alpha)$ in $\overline{\Fp}[x_\gamma,\gamma\in\Phi^+]$. It is clear that if $\lambda\in\nt^*, \alpha\in\Phi^+$, then
\begin{align*}
    (\exp(\chi)\cdot\lambda)(e_\alpha) 
    &= \left(\exp(\chi)\cdot\left(\sum_{\beta\in \supp(\lambda)}\lambda(e_\beta)e_\beta^*\right)\right)(e_\alpha) \\
    &= \sum_{\beta\in \supp(\lambda)}\lambda(e_\beta)(\exp(\chi).e_\beta^*)(e_\alpha)\\
    &= \sum_{\beta\in \supp(\lambda)}\lambda(e_\beta)T_{\alpha,\beta}(x).
\end{align*}

Now consider the following construction (cf. the construction of the set $\zeta_{D,\gamma}$ used in the proof of the sufficient condition in Proposition \ref{suf}). Let $\lambda\in S,~D = \supp(\lambda)$, and let the root $\gamma\in\Phi^+$ satisfy condition (\ref{rk}). Starting from $D$ and $\gamma$, we construct a set $\zeta'_{D,\gamma}\subset\Phi^+\times\Phi^+$: by definition, it consists of those pairs $(\alpha,\beta)$ for which the polynomial $T_{\alpha,\beta}(x)$ can be expressed as a rational function of $\{(\exp(\chi)\cdot\lambda)(e_\delta)\}_{\delta\succ\gamma}$.

\begin{example}

Table 1 helps visualize  the structure of the polynomials $T_{\alpha,\beta}(x)$. One simply examines the cell in row $\alpha$ and column $\beta$. For each pair $\gamma,\delta\in\Phi^+$ such that $\gamma+\delta\in\Phi^+$, denote the corresponding non-zero structure constant by $C_{\gamma,\delta}$, so that $[e_\gamma,e_\delta]=C_{\gamma,\delta}e_{\gamma+\delta}$. Each monomial in $T_{\alpha,\beta}(x)$ has the form $$(-1)^{n+1}\dfrac{1}{n!}C_{\gamma_1,\alpha}C_{\gamma_2,\alpha+\gamma_1}\ldots C_{\gamma_n,\alpha+\gamma_1+\ldots+\gamma_{n-1}}x_{\gamma_1}x_{\gamma_2}\ldots x_{\gamma_n},$$
where $\gamma_1,\ldots,\gamma_n\in\Phi^+,~\alpha+\gamma_1+\ldots+\gamma_i\in\Phi^+$ for all $i$ from $1$ to $n$, and $\alpha+\gamma_1+\ldots+\gamma_n=\beta$. Hence, it suffices to consider all possible decompositions of $\beta-\alpha$ into a sum of positive roots. For example,
\begin{center}
    $T_{\alpha_3,\alpha_2+2\alpha_3+\alpha_4}(x)=-x_{\alpha_2+\alpha_3+\alpha_4}+\dfrac{1}{2}x_{\alpha_2}x_{\alpha_3+\alpha_4}+\dfrac{3}{2}x_{\alpha_4}x_{\alpha_2+\alpha_3}-\dfrac{2}{3}x_{\alpha_2}x_{\alpha_3}x_{\alpha_4}$.
\end{center}

Note that if $\alpha\nleq\beta$, then $T_{\alpha,\beta}(x)=0$, while $T_{\alpha,\alpha}(x)=1$. In all other cases, $T_{\alpha,\beta}(x)$ has positive degree. We call a pair $(\alpha,\beta)$ nontrivial if $\alpha<\beta$ (i.e., $T_{\alpha,\beta}(x)\neq const$). Obviously, each trivial pair belongs to all $\zeta'_{D,\gamma}$.

Let us refer to Example \ref{exN}. There $D=\supp(\lambda)=\{\alpha_4,\alpha_3+\alpha_4,\alpha_2,\alpha_2+2\alpha_3\}$. The sets $\zeta'_{D,\alpha_2+2\alpha_3}$ and $\zeta'_{D,\alpha_2+\alpha_3+\alpha_4}$ contain no nontrivial pairs. Indeed, there are no $\beta$ in $D$ such that $\beta>\alpha_2+2\alpha_3$ or $\beta>\alpha_2+\alpha_3+\alpha_4$. Since $$(\exp(\chi)\cdot\lambda)(e_{\alpha_2+\alpha_3})=2x_{\alpha_3}\lambda(e_{\alpha_2+2\alpha_3})=2x_{\alpha_3}(\exp(\chi)\cdot\lambda)(e_{\alpha_2+2\alpha_3})$$ and $T_{\alpha_2,\alpha_2+2\alpha_3}(x)=-x_{\alpha_3}^2$, the pair $(\alpha_2,\alpha_2+2\alpha_3)$ belongs to $\zeta'_{D,\alpha_2}$. This pair also lies in $\zeta'_{D,\alpha_3+\alpha_4}$ and $\zeta'_{D,\alpha_4}$. Moreover, $\zeta'_{D,\alpha_4}$ contains one additional nontrivial pair, namely, $(\alpha_4,\alpha_3+\alpha_4)$, because $T_{\alpha_4,\alpha_3+\alpha_4}(x)=x_{\alpha_3}$. The remaining pairs of the form $(\gamma,\alpha)$ with $\alpha\in D$ and $\gamma$ satisfies condition (\ref{rk}) are trivial and therefore belong to all $\zeta'_{D,\gamma}$. The lemma below explains why we focus precisely on these pairs.
\end{example}

\begin{lemma}
    \label{zeta'}
Let $\lambda\in S, D=\supp(\lambda)$\textup, and let $\gamma\in\Phi^+$ satisfy condition \textup{(\ref{rk})}\textup. Suppose that for every $\alpha\in D$\textup, the pair $(\gamma,\alpha)$ belongs to $\zeta'_{D,\gamma}$\textup, and formula \textup{(\ref{FD})} holds for all $\gamma'\succ\gamma$ that satisfy condition \textup{(\ref{rk})}. Then formula \textup{(\ref{FD})} also holds for $\lambda$\textup, $D$ and $\gamma$.
\end{lemma}
\begin{proof}
Consider $$(\exp(\chi)\cdot\lambda)(e_\gamma)=\sum_{\alpha\in D}\lambda(e_\alpha)T_{\gamma,\alpha}(x)=\lambda(e_\gamma)+\sum_{\alpha\in D,\alpha\succ\gamma}\lambda(e_\alpha)T_{\gamma,\alpha}(x).$$ Let $\alpha\in D$ with $\alpha\succ\gamma$. Since $(\gamma,\alpha)\in\zeta'_{D,\gamma}$, the polynomial $T_{\gamma,\alpha}(x)$ can be expressed as a rational function of $\{(\exp(\chi)\cdot\lambda)(e_\beta)\}_{\beta\succ\gamma}$. By the definition of $S$, the root $\alpha$ satisfies condition (\ref{rk}). Therefore, by the inductive hypothesis, $\lambda(e\alpha)$ can also be expressed as a rational function of $\{(\exp(\chi)\cdot\lambda)(e_\beta)\}_{\beta\succ\gamma}$, because $$\lambda(e_\alpha)=(\exp(\chi)\cdot\lambda)(e_\alpha)-F_{D,\alpha}((\exp(\chi)\cdot\lambda)(e_\beta),\beta\succ\alpha).$$ Hence, the difference $(\exp(\chi)\cdot\lambda)(e_\gamma)-\lambda(e_\gamma)$ is a rational function of $\{(\exp(\chi)\cdot\lambda)(e_\beta)\}_{\beta\succ\gamma}$, which means condition (\ref{FD}) is satisfied.
\end{proof}

We are now ready to prove Theorem \ref{fun}.

\textsc{Proof of Theorem \ref{fun}.} If $D=\{\alpha\}$, then Theorem~\ref{fun} holds by Lemma~\ref{uno}. Indeed, in this case $(\exp(\chi)\cdot\lambda)(e_\gamma) = \lambda(e_\alpha)\cdot T_{\gamma,\alpha}(x)$. Moreover, $T_{\gamma,\alpha}(x)=0$ for $\gamma\succ\alpha$, while $T_{\gamma,\alpha}(x)=1$ for $\gamma=\alpha$. Finally, for $\gamma\prec\alpha$ satisfying condition (\ref{rk}), the polynomial $T_{\gamma,\alpha}(x)$ can be expressed as a rational function of $\{(\exp(\chi)\cdot\lambda)(e_\delta)\}_{\delta\succ\gamma}$.

Now we will prove formula (\ref{FD}) step by step for all $\lambda\in S$ by considering the supports from the list one after another. (Note that most supports can be treated similarly to Example~\ref{exN}; see below.) Recall that
\begin{center}
$\lambda=\sum_{\beta\in\Phi^+}\lambda_{\beta}e_{\beta}^*\in\nt^*,~\mu=\sum_{\beta\in\Phi^+}\mu_{\beta}\otimes e_{\beta}^*=\exp(\chi)\cdot\lambda\in M^*$.
\end{center}
\begin{itemize}
    \item Suppose the largest root (with respect to the lexicographic order) in the support is smaller than $\alpha_2+\alpha_3+\alpha_4$. For all such supports and for every $\gamma\in\Phi^+$ satisfying condition (\ref{rk}), Lemma \ref{zeta'} applies, so the proof is complete. (Indeed, for all relevant $\gamma$, the polynomials $T_{\gamma,\alpha}(x)$ are, up to a non-zero scalar factor, equal to $x_{\alpha_2}$, $x_{\alpha_3}$, $x_{\alpha_4}$, $0$, or $1$.)
    \item Suppose the largest root in the support equals $\alpha_2+\alpha_3+\alpha_4$. There are only two such supports in the list: $\{\alpha_2+\alpha_3+\alpha_4\}$ and $\{\alpha_2+\alpha_3+\alpha_4,\alpha_3\}$. Since Lemma \ref{uno} covers the case of one-element supports, the pair $(\alpha_3,\alpha_2+\alpha_3+\alpha_4)$ belongs to $\zeta'_{D,\alpha_3}$ regardless of whether $\alpha_3$ lies in $D$ (the presence of simple roots in the support does not affect $\zeta'_{D,\gamma}$). The conditions of Lemma \ref{zeta'} are also satisfied, so nothing remains to prove in this case.
    \item Next, suppose the largest root in the support equals $\alpha_2+2\alpha_3$. Here the first nontrivial case appears, i.e., Theorem \ref{fun} cannot be proved for all supports of this type by a straightforward application of Lemma~\ref{zeta'}. Such a support may also contain $\alpha_2+\alpha_3+\alpha_4$ and $\alpha_4$. The complication arises because the table entries at positions $(\alpha_3,\alpha_2+2\alpha_3)$ and $(\alpha_4,\alpha_2+\alpha_3+\alpha_4)$ both involve the same root $\alpha_2+\alpha_3$. However, $T_{\alpha_3,\alpha_2+2\alpha_3}(x)\neq T_{\alpha_4,\alpha_2+\alpha_3+\alpha_4}(x)$. Consequently, a situation may occur where $\alpha_4$ satisfies condition (\ref{rk}) but Lemma \ref{zeta'} cannot be applied directly. Therefore, in this case we must compute the coordinates of $\mu$ explicitly:
    \begin{equation*}
    \begin{split}
        \mu_{\alpha_2+2\alpha_3}=&~\lambda(e_{\alpha_2+2\alpha_3})=\lambda_{\alpha_2+2\alpha_3},\\
        \mu_{\alpha_2+\alpha_3+\alpha_4}=&~\lambda(e_{\alpha_2+\alpha_3+\alpha_4})=\lambda_{\alpha_2+\alpha_3+\alpha_4},\\
        \mu_{\alpha_2+\alpha_3}=&-\lambda([x,e_{\alpha_2+\alpha_3}])=\\
        =&-\lambda([x_{\alpha_4}e_{\alpha_4}+x_{\alpha_3}e_{\alpha_3},e_{\alpha_2+\alpha_3}])=-C_{\alpha_4,\alpha_2+\alpha_3}x_{\alpha_4}\lambda(e_{\alpha_2+\alpha_3+\alpha_4})-\\
        &-C_{\alpha_3,\alpha_2+\alpha_3}x_{\alpha_3}\lambda(e_{\alpha_2+2\alpha_3})=x_{\alpha_4}\lambda_{\alpha_2+\alpha_3+\alpha_4}-2x_{\alpha_3}\lambda_{\alpha_2+2\alpha_3},\\
        \mu_{\alpha_3+\alpha_4}=&-\lambda([x,e_{\alpha_3+\alpha_4}])=-\lambda([x_{\alpha_2}e_{\alpha_2},e_{\alpha_3+\alpha_4}])=-C_{\alpha_2,\alpha_3+\alpha_4}x_{\alpha_2}\lambda(e_{\alpha_2+\alpha_3+\alpha_4})\\
        =&-x_{\alpha_2}\lambda_{\alpha_2+\alpha_3+\alpha_4},\\
        \mu_{\alpha_3}=&-\lambda([x,e_{\alpha_3}])+\dfrac{1}{2}\lambda([x,[x,e_{\alpha_3}]])=-\lambda([x_{\alpha_2+\alpha_3}e_{\alpha_2+\alpha_3},e_{\alpha_3}])+\\
        &+\dfrac{1}{2}\lambda([x_{\alpha_2}e_{\alpha_2},[x_{\alpha_3}e_{\alpha_3}+x_{\alpha_4}e_{\alpha_4},e_{\alpha_3}])+\dfrac{1}{2}\lambda([x_{\alpha_4}e_{\alpha_4},[x_{\alpha_2}e_{\alpha_2},e_{\alpha_3}])\\
        =&-C_{\alpha_2+\alpha_3}x_{\alpha_2+\alpha_3}\lambda(e_{\alpha_2+2\alpha_3})+\dfrac{1}{2}C_{\alpha_2,\alpha_3}C_{\alpha_3,\alpha_2+\alpha_3}x_{\alpha_2}x_{\alpha_3}\lambda(e_{\alpha_2+2\alpha_3})+\\
        &+\dfrac{1}{2}(C_{\alpha_4,\alpha_3}C_{\alpha_2,\alpha_3+\alpha_4}+C_{\alpha_2,\alpha_3}C_{\alpha_4,\alpha_2+\alpha_3})x_{\alpha_2}x_{\alpha_4}\lambda(e_{\alpha_2+\alpha_3+\alpha_4})\\
        =&~2x_{\alpha_2+\alpha_3}\lambda_{\alpha_2+2\alpha_3}+x_{\alpha_2}x_{\alpha_3}\lambda_{\alpha_2+2\alpha_3}-x_{\alpha_2}x_{\alpha_4}\lambda_{\alpha_2+\alpha_3+\alpha_4},\\
        \mu_{\alpha_4}=&~\lambda(e_{\alpha_4})-\lambda([x,e_{\alpha_4}])+\frac{1}{2}\lambda([x,[x,e_{\alpha_4}]])=\lambda_{\alpha_4}-\lambda([x_{\alpha_2+\alpha_3}e_{\alpha_2+\alpha_3},e_{\alpha_4}])+\\
        &+\dfrac{1}{2}\lambda([x_{\alpha_3}e_{\alpha_3},[x_{\alpha_2}e_{\alpha_2},e_{\alpha_4}]])=\lambda_{\alpha_4}-C_{\alpha_2+\alpha_3,\alpha_4}x_{\alpha_2+\alpha_3}\lambda(e_{\alpha_2+\alpha_3+\alpha_4})+\\
        &+\dfrac{1}{2}C_{\alpha_3,\alpha_4}C_{\alpha_2,\alpha_3+\alpha_4}x_{\alpha_2}x_{\alpha_3}\lambda(e_{\alpha_2+\alpha_3+\alpha_4})=\lambda_{\alpha_4}+x_{\alpha_2+\alpha_3}\lambda_{\alpha_2+\alpha_3+\alpha_4}+\\
        &+\dfrac{1}{2}x_{\alpha_2}x_{\alpha_3}\lambda_{\alpha_2+\alpha_3+\alpha_4}=\lambda_{\alpha_4}+\dfrac{\mu_{\alpha_2+\alpha_3}\mu_{\alpha_3+\alpha_4}-\mu_{\alpha_2+\alpha_3+\alpha_4}\mu_{\alpha_3}}{2\mu_{\alpha_2+2\alpha_3}}.
    \end{split}
    \end{equation*}

    This shows that formula (\ref{FD}) holds for the root $\alpha_4$.

    \item Now suppose the largest root in the support equals $\alpha_2+2\alpha_3+\alpha_4$ or $\alpha_2+2\alpha_3+2\alpha_4$. In these cases, the statement also follows almost trivially from Lemma ~\ref{uno} and Lemma~\ref{zeta'}. The only nontrivial situation occurs for the support
    \begin{center}
        $D=\{\alpha_2+2\alpha_3+2\alpha_4,\alpha_2+2\alpha_3,\alpha_2\}$.
    \end{center}
     At first glance, it might seem that $T_{\alpha_2,\alpha_2+2\alpha_3}(x)$ cannot be expressed as a rational function in $\{(\exp(\chi)\cdot\lambda)(e_\beta)\}_{\beta\succ\alpha_2}$. However, this is not the case. From Lemma \ref{uno} we know that $T_{\alpha_2+\alpha_3,\alpha_2+2\alpha_3+2\alpha_4}(x)$ can be expressed as a rational function of \[\{(\exp(\chi)\cdot\lambda)(e_\beta)\}_{\beta\succ\alpha_2+\alpha_3+\alpha_4}.\] Hence, $x_{\alpha_3}$ can be recovered from $(\exp(\chi)\cdot\lambda)(e_{\alpha_2+\alpha_3})$ and $T_{\alpha_2+\alpha_3,\alpha_2+2\alpha_3+2\alpha_4}(x)$. Since $T_{\alpha_2,\alpha_2+2\alpha_3}(x)$ depends only on $x_{\alpha_3}$.
    \item If the largest root in the support equals $\alpha_1$ or $\alpha_1+\alpha_2$, then the nontrivial case is when $D$ contains $\alpha_1+\alpha_2,\alpha_2+2\alpha_3,\alpha_2+\alpha_3+\alpha_4$. This case follows immediately from the formulas in Example \ref{exam1} together with the case $D=\{\alpha_2+2\alpha_3,\alpha_2+\alpha_3+\alpha_4\}$.
    \item If the largest root in the support equals $\alpha_1+\alpha_2+\alpha_3$, then all nontrivial cases reduce to $D=\{\alpha_1+\alpha_2+\alpha_3,\alpha_2+2\alpha_3+\alpha_4,\alpha_2\}$. Let us write down the relevant coordinates:
    \begin{equation*}
    \begin{split}
        \mu_{\alpha_1+\alpha_2+\alpha_3}&=\lambda_{\alpha_1+\alpha_2+\alpha_3},\\
        \mu_{\alpha_1+\alpha_2}&=-x_{\alpha_3}\lambda_{\alpha_1+\alpha_2+\alpha_3},\\
        \mu_{\alpha_1}&=-x_{\alpha_2+\alpha_3}\lambda_{\alpha_1+\alpha_2+\alpha_3}+\dfrac{1}{2}x_{\alpha_2}x_{\alpha_3}\lambda_{\alpha_1+\alpha_2+\alpha_3},\\
        \mu_{\alpha_2+2\alpha_3+\alpha_4}&=\lambda_{\alpha_2+2\alpha_3+\alpha_4},\\
        \mu_{\alpha_2+2\alpha_3}&=-x_{\alpha_4}\lambda_{\alpha_2+2\alpha_3+\alpha_4},\\
        \mu_{\alpha_2+\alpha_3+\alpha_4}&=x_{\alpha_3}\lambda_{\alpha_2+2\alpha_3+\alpha_4}=-\dfrac{\mu_{\alpha_1+\alpha_2}\mu_{\alpha_2+2\alpha_3+\alpha_4}}{\mu_{\alpha_1+\alpha_2+\alpha_3}},\\
        \mu_{\alpha_2+\alpha_3}&=x_{\alpha_1}\lambda_{\alpha_1+\alpha_2+\alpha_3}+x_{\alpha_3+\alpha_4}\lambda_{\alpha_2+2\alpha_3+\alpha_4}-\dfrac{3}{2}x_{\alpha_3}x_{\alpha_4}\lambda_{\alpha_2+2\alpha_3+\alpha_4},\\
        \mu_{\alpha_2}&=\lambda_{\alpha_2}-x_{\alpha_1}x_{\alpha_3}\lambda_{\alpha_1+\alpha_2+\alpha_3}-x_{\alpha_3+\alpha_4}x_{\alpha_3}\lambda_{\alpha_2+2\alpha_3+\alpha_4}+\dfrac{1}{2}x_{\alpha_3}^2x_{\alpha_4}\lambda_{\alpha_2+2\alpha_3+\alpha_4}\\
        &=\lambda_{\alpha_2}+\dfrac{\mu_{\alpha_1+\alpha_2}\mu_{\alpha_2+\alpha_3}}{\mu_{\alpha_1+\alpha_2+\alpha_3}}+\dfrac{\mu_{\alpha_2+2\alpha_3}\mu_{\alpha_1+\alpha_2}^2}{\mu_{\alpha_1+\alpha_2+\alpha_3}^2},\\
        \mu_{\alpha_3+\alpha_4}&=-x_{\alpha_2+\alpha_3}\lambda_{\alpha_2+2\alpha_3+\alpha_4}+\dfrac{1}{2}x_{\alpha_2}x_{\alpha_3}\lambda_{\alpha_2+2\alpha_3+\alpha_4}=\dfrac{\mu_{\alpha_1}\mu_{\alpha_2+2\alpha_3+\alpha_4}}{\mu_{\alpha_1+\alpha_2+\alpha_3}}.
    \end{split}
    \end{equation*}

    We conclude that for the roots $\alpha_2$ and $\alpha_3+\alpha_4$, formula (\ref{FD}) holds.

    \item Next, if the largest root in the support equals $\alpha_1+\alpha_2+\alpha_3+\alpha_4$, then the most interesting supports are those containing the roots \[\alpha_1+\alpha_2+\alpha_3+\alpha_4, \quad \alpha_2+2\alpha_3+2\alpha_4 \quad \text{and} \quad \alpha_2+2\alpha_3+\alpha_4.\] One can easily show that
     $$\mu_{\alpha_2}=\lambda_{\alpha_2}-\dfrac{\mu_{\alpha_1+\alpha_2}^2\mu_{\alpha_2+2\alpha_3+2\alpha_4}}{\mu_{\alpha_1+\alpha_2+\alpha_3+\alpha_4}^2}+\dfrac{\mu_{\alpha_1+\alpha_2}\mu_{\alpha_2+\alpha_3+\alpha_4}}{\mu_{\alpha_1+\alpha_2+\alpha_3+\alpha_4}}-x_{\alpha_3}^2\lambda_{\alpha_2+2\alpha_3}.$$
    First, consider the case $\lambda_{\alpha_2+2\alpha_3}=0$. Then clearly the root $\alpha_2$ satisfies formula (\ref{FD}). Since $\alpha_2+2\alpha_3\notin D$, the support $D$ may contain either $\alpha_3$ or $\alpha_2+\alpha_3$. But
    \begin{equation*}
    \begin{split}
        \mu_{\alpha_3}&=\lambda_{\alpha_3}+\dfrac{\mu_{\alpha_1+\alpha_2+\alpha_3}\mu_{\alpha_3+\alpha_4}+\mu_{\alpha_1}\mu_{\alpha_2+2\alpha_3+\alpha_4}}{\mu_{\alpha_1+\alpha_2+\alpha_3+\alpha_4}}+2\dfrac{\mu_{\alpha_1+\alpha_2+\alpha_3}\mu_{\alpha_1}\mu_{\alpha_2+2\alpha_3+2\alpha_4}}{\mu_{\alpha_1+\alpha_2+\alpha_3+\alpha_4}^2},\\
        \mu_{\alpha_2+\alpha_3}&=\lambda_{\alpha_2+\alpha_3}+\dfrac{-\mu_{\alpha_2+2\alpha_3+\alpha_4}\mu_{\alpha_1+\alpha_2}+\mu_{\alpha_1+\alpha_2+\alpha_3}\mu_{\alpha_2+\alpha_3+\alpha_4}}{\mu_{\alpha_1+\alpha_2+\alpha_3+\alpha_4}}-\\&-2\dfrac{\mu_{\alpha_2+2\alpha_3+2\alpha_4}\mu_{\alpha_1+\alpha_2+\alpha_3}\mu_{\alpha_1+\alpha_2}}{\mu_{\alpha_1+\alpha_2+\alpha_3+\alpha_4}^2}.
    \end{split}
    \end{equation*}
     One can see that formula (\ref{FD}) is satisfied in both situations.
    Second, if  $\lambda_{\alpha_2+2\alpha_3}\neq0$, then formula (\ref{FD}) for $\alpha_2$ follows from the fact that $\lambda_{\alpha_2+2\alpha_3}$ can be obtained from $\mu_{\alpha_2+2\alpha_3}$, while $x_{\alpha_3}$ can be obtained from $\mu_{\alpha_2+\alpha_3}$.

    \item Suppose the largest root in the support equals $\alpha_1+\alpha_2+2\alpha_3$. Here the analysis essentially reduces to the two supports $\{\alpha_1+\alpha_2+2\alpha_3,\alpha_1+\alpha_2+\alpha_3+\alpha_4\}$ and $\{\alpha_1+\alpha_2+2\alpha_3,\alpha_2+2\alpha_3+\alpha_4\}$. (All other cases can be reduced to these two by arguments similar to those above.)
    For the first support, we have the following equations for the coordinates of $\mu$:
    \begin{equation*}
    \begin{split}
        \mu_{\alpha_2+\alpha_3}&=\dfrac{\mu_{\alpha_1+\alpha_2+\alpha_3}\mu_{\alpha_2+2\alpha_3}}{\mu_{\alpha_1+\alpha_2+2\alpha_3}};\\
        \mu_{\alpha_2}&=\dfrac{\mu_{\alpha_1+\alpha_2}\mu_{\alpha_2+2\alpha_3}}{\mu_{\alpha_1+\alpha_2+2\alpha_3}};\\
        \mu_{\alpha_4}&=\dfrac{-\mu_{\alpha_3}\mu_{\alpha_1+\alpha_2+\alpha_3+\alpha_4}-\mu_{\alpha_1+\alpha_2}\mu_{\alpha_3+\alpha_4}}{2\mu_{\alpha_1+\alpha_2+2\alpha_3}}.
    \end{split}
    \end{equation*}
    For the second support, the equations for the coordinates of $\mu$ are:
    \begin{equation*}
    \begin{split}
        \mu_{\alpha_2}&=-\dfrac{\mu_{\alpha_1+\alpha_2+\alpha_3}\mu_{\alpha_2+\alpha_3}}{2\mu_{\alpha_1+\alpha_2+2\alpha_3}}+\dfrac{\mu_{\alpha_1+\alpha_2+\alpha_3}^2\mu_{\alpha_2+2\alpha_3}}{4\mu_{\alpha_1+\alpha_2+2\alpha_3}^2};\\
        \mu_{\alpha_4}&=\dfrac{\mu_{\alpha_1+\alpha_2+\alpha_3}\mu_{\alpha_3+\alpha_4}-2\mu_{\alpha_1}\mu_{\alpha_2+2\alpha_3+\alpha_4}}{2\mu_{\alpha_1+\alpha_2+2\alpha_3}}.
    \end{split}
    \end{equation*}
    \item Suppose the largest root in the support equals $\alpha_1+\alpha_2+2\alpha_3+\alpha_4$. Although there are many such supports, they are all amenable to a simple analysis. Each such support does not contain the roots $\alpha_1+\alpha_2+2\alpha_3$, $\alpha_1+\alpha_2+\alpha_3+\alpha_4$, and $\alpha_1+\alpha_2+\alpha_3$. Consequently, $x_{\alpha_4}$, $x_{\alpha_3}$, and $x_{\alpha_3+\alpha_4}$ can always be expressed as rational function of $\mu_{\alpha_1+\alpha_2+2\alpha_3}$, $\mu_{\alpha_1+\alpha_2+\alpha_3+\alpha_4}$, $\mu_{\alpha_1+\alpha_2+\alpha_3}$ and $\mu_{\alpha_1+\alpha_2+2\alpha_3+\alpha_4}$ (in fact, as polynomials with coefficients from $\overline{\Fp}
    [\mu_{\alpha_1+\alpha_2+2\alpha_3+\alpha_4}^{-1}]$). Moreover, $x_{\alpha_1}$ can be expressed as  a rational function of $\mu_{\alpha_2+2\alpha_3+\alpha_4}$ and $\mu_{\alpha_1+\alpha_2+2\alpha_3+\alpha_4}$. It follows that if $\lambda_{\alpha_1+\alpha_2+2\alpha_3+\alpha_4}$ appears in a monomial in a coordinate $\mu_\alpha$, then that monomial can be expressed as a rational function of the coordinates $\mu_\beta$ with $\beta\succ\alpha$. Furthermore, almost all other monomials can also be expressed as polynomials in $x_{\alpha_4},~x_{\alpha_3},~x_{\alpha_3+\alpha_4}$ and $x_{\alpha_1}$. Finally, those monomials that cannot be expressed in this way satisfy formula (\ref{FD}) by the analysis above. A similar analysis applies to the cases where the largest root of the support is $2\alpha_1+3\alpha_2+4\alpha_3+2\alpha_4$.
\end{itemize}

The remaining steps can be analyzed similarly: either Lemma \ref{zeta'} applies, or the equations for the coordinates can be written out explicitly in a straightforward manner. Thus, Theorem \ref{fun} is proved.
\hfill$\square$

Finally, we formulate and prove the last theorem, which leads to the main result.

\begin{theorem}
    \label{dif}
Let $\lambda,\lambda'\in S$ with $\lambda\neq\lambda'$. Then $\overline{N}\cdot\lambda\neq \overline{N}\cdot\lambda'$.
\end{theorem}
\begin{proof}
    Assume the contrary. Let $\lambda \neq \lambda'$ but $\overline{N}\cdot\lambda = \overline{N}\cdot\lambda'$. Let $\gamma$ be the largest positive root such that $\lambda(e_\gamma) \neq \lambda'(e_\gamma)$. Then $\gamma$ belongs to the support of at least one of the linear forms $\lambda$ or $\lambda'$. Since both forms lie in $S$, the root $\gamma$ satisfies condition (\ref{rk}). Set $\mu = \exp(\chi)\cdot\lambda,~\mu'=\exp(\chi)\cdot\lambda'$. By Theorem \ref{fun}, there exist rational functions $F_{D,\gamma}$ and $F_{D',\gamma}$ such that $$\mu(e_\gamma) = \lambda(e_\gamma) + F_{D,\gamma}(\mu(e_\beta), \beta \succ \gamma)$$ and $$\mu'(e_\gamma) = \lambda'(e_\gamma) + F_{D',\gamma}(\mu'(e_\beta), \beta \succ \gamma).$$
Note that if $\lambda(e_{\beta})=\lambda'(e_{\beta})$ for all $\beta\succ\gamma$, then $\mu(e_\gamma)-\lambda(e_\gamma) = \mu'(e_\gamma)-\lambda'(e_\gamma)$ and $\mu(e_{\beta})=\mu'(e_{\beta})$ for all $\beta\succ\gamma$ . Therefore, we may choose our functions $F_{D,\gamma}$ for different $D$ and $\gamma$ in such a way that $F_{D,\gamma}=F_{D',\gamma}$ whenever $\{\beta\in D\mid\beta\succ\gamma\}=\{\beta\in D'\mid\beta\succ\gamma\}$.
Let $\wt\mu = \exp(\wt{x})\cdot\lambda$ be an arbitrary element of their common orbit, where $\wt{x} = \sum_{\alpha\in\Phi^+}\wt{x}_\alpha e_\alpha \in \overline\nt$. Then $$\wt\mu(e_\gamma) = \lambda(e_\gamma) + F_{D,\gamma}(\wt\mu(e_\beta), \beta \succ \gamma)$$ and $$\wt\mu(e_\gamma) = \lambda'(e_\gamma) + F_{D,\gamma}(\wt\mu(e_\beta), \beta \succ \gamma).$$
Note that $F_{D,\gamma}(\wt\mu(e_\beta), \beta \succ \gamma)$ does not become infinite because $\wt\mu(e_\gamma)$, $\lambda(e_\gamma)$, and $\lambda'(e_\gamma)$ are well-defined numbers. Hence, $\lambda(e_\gamma) = \lambda'(e_\gamma)$. This is a contradiction.
\end{proof}

Thus, distinct forms in $S$ have distinct orbits. Note that $N\cdot\lambda \subseteq\overline{N}\cdot\lambda$, so if $\lambda,\lambda'\in S$ and $\lambda\neq\lambda'$, then $N\cdot\lambda\neq N\cdot\lambda'$.

It remains to prove that all coadjoint orbits are obtained in this way. Consider the orbits over the finite field $\Fp_q$. One can easily compute $|S|$. Let $D$ be a support from the list, with $|D|=k$. There are $(q-1)^k$ linear forms with this support. If they all lie in $S$, then this support contributes $(q-1)^k$ forms to $S$. If there is one restriction on the scalars, then the support yields $(q-1)^{k-1}$ forms in $S$, because the restriction expresses one coordinate in terms of the others. For example, $$\lambda_{\alpha_1+\alpha_2}=\dfrac{\lambda_{\alpha_1+\alpha_2+2\alpha_3}\lambda_{\alpha_2+\alpha_3+\alpha_4}^2}{\lambda_{\alpha_2+2\alpha_3+\alpha_4}^2}$$ for the supports 30, 31, 32, see page \pageref{tablesupp}. Thus, a straightforward calculation gives
\begin{equation*}
    \begin{split}
        |S|&=1+24(q-1)+140(q-1)^2+288(q-1)^3+256(q-1)^4\\
        &+124(q-1)^5+40(q-1)^6+9(q-1)^7+(q-1)^8.
    \end{split}
\end{equation*}
The sum of the coefficients is exactly 883, which is the number of supports in the list plus three supports with restrictions.

Since each form in $S$ determines its own orbit, the set $S$ accounts for $|S|$ orbits. However, in the work \cite{GoodwinMoschRohle16}, the number of coadjoint orbits for $F_4$ over a finite field $\Fp_q$ is computed by group-theoretic methods, and it coincides exactly with $|S|$. Hence, over $\Fp_q$ there are no other orbits. This completes the proof of the main result.

The same approach can be applied to other root systems $\Phi$. Fix a total lexicographic order $\succ$ on the positive roots $\Phi^+$. For $\lambda\in\nt^*$ and $\gamma\in\Phi^+$, define the matrices
\begin{equation*}
\begin{split}
    A_{\lambda,\gamma}:=\left(\lambda([e_\alpha, e_\beta])\right)_{\alpha\in\Phi^+, \alpha\succeq\gamma, \beta\in\Phi^+}\cdot\\
    B_{\lambda,\gamma}:=\left(\lambda([e_\alpha, e_\beta])\right)_{\alpha\in\Phi^+, \alpha\succ\gamma, \beta\in\Phi^+}.
\end{split}
\end{equation*}
Let $S$ be the set of $\lambda\in\nt^*$ satisfying the following condition: every $\gamma\in \supp(\lambda)$ fulfills condition (\ref{rk}). We can now state a conjecture.

\begin{conjecture}
    \label{hyp}
For every coadjoint orbit\textup{,} there exists a unique linear form $\lambda\in S$ lying on that orbit.
\end{conjecture}

\begin{center}\large{\textbf{Appendix}}\end{center}
\begin{multicols}{3}
\begin{enumerate}\scriptsize
\item $\emptyset$
\item $\alpha_4$
\item $\alpha_3$
\item $\alpha_3,~\alpha_4$
\item $\alpha_3+\alpha_4$
\item $\alpha_2$
\item $\alpha_2,~\alpha_3+\alpha_4$
\item $\alpha_2,~\alpha_3$
\item $\alpha_2,~\alpha_3,~\alpha_4$
\item $\alpha_2,~\alpha_4$
\item $\alpha_2+\alpha_3$
\item $\alpha_2+\alpha_3,~\alpha_3+\alpha_4$
\item $\alpha_2+\alpha_3,~\alpha_3+\alpha_4,~\alpha_4$
\item $\alpha_2+\alpha_3,~\alpha_4$
\item $\alpha_2+\alpha_3+\alpha_4$
\item $\alpha_2+\alpha_3+\alpha_4,~\alpha_3$
\item $\alpha_2+2\alpha_3$
\item $\alpha_2+2\alpha_3,~\alpha_2+\alpha_3+\alpha_4$
\item $\alpha_2+2\alpha_3,~\alpha_2+\alpha_3+\alpha_4,~\alpha_4$
\item $\alpha_2+2\alpha_3,~\alpha_2$
\item $\alpha_2+2\alpha_3,~\alpha_2,~\alpha_3+\alpha_4$
\item $\alpha_2+2\alpha_3,~\alpha_2,~\alpha_3+\alpha_4,~\alpha_4$
\item $\alpha_2+2\alpha_3,~\alpha_2,~\alpha_4$
\item $\alpha_2+2\alpha_3,~\alpha_3+\alpha_4$
\item $\alpha_2+2\alpha_3,~\alpha_3+\alpha_4,~\alpha_4$
\item $\alpha_2+2\alpha_3,~\alpha_4$
\item $\alpha_2+2\alpha_3+\alpha_4$
\item $\alpha_2+2\alpha_3+\alpha_4,~\alpha_2$
\item $\alpha_2+2\alpha_3+2\alpha_4$
\item $\alpha_2+2\alpha_3+2\alpha_4,~\alpha_2+2\alpha_3$
\item $\alpha_2+2\alpha_3+2\alpha_4,~\alpha_2+2\alpha_3,~\alpha_2$
\item $\alpha_2+2\alpha_3+2\alpha_4,~\alpha_2+\alpha_3$
\item $\alpha_2+2\alpha_3+2\alpha_4,~\alpha_2$
\item $\alpha_2+2\alpha_3+2\alpha_4,~\alpha_2,~\alpha_3$
\item $\alpha_2+2\alpha_3+2\alpha_4,~\alpha_3$
\item $\alpha_1$
\item $\alpha_1,~\alpha_2+2\alpha_3+2\alpha_4$
\item $\alpha_1,~\alpha_2+2\alpha_3+2\alpha_4,~\alpha_2+2\alpha_3$
\item $\alpha_1,~\alpha_2+2\alpha_3+2\alpha_4,~\alpha_2+2\alpha_3,~\alpha_2$
\item $\alpha_1,~\alpha_2+2\alpha_3+2\alpha_4,~\alpha_2+\alpha_3$
\item $\alpha_1,~\alpha_2+2\alpha_3+2\alpha_4,~\alpha_2$
\item $\alpha_1,~\alpha_2+2\alpha_3+2\alpha_4,~\alpha_2,~\alpha_3$
\item $\alpha_1,~\alpha_2+2\alpha_3+2\alpha_4,~\alpha_3$
\item $\alpha_1,~\alpha_2+2\alpha_3+\alpha_4$
\item $\alpha_1,~\alpha_2+2\alpha_3+\alpha_4,~\alpha_2$
\item $\alpha_1,~\alpha_2+2\alpha_3$
\item $\alpha_1,~\alpha_2+2\alpha_3,~\alpha_2+\alpha_3+\alpha_4$
\item $\alpha_1,~\alpha_2+2\alpha_3,~\alpha_2+\alpha_3+\alpha_4,~\alpha_4$
\item $\alpha_1,~\alpha_2+2\alpha_3,~\alpha_2$
\item $\alpha_1,~\alpha_2+2\alpha_3,~\alpha_2,~\alpha_3+\alpha_4$
\item $\alpha_1,~\alpha_2+2\alpha_3,~\alpha_2,~\alpha_3+\alpha_4,~\alpha_4$
\item $\alpha_1,~\alpha_2+2\alpha_3,~\alpha_2,~\alpha_4$
\item $\alpha_1,~\alpha_2+2\alpha_3,~\alpha_3+\alpha_4$
\item $\alpha_1,~\alpha_2+2\alpha_3,~\alpha_3+\alpha_4,~\alpha_4$
\item $\alpha_1,~\alpha_2+2\alpha_3,~\alpha_4$
\item $\alpha_1,~\alpha_2+\alpha_3+\alpha_4$
\item $\alpha_1,~\alpha_2+\alpha_3+\alpha_4,~\alpha_3$
\item $\alpha_1,~\alpha_2+\alpha_3$
\item $\alpha_1,~\alpha_2+\alpha_3,~\alpha_3+\alpha_4$
\item $\alpha_1,~\alpha_2+\alpha_3,~\alpha_3+\alpha_4,~\alpha_4$
\item $\alpha_1,~\alpha_2+\alpha_3,~\alpha_4$
\item $\alpha_1,~\alpha_2$
\item $\alpha_1,~\alpha_2,~\alpha_3+\alpha_4$
\item $\alpha_1,~\alpha_2,~\alpha_3$
\item $\alpha_1,~\alpha_2,~\alpha_3,~\alpha_4$
\item $\alpha_1,~\alpha_2,~\alpha_4$
\item $\alpha_1,~\alpha_3+\alpha_4$
\item $\alpha_1,~\alpha_3$
\item $\alpha_1,~\alpha_3,~\alpha_4$
\item $\alpha_1,~\alpha_4$
\item $\alpha_1+\alpha_2$
\item $\alpha_1+\alpha_2,~\alpha_2+2\alpha_3+2\alpha_4$
\item $\alpha_1+\alpha_2,~\alpha_2+2\alpha_3+2\alpha_4,~\alpha_2+2\alpha_3$
\item $\alpha_1+\alpha_2,~\alpha_2+2\alpha_3+2\alpha_4,~\alpha_2+\alpha_3$
\item $\alpha_1+\alpha_2,~\alpha_2+2\alpha_3+2\alpha_4,~\alpha_2+\alpha_3,~\alpha_3$
\item $\alpha_1+\alpha_2,~\alpha_2+2\alpha_3+2\alpha_4,~\alpha_3$
\item $\alpha_1+\alpha_2,~\alpha_2+2\alpha_3+\alpha_4$
\item $\alpha_1+\alpha_2,~\alpha_2+2\alpha_3$
\item $\alpha_1+\alpha_2,~\alpha_2+2\alpha_3,~\alpha_2+\alpha_3+\alpha_4$
\item $\alpha_1+\alpha_2,~\alpha_2+2\alpha_3,~\alpha_2+\alpha_3+\alpha_4,~\alpha_3+\alpha_4$
\item $\alpha_1+\alpha_2,~\alpha_2+2\alpha_3,~\alpha_2+\alpha_3+\alpha_4,~\alpha_3+\alpha_4,~\alpha_4$
\item $\alpha_1+\alpha_2,~\alpha_2+2\alpha_3,~\alpha_2+\alpha_3+\alpha_4,~\alpha_4$
\item $\alpha_1+\alpha_2,~\alpha_2+2\alpha_3,~\alpha_3+\alpha_4$
\item $\alpha_1+\alpha_2,~\alpha_2+2\alpha_3,~\alpha_3+\alpha_4,~\alpha_4$
\item $\alpha_1+\alpha_2,~\alpha_2+2\alpha_3,~\alpha_4$
\item $\alpha_1+\alpha_2,~\alpha_2+\alpha_3+\alpha_4$
\item $\alpha_1+\alpha_2,~\alpha_2+\alpha_3+\alpha_4,~\alpha_3+\alpha_4$
\item $\alpha_1+\alpha_2,~\alpha_2+\alpha_3+\alpha_4,~\alpha_3+\alpha_4,~\alpha_3$
\item $\alpha_1+\alpha_2,~\alpha_2+\alpha_3+\alpha_4,~\alpha_3$
\item $\alpha_1+\alpha_2,~\alpha_2+\alpha_3$
\item $\alpha_1+\alpha_2,~\alpha_2+\alpha_3,~\alpha_3+\alpha_4$
\item $\alpha_1+\alpha_2,~\alpha_2+\alpha_3,~\alpha_3$
\item $\alpha_1+\alpha_2,~\alpha_2+\alpha_3,~\alpha_3,~\alpha_4$
\item $\alpha_1+\alpha_2,~\alpha_2+\alpha_3,~\alpha_4$
\item $\alpha_1+\alpha_2,~\alpha_3+\alpha_4$
\item $\alpha_1+\alpha_2,~\alpha_3$
\item $\alpha_1+\alpha_2,~\alpha_3,~\alpha_4$
\item $\alpha_1+\alpha_2,~\alpha_4$
\item $\alpha_1+\alpha_2+\alpha_3$
\item $\alpha_1+\alpha_2+\alpha_3,~\alpha_2+2\alpha_3+2\alpha_4$
\item $\alpha_1+\alpha_2+\alpha_3,~\alpha_2+2\alpha_3+2\alpha_4,~\alpha_2+2\alpha_3$
\item $\alpha_1+\alpha_2+\alpha_3,~\alpha_2+2\alpha_3+2\alpha_4,~\alpha_2+2\alpha_3,~\alpha_2$
\item $\alpha_1+\alpha_2+\alpha_3,~\alpha_2+2\alpha_3+2\alpha_4,~\alpha_2$
\item $\alpha_1+\alpha_2+\alpha_3,~\alpha_2+2\alpha_3+\alpha_4$
\item $\alpha_1+\alpha_2+\alpha_3,~\alpha_2+2\alpha_3+\alpha_4,~\alpha_2+\alpha_3+\alpha_4$
\item $\alpha_1+\alpha_2+\alpha_3,~\alpha_2+2\alpha_3+\alpha_4,~\alpha_2$
\item $\alpha_1+\alpha_2+\alpha_3,~\alpha_2+2\alpha_3+\alpha_4,~\alpha_2,~\alpha_3+\alpha_4$
\item $\alpha_1+\alpha_2+\alpha_3,~\alpha_2+2\alpha_3+\alpha_4,~\alpha_3+\alpha_4$
\item $\alpha_1+\alpha_2+\alpha_3,~\alpha_2+2\alpha_3$
\item $\alpha_1+\alpha_2+\alpha_3,~\alpha_2+2\alpha_3,~\alpha_2+\alpha_3+\alpha_4$
\item $\alpha_1+\alpha_2+\alpha_3,~\alpha_2+2\alpha_3,~\alpha_2+\alpha_3+\alpha_4,~\alpha_4$
\item $\alpha_1+\alpha_2+\alpha_3,~\alpha_2+2\alpha_3,~\alpha_2$
\item $\alpha_1+\alpha_2+\alpha_3,~\alpha_2+2\alpha_3,~\alpha_2,~\alpha_3+\alpha_4$
\item $\alpha_1+\alpha_2+\alpha_3,~\alpha_2+2\alpha_3,~\alpha_2,~\alpha_3+\alpha_4,~\alpha_4$
\item $\alpha_1+\alpha_2+\alpha_3,~\alpha_2+2\alpha_3,~\alpha_2,~\alpha_4$
\item $\alpha_1+\alpha_2+\alpha_3,~\alpha_2+2\alpha_3,~\alpha_3+\alpha_4$
\item $\alpha_1+\alpha_2+\alpha_3,~\alpha_2+2\alpha_3,~\alpha_3+\alpha_4,~\alpha_4$
\item $\alpha_1+\alpha_2+\alpha_3,~\alpha_2+2\alpha_3,~\alpha_4$
\item $\alpha_1+\alpha_2+\alpha_3,~\alpha_2+\alpha_3+\alpha_4$
\item $\alpha_1+\alpha_2+\alpha_3,~\alpha_2+\alpha_3+\alpha_4,~\alpha_4$
\item $\alpha_1+\alpha_2+\alpha_3,~\alpha_2$
\item $\alpha_1+\alpha_2+\alpha_3,~\alpha_2,~\alpha_3+\alpha_4$
\item $\alpha_1+\alpha_2+\alpha_3,~\alpha_2,~\alpha_3+\alpha_4,~\alpha_4$
\item $\alpha_1+\alpha_2+\alpha_3,~\alpha_2,~\alpha_4$
\item $\alpha_1+\alpha_2+\alpha_3,~\alpha_3+\alpha_4$
\item $\alpha_1+\alpha_2+\alpha_3,~\alpha_3+\alpha_4,~\alpha_4$
\item $\alpha_1+\alpha_2+\alpha_3,~\alpha_4$
\item $\alpha_1+\alpha_2+\alpha_3+\alpha_4$
\item $\alpha_1+\alpha_2+\alpha_3+\alpha_4,~\alpha_2+2\alpha_3+2\alpha_4$
\item $\alpha_1+\alpha_2+\alpha_3+\alpha_4,~\alpha_2+2\alpha_3+2\alpha_4,~\alpha_2+2\alpha_3+\alpha_4$
\item $\alpha_1+\alpha_2+\alpha_3+\alpha_4,~\alpha_2+2\alpha_3+2\alpha_4,~\alpha_2+2\alpha_3+\alpha_4,~\alpha_2+2\alpha_3$
\item $\alpha_1+\alpha_2+\alpha_3+\alpha_4,~\alpha_2+2\alpha_3+2\alpha_4,~\alpha_2+2\alpha_3+\alpha_4,~\alpha_2+2\alpha_3,~\alpha_2$
\item $\alpha_1+\alpha_2+\alpha_3+\alpha_4,~\alpha_2+2\alpha_3+2\alpha_4,~\alpha_2+2\alpha_3+\alpha_4,~\alpha_2+\alpha_3$
\item $\alpha_1+\alpha_2+\alpha_3+\alpha_4,~\alpha_2+2\alpha_3+2\alpha_4,~\alpha_2+2\alpha_3+\alpha_4,~\alpha_2$
\item $\alpha_1+\alpha_2+\alpha_3+\alpha_4,~\alpha_2+2\alpha_3+2\alpha_4,~\alpha_2+2\alpha_3+\alpha_4,~\alpha_2,~\alpha_3$
\item $\alpha_1+\alpha_2+\alpha_3+\alpha_4,~\alpha_2+2\alpha_3+2\alpha_4,~\alpha_2+2\alpha_3+\alpha_4,~\alpha_3$
\item $\alpha_1+\alpha_2+\alpha_3+\alpha_4,~\alpha_2+2\alpha_3+2\alpha_4,~\alpha_2+2\alpha_3$
\item $\alpha_1+\alpha_2+\alpha_3+\alpha_4,~\alpha_2+2\alpha_3+2\alpha_4,~\alpha_2+2\alpha_3,~\alpha_2$
\item $\alpha_1+\alpha_2+\alpha_3+\alpha_4,~\alpha_2+2\alpha_3+2\alpha_4,~\alpha_2+\alpha_3$
\item $\alpha_1+\alpha_2+\alpha_3+\alpha_4,~\alpha_2+2\alpha_3+2\alpha_4,~\alpha_2$
\item $\alpha_1+\alpha_2+\alpha_3+\alpha_4,~\alpha_2+2\alpha_3+2\alpha_4,~\alpha_2,~\alpha_3$
\item $\alpha_1+\alpha_2+\alpha_3+\alpha_4,~\alpha_2+2\alpha_3+2\alpha_4,~\alpha_3$
\item $\alpha_1+\alpha_2+\alpha_3+\alpha_4,~\alpha_2+2\alpha_3+\alpha_4$
\item $\alpha_1+\alpha_2+\alpha_3+\alpha_4,~\alpha_2+2\alpha_3+\alpha_4,~\alpha_2+2\alpha_3$
\item $\alpha_1+\alpha_2+\alpha_3+\alpha_4,~\alpha_2+2\alpha_3+\alpha_4,~\alpha_2+2\alpha_3,~\alpha_2$
\item $\alpha_1+\alpha_2+\alpha_3+\alpha_4,~\alpha_2+2\alpha_3+\alpha_4,~\alpha_2+\alpha_3$
\item $\alpha_1+\alpha_2+\alpha_3+\alpha_4,~\alpha_2+2\alpha_3+\alpha_4,~\alpha_2$
\item $\alpha_1+\alpha_2+\alpha_3+\alpha_4,~\alpha_2+2\alpha_3+\alpha_4,~\alpha_2,~\alpha_3$
\item $\alpha_1+\alpha_2+\alpha_3+\alpha_4,~\alpha_2+2\alpha_3+\alpha_4,~\alpha_3$
\item $\alpha_1+\alpha_2+\alpha_3+\alpha_4,~\alpha_2+2\alpha_3$
\item $\alpha_1+\alpha_2+\alpha_3+\alpha_4,~\alpha_2+2\alpha_3,~\alpha_2$
\item $\alpha_1+\alpha_2+\alpha_3+\alpha_4,~\alpha_2+\alpha_3$
\item $\alpha_1+\alpha_2+\alpha_3+\alpha_4,~\alpha_2$
\item $\alpha_1+\alpha_2+\alpha_3+\alpha_4,~\alpha_2,~\alpha_3$
\item $\alpha_1+\alpha_2+\alpha_3+\alpha_4,~\alpha_3$
\item $\alpha_1+\alpha_2+2\alpha_3$
\item $\alpha_1+\alpha_2+2\alpha_3,~\alpha_1+\alpha_2+\alpha_3+\alpha_4$
\item $\alpha_1+\alpha_2+2\alpha_3,~\alpha_1+\alpha_2+\alpha_3+\alpha_4,~\alpha_2+2\alpha_3+2\alpha_4$
\item $\alpha_1+\alpha_2+2\alpha_3,~\alpha_1+\alpha_2+\alpha_3+\alpha_4,~\alpha_2+2\alpha_3+2\alpha_4,~\alpha_2+\alpha_3+\alpha_4$
\item $\alpha_1+\alpha_2+2\alpha_3,~\alpha_1+\alpha_2+\alpha_3+\alpha_4,~\alpha_2+2\alpha_3+2\alpha_4,~\alpha_2+\alpha_3+\alpha_4,~\alpha_2+\alpha_3$
\item $\alpha_1+\alpha_2+2\alpha_3,~\alpha_1+\alpha_2+\alpha_3+\alpha_4,~\alpha_2+2\alpha_3+2\alpha_4,~\alpha_2+\alpha_3+\alpha_4,~\alpha_2+\alpha_3,~\alpha_2$
\item $\alpha_1+\alpha_2+2\alpha_3,~\alpha_1+\alpha_2+\alpha_3+\alpha_4,~\alpha_2+2\alpha_3+2\alpha_4,~\alpha_2+\alpha_3+\alpha_4,~\alpha_2$
\item $\alpha_1+\alpha_2+2\alpha_3,~\alpha_1+\alpha_2+\alpha_3+\alpha_4,~\alpha_2+2\alpha_3+2\alpha_4,~\alpha_2+\alpha_3$
\item $\alpha_1+\alpha_2+2\alpha_3,~\alpha_1+\alpha_2+\alpha_3+\alpha_4,~\alpha_2+2\alpha_3+2\alpha_4,~\alpha_2+\alpha_3,~\alpha_2$
\item $\alpha_1+\alpha_2+2\alpha_3,~\alpha_1+\alpha_2+\alpha_3+\alpha_4,~\alpha_2+2\alpha_3+2\alpha_4,~\alpha_2$
\item $\alpha_1+\alpha_2+2\alpha_3,~\alpha_1+\alpha_2+\alpha_3+\alpha_4,~\alpha_2+2\alpha_3+\alpha_4$
\item $\alpha_1+\alpha_2+2\alpha_3,~\alpha_1+\alpha_2+\alpha_3+\alpha_4,~\alpha_2+2\alpha_3+\alpha_4,~\alpha_2+\alpha_3$
\item $\alpha_1+\alpha_2+2\alpha_3,~\alpha_1+\alpha_2+\alpha_3+\alpha_4,~\alpha_2+2\alpha_3+\alpha_4,~\alpha_2+\alpha_3,~\alpha_2$
\item $\alpha_1+\alpha_2+2\alpha_3,~\alpha_1+\alpha_2+\alpha_3+\alpha_4,~\alpha_2+2\alpha_3+\alpha_4,~\alpha_2$
\item $\alpha_1+\alpha_2+2\alpha_3,~\alpha_1+\alpha_2+\alpha_3+\alpha_4,~\alpha_2+\alpha_3+\alpha_4$
\item $\alpha_1+\alpha_2+2\alpha_3,~\alpha_1+\alpha_2+\alpha_3+\alpha_4,~\alpha_2+\alpha_3+\alpha_4,~\alpha_2$
\item $\alpha_1+\alpha_2+2\alpha_3,~\alpha_1+\alpha_2+\alpha_3+\alpha_4,~\alpha_2+\alpha_3$
\item $\alpha_1+\alpha_2+2\alpha_3,~\alpha_1+\alpha_2+\alpha_3+\alpha_4,~\alpha_2$
\item $\alpha_1+\alpha_2+2\alpha_3,~\alpha_1+\alpha_2+\alpha_3+\alpha_4,~\alpha_2,~\alpha_4$
\item $\alpha_1+\alpha_2+2\alpha_3,~\alpha_1+\alpha_2+\alpha_3+\alpha_4,~\alpha_4$
\item $\alpha_1+\alpha_2+2\alpha_3,~\alpha_1+\alpha_2$
\item $\alpha_1+\alpha_2+2\alpha_3,~\alpha_1+\alpha_2,~\alpha_2+2\alpha_3+2\alpha_4$
\item $\alpha_1+\alpha_2+2\alpha_3,~\alpha_1+\alpha_2,~\alpha_2+2\alpha_3+2\alpha_4,~\alpha_2+\alpha_3$
\item $\alpha_1+\alpha_2+2\alpha_3,~\alpha_1+\alpha_2,~\alpha_2+2\alpha_3+2\alpha_4,~\alpha_2+\alpha_3,~\alpha_2$
\item $\alpha_1+\alpha_2+2\alpha_3,~\alpha_1+\alpha_2,~\alpha_2+2\alpha_3+2\alpha_4,~\alpha_2$
\item $\alpha_1+\alpha_2+2\alpha_3,~\alpha_1+\alpha_2,~\alpha_2+2\alpha_3+\alpha_4$
\item $\alpha_1+\alpha_2+2\alpha_3,~\alpha_1+\alpha_2,~\alpha_2+2\alpha_3+\alpha_4,~\alpha_2+\alpha_3+\alpha_4$
\item $\alpha_1+\alpha_2+2\alpha_3,~\alpha_1+\alpha_2,~\alpha_2+\alpha_3+\alpha_4$
\item $\alpha_1+\alpha_2+2\alpha_3,~\alpha_1+\alpha_2,~\alpha_2+\alpha_3$
\item $\alpha_1+\alpha_2+2\alpha_3,~\alpha_1+\alpha_2,~\alpha_2+\alpha_3,~\alpha_2$
\item $\alpha_1+\alpha_2+2\alpha_3,~\alpha_1+\alpha_2,~\alpha_2+\alpha_3,~\alpha_2,~\alpha_3+\alpha_4$
\item $\alpha_1+\alpha_2+2\alpha_3,~\alpha_1+\alpha_2,~\alpha_2+\alpha_3,~\alpha_2,~\alpha_3+\alpha_4,~\alpha_4$
\item $\alpha_1+\alpha_2+2\alpha_3,~\alpha_1+\alpha_2,~\alpha_2+\alpha_3,~\alpha_2,~\alpha_4$
\item $\alpha_1+\alpha_2+2\alpha_3,~\alpha_1+\alpha_2,~\alpha_2+\alpha_3,~\alpha_3+\alpha_4$
\item $\alpha_1+\alpha_2+2\alpha_3,~\alpha_1+\alpha_2,~\alpha_2+\alpha_3,~\alpha_3+\alpha_4,~\alpha_4$
\item $\alpha_1+\alpha_2+2\alpha_3,~\alpha_1+\alpha_2,~\alpha_2+\alpha_3,~\alpha_4$
\item $\alpha_1+\alpha_2+2\alpha_3,~\alpha_1+\alpha_2,~\alpha_2$
\item $\alpha_1+\alpha_2+2\alpha_3,~\alpha_1+\alpha_2,~\alpha_2,~\alpha_3+\alpha_4$
\item $\alpha_1+\alpha_2+2\alpha_3,~\alpha_1+\alpha_2,~\alpha_2,~\alpha_3+\alpha_4,~\alpha_4$
\item $\alpha_1+\alpha_2+2\alpha_3,~\alpha_1+\alpha_2,~\alpha_2,~\alpha_4$
\item $\alpha_1+\alpha_2+2\alpha_3,~\alpha_1+\alpha_2,~\alpha_3+\alpha_4$
\item $\alpha_1+\alpha_2+2\alpha_3,~\alpha_1+\alpha_2,~\alpha_3+\alpha_4,~\alpha_4$
\item $\alpha_1+\alpha_2+2\alpha_3,~\alpha_1+\alpha_2,~\alpha_4$
\item $\alpha_1+\alpha_2+2\alpha_3,~\alpha_2+2\alpha_3+2\alpha_4$
\item $\alpha_1+\alpha_2+2\alpha_3,~\alpha_2+2\alpha_3+2\alpha_4,~\alpha_2+\alpha_3$
\item $\alpha_1+\alpha_2+2\alpha_3,~\alpha_2+2\alpha_3+2\alpha_4,~\alpha_2+\alpha_3,~\alpha_2$
\item $\alpha_1+\alpha_2+2\alpha_3,~\alpha_2+2\alpha_3+2\alpha_4,~\alpha_2$
\item $\alpha_1+\alpha_2+2\alpha_3,~\alpha_2+2\alpha_3+\alpha_4$
\item $\alpha_1+\alpha_2+2\alpha_3,~\alpha_2+2\alpha_3+\alpha_4,~\alpha_2+\alpha_3+\alpha_4$
\item $\alpha_1+\alpha_2+2\alpha_3,~\alpha_2+2\alpha_3+\alpha_4,~\alpha_2$
\item $\alpha_1+\alpha_2+2\alpha_3,~\alpha_2+2\alpha_3+\alpha_4,~\alpha_2,~\alpha_4$
\item $\alpha_1+\alpha_2+2\alpha_3,~\alpha_2+2\alpha_3+\alpha_4,~\alpha_4$
\item $\alpha_1+\alpha_2+2\alpha_3,~\alpha_2+\alpha_3+\alpha_4$
\item $\alpha_1+\alpha_2+2\alpha_3,~\alpha_2+\alpha_3$
\item $\alpha_1+\alpha_2+2\alpha_3,~\alpha_2+\alpha_3,~\alpha_2$
\item $\alpha_1+\alpha_2+2\alpha_3,~\alpha_2+\alpha_3,~\alpha_2,~\alpha_3+\alpha_4$
\item $\alpha_1+\alpha_2+2\alpha_3,~\alpha_2+\alpha_3,~\alpha_2,~\alpha_3+\alpha_4,~\alpha_4$
\item $\alpha_1+\alpha_2+2\alpha_3,~\alpha_2+\alpha_3,~\alpha_2,~\alpha_4$
\item $\alpha_1+\alpha_2+2\alpha_3,~\alpha_2+\alpha_3,~\alpha_3+\alpha_4$
\item $\alpha_1+\alpha_2+2\alpha_3,~\alpha_2+\alpha_3,~\alpha_3+\alpha_4,~\alpha_4$
\item $\alpha_1+\alpha_2+2\alpha_3,~\alpha_2+\alpha_3,~\alpha_4$
\item $\alpha_1+\alpha_2+2\alpha_3,~\alpha_2$
\item $\alpha_1+\alpha_2+2\alpha_3,~\alpha_2,~\alpha_3+\alpha_4$
\item $\alpha_1+\alpha_2+2\alpha_3,~\alpha_2,~\alpha_3+\alpha_4,~\alpha_4$
\item $\alpha_1+\alpha_2+2\alpha_3,~\alpha_2,~\alpha_4$
\item $\alpha_1+\alpha_2+2\alpha_3,~\alpha_3+\alpha_4$
\item $\alpha_1+\alpha_2+2\alpha_3,~\alpha_3+\alpha_4,~\alpha_4$
\item $\alpha_1+\alpha_2+2\alpha_3,~\alpha_4$
\item $\alpha_1+\alpha_2+2\alpha_3+\alpha_4$
\item $\alpha_1+\alpha_2+2\alpha_3+\alpha_4,~\alpha_1+\alpha_2$
\item $\alpha_1+\alpha_2+2\alpha_3+\alpha_4,~\alpha_1+\alpha_2,~\alpha_2+2\alpha_3+2\alpha_4$
\item $\alpha_1+\alpha_2+2\alpha_3+\alpha_4,~\alpha_1+\alpha_2,~\alpha_2+2\alpha_3+2\alpha_4,~\alpha_2+2\alpha_3$
\item $\alpha_1+\alpha_2+2\alpha_3+\alpha_4,~\alpha_1+\alpha_2,~\alpha_2+2\alpha_3+2\alpha_4,~\alpha_2+2\alpha_3,~\alpha_2+\alpha_3+\alpha_4$
\item $\alpha_1+\alpha_2+2\alpha_3+\alpha_4,~\alpha_1+\alpha_2,~\alpha_2+2\alpha_3+2\alpha_4,~\alpha_2+2\alpha_3,~\alpha_2+\alpha_3+\alpha_4,~\alpha_2+\alpha_3$
\item $\alpha_1+\alpha_2+2\alpha_3+\alpha_4,~\alpha_1+\alpha_2,~\alpha_2+2\alpha_3+2\alpha_4,~\alpha_2+2\alpha_3,~\alpha_2+\alpha_3+\alpha_4,~\alpha_2+\alpha_3,~\alpha_2$
\item $\alpha_1+\alpha_2+2\alpha_3+\alpha_4,~\alpha_1+\alpha_2,~\alpha_2+2\alpha_3+2\alpha_4,~\alpha_2+2\alpha_3,~\alpha_2+\alpha_3+\alpha_4,~\alpha_2$
\item $\alpha_1+\alpha_2+2\alpha_3+\alpha_4,~\alpha_1+\alpha_2,~\alpha_2+2\alpha_3+2\alpha_4,~\alpha_2+2\alpha_3,~\alpha_2+\alpha_3$
\item $\alpha_1+\alpha_2+2\alpha_3+\alpha_4,~\alpha_1+\alpha_2,~\alpha_2+2\alpha_3+2\alpha_4,~\alpha_2+2\alpha_3,~\alpha_2+\alpha_3,~\alpha_2$
\item $\alpha_1+\alpha_2+2\alpha_3+\alpha_4,~\alpha_1+\alpha_2,~\alpha_2+2\alpha_3+2\alpha_4,~\alpha_2+2\alpha_3,~\alpha_2$
\item $\alpha_1+\alpha_2+2\alpha_3+\alpha_4,~\alpha_1+\alpha_2,~\alpha_2+2\alpha_3+2\alpha_4,~\alpha_2+\alpha_3+\alpha_4$
\item $\alpha_1+\alpha_2+2\alpha_3+\alpha_4,~\alpha_1+\alpha_2,~\alpha_2+2\alpha_3+2\alpha_4,~\alpha_2+\alpha_3+\alpha_4,~\alpha_2+\alpha_3$
\item $\alpha_1+\alpha_2+2\alpha_3+\alpha_4,~\alpha_1+\alpha_2,~\alpha_2+2\alpha_3+2\alpha_4,~\alpha_2+\alpha_3+\alpha_4,~\alpha_2+\alpha_3,~\alpha_2$
\item $\alpha_1+\alpha_2+2\alpha_3+\alpha_4,~\alpha_1+\alpha_2,~\alpha_2+2\alpha_3+2\alpha_4,~\alpha_2+\alpha_3+\alpha_4,~\alpha_2$
\item $\alpha_1+\alpha_2+2\alpha_3+\alpha_4,~\alpha_1+\alpha_2,~\alpha_2+2\alpha_3+2\alpha_4,~\alpha_2+\alpha_3$
\item $\alpha_1+\alpha_2+2\alpha_3+\alpha_4,~\alpha_1+\alpha_2,~\alpha_2+2\alpha_3+2\alpha_4,~\alpha_2+\alpha_3,~\alpha_2$
\item $\alpha_1+\alpha_2+2\alpha_3+\alpha_4,~\alpha_1+\alpha_2,~\alpha_2+2\alpha_3+2\alpha_4,~\alpha_2$
\item $\alpha_1+\alpha_2+2\alpha_3+\alpha_4,~\alpha_1+\alpha_2,~\alpha_2+2\alpha_3$
\item $\alpha_1+\alpha_2+2\alpha_3+\alpha_4,~\alpha_1+\alpha_2,~\alpha_2+2\alpha_3,~\alpha_2+\alpha_3+\alpha_4$
\item $\alpha_1+\alpha_2+2\alpha_3+\alpha_4,~\alpha_1+\alpha_2,~\alpha_2+2\alpha_3,~\alpha_2+\alpha_3+\alpha_4,~\alpha_2+\alpha_3$
\item $\alpha_1+\alpha_2+2\alpha_3+\alpha_4,~\alpha_1+\alpha_2,~\alpha_2+2\alpha_3,~\alpha_2+\alpha_3+\alpha_4,~\alpha_2+\alpha_3,~\alpha_2$
\item $\alpha_1+\alpha_2+2\alpha_3+\alpha_4,~\alpha_1+\alpha_2,~\alpha_2+2\alpha_3,~\alpha_2+\alpha_3+\alpha_4,~\alpha_2$
\item $\alpha_1+\alpha_2+2\alpha_3+\alpha_4,~\alpha_1+\alpha_2,~\alpha_2+2\alpha_3,~\alpha_2+\alpha_3$
\item $\alpha_1+\alpha_2+2\alpha_3+\alpha_4,~\alpha_1+\alpha_2,~\alpha_2+2\alpha_3,~\alpha_2+\alpha_3,~\alpha_2$
\item $\alpha_1+\alpha_2+2\alpha_3+\alpha_4,~\alpha_1+\alpha_2,~\alpha_2+2\alpha_3,~\alpha_2$
\item $\alpha_1+\alpha_2+2\alpha_3+\alpha_4,~\alpha_1+\alpha_2,~\alpha_2+\alpha_3+\alpha_4$
\item $\alpha_1+\alpha_2+2\alpha_3+\alpha_4,~\alpha_1+\alpha_2,~\alpha_2+\alpha_3+\alpha_4,~\alpha_2+\alpha_3$
\item $\alpha_1+\alpha_2+2\alpha_3+\alpha_4,~\alpha_1+\alpha_2,~\alpha_2+\alpha_3+\alpha_4,~\alpha_2+\alpha_3,~\alpha_2$
\item $\alpha_1+\alpha_2+2\alpha_3+\alpha_4,~\alpha_1+\alpha_2,~\alpha_2+\alpha_3+\alpha_4,~\alpha_2$
\item $\alpha_1+\alpha_2+2\alpha_3+\alpha_4,~\alpha_1+\alpha_2,~\alpha_2+\alpha_3$
\item $\alpha_1+\alpha_2+2\alpha_3+\alpha_4,~\alpha_1+\alpha_2,~\alpha_2+\alpha_3,~\alpha_2$
\item $\alpha_1+\alpha_2+2\alpha_3+\alpha_4,~\alpha_1+\alpha_2,~\alpha_2$
\item $\alpha_1+\alpha_2+2\alpha_3+\alpha_4,~\alpha_2+2\alpha_3+2\alpha_4$
\item $\alpha_1+\alpha_2+2\alpha_3+\alpha_4,~\alpha_2+2\alpha_3+2\alpha_4,~\alpha_2+2\alpha_3$
\item $\alpha_1+\alpha_2+2\alpha_3+\alpha_4,~\alpha_2+2\alpha_3+2\alpha_4,~\alpha_2+2\alpha_3,~\alpha_2+\alpha_3+\alpha_4$
\item $\alpha_1+\alpha_2+2\alpha_3+\alpha_4,~\alpha_2+2\alpha_3+2\alpha_4,~\alpha_2+2\alpha_3,~\alpha_2+\alpha_3+\alpha_4,~\alpha_2+\alpha_3$
\item $\alpha_1+\alpha_2+2\alpha_3+\alpha_4,~\alpha_2+2\alpha_3+2\alpha_4,~\alpha_2+2\alpha_3,~\alpha_2+\alpha_3+\alpha_4,~\alpha_2+\alpha_3,~\alpha_2$
\item $\alpha_1+\alpha_2+2\alpha_3+\alpha_4,~\alpha_2+2\alpha_3+2\alpha_4,~\alpha_2+2\alpha_3,~\alpha_2+\alpha_3+\alpha_4,~\alpha_2$
\item $\alpha_1+\alpha_2+2\alpha_3+\alpha_4,~\alpha_2+2\alpha_3+2\alpha_4,~\alpha_2+2\alpha_3,~\alpha_2+\alpha_3$
\item $\alpha_1+\alpha_2+2\alpha_3+\alpha_4,~\alpha_2+2\alpha_3+2\alpha_4,~\alpha_2+2\alpha_3,~\alpha_2+\alpha_3,~\alpha_2$
\item $\alpha_1+\alpha_2+2\alpha_3+\alpha_4,~\alpha_2+2\alpha_3+2\alpha_4,~\alpha_2+2\alpha_3,~\alpha_2$
\item $\alpha_1+\alpha_2+2\alpha_3+\alpha_4,~\alpha_2+2\alpha_3+2\alpha_4,~\alpha_2+\alpha_3+\alpha_4$
\item $\alpha_1+\alpha_2+2\alpha_3+\alpha_4,~\alpha_2+2\alpha_3+2\alpha_4,~\alpha_2+\alpha_3+\alpha_4,~\alpha_2+\alpha_3$
\item $\alpha_1+\alpha_2+2\alpha_3+\alpha_4,~\alpha_2+2\alpha_3+2\alpha_4,~\alpha_2+\alpha_3+\alpha_4,~\alpha_2+\alpha_3,~\alpha_2$
\item $\alpha_1+\alpha_2+2\alpha_3+\alpha_4,~\alpha_2+2\alpha_3+2\alpha_4,~\alpha_2+\alpha_3+\alpha_4,~\alpha_2$
\item $\alpha_1+\alpha_2+2\alpha_3+\alpha_4,~\alpha_2+2\alpha_3+2\alpha_4,~\alpha_2+\alpha_3$
\item $\alpha_1+\alpha_2+2\alpha_3+\alpha_4,~\alpha_2+2\alpha_3+2\alpha_4,~\alpha_2+\alpha_3,~\alpha_2$
\item $\alpha_1+\alpha_2+2\alpha_3+\alpha_4,~\alpha_2+2\alpha_3+2\alpha_4,~\alpha_2$
\item $\alpha_1+\alpha_2+2\alpha_3+\alpha_4,~\alpha_2+2\alpha_3$
\item $\alpha_1+\alpha_2+2\alpha_3+\alpha_4,~\alpha_2+2\alpha_3,~\alpha_2+\alpha_3+\alpha_4$
\item $\alpha_1+\alpha_2+2\alpha_3+\alpha_4,~\alpha_2+2\alpha_3,~\alpha_2+\alpha_3+\alpha_4,~\alpha_2+\alpha_3$
\item $\alpha_1+\alpha_2+2\alpha_3+\alpha_4,~\alpha_2+2\alpha_3,~\alpha_2+\alpha_3+\alpha_4,~\alpha_2+\alpha_3,~\alpha_2$
\item $\alpha_1+\alpha_2+2\alpha_3+\alpha_4,~\alpha_2+2\alpha_3,~\alpha_2+\alpha_3+\alpha_4,~\alpha_2$
\item $\alpha_1+\alpha_2+2\alpha_3+\alpha_4,~\alpha_2+2\alpha_3,~\alpha_2+\alpha_3$
\item $\alpha_1+\alpha_2+2\alpha_3+\alpha_4,~\alpha_2+2\alpha_3,~\alpha_2+\alpha_3,~\alpha_2$
\item $\alpha_1+\alpha_2+2\alpha_3+\alpha_4,~\alpha_2+2\alpha_3,~\alpha_2$
\item $\alpha_1+\alpha_2+2\alpha_3+\alpha_4,~\alpha_2+\alpha_3+\alpha_4$
\item $\alpha_1+\alpha_2+2\alpha_3+\alpha_4,~\alpha_2+\alpha_3+\alpha_4,~\alpha_2+\alpha_3$
\item $\alpha_1+\alpha_2+2\alpha_3+\alpha_4,~\alpha_2+\alpha_3+\alpha_4,~\alpha_2+\alpha_3,~\alpha_2$
\item $\alpha_1+\alpha_2+2\alpha_3+\alpha_4,~\alpha_2+\alpha_3+\alpha_4,~\alpha_2$
\item $\alpha_1+\alpha_2+2\alpha_3+\alpha_4,~\alpha_2+\alpha_3$
\item $\alpha_1+\alpha_2+2\alpha_3+\alpha_4,~\alpha_2+\alpha_3,~\alpha_2$
\item $\alpha_1+\alpha_2+2\alpha_3+\alpha_4,~\alpha_2$
\item $\alpha_1+\alpha_2+2\alpha_3+2\alpha_4$
\item $\alpha_1+\alpha_2+2\alpha_3+2\alpha_4,~\alpha_1+\alpha_2+2\alpha_3$
\item $\alpha_1+\alpha_2+2\alpha_3+2\alpha_4,~\alpha_1+\alpha_2+2\alpha_3,~\alpha_1+\alpha_2$
\item $\alpha_1+\alpha_2+2\alpha_3+2\alpha_4,~\alpha_1+\alpha_2+2\alpha_3,~\alpha_1+\alpha_2,~\alpha_2+2\alpha_3+\alpha_4$
\item $\alpha_1+\alpha_2+2\alpha_3+2\alpha_4,~\alpha_1+\alpha_2+2\alpha_3,~\alpha_1+\alpha_2,~\alpha_2+2\alpha_3+\alpha_4,~\alpha_2+2\alpha_3$
\item $\alpha_1+\alpha_2+2\alpha_3+2\alpha_4,~\alpha_1+\alpha_2+2\alpha_3,~\alpha_1+\alpha_2,~\alpha_2+2\alpha_3+\alpha_4,~\alpha_2+2\alpha_3,~\alpha_2+\alpha_3+\alpha_4$
\item $\alpha_1+\alpha_2+2\alpha_3+2\alpha_4,~\alpha_1+\alpha_2+2\alpha_3,~\alpha_1+\alpha_2,~\alpha_2+2\alpha_3+\alpha_4,~\alpha_2+2\alpha_3,~\alpha_2+\alpha_3+\alpha_4,~\alpha_2+\alpha_3$
\item $\alpha_1+\alpha_2+2\alpha_3+2\alpha_4,~\alpha_1+\alpha_2+2\alpha_3,~\alpha_1+\alpha_2,~\alpha_2+2\alpha_3+\alpha_4,~\alpha_2+2\alpha_3,~\alpha_2+\alpha_3+\alpha_4,~\alpha_2+\alpha_3,~\alpha_2$
\item $\alpha_1+\alpha_2+2\alpha_3+2\alpha_4,~\alpha_1+\alpha_2+2\alpha_3,~\alpha_1+\alpha_2,~\alpha_2+2\alpha_3+\alpha_4,~\alpha_2+2\alpha_3,~\alpha_2+\alpha_3+\alpha_4,~\alpha_2$
\item $\alpha_1+\alpha_2+2\alpha_3+2\alpha_4,~\alpha_1+\alpha_2+2\alpha_3,~\alpha_1+\alpha_2,~\alpha_2+2\alpha_3+\alpha_4,~\alpha_2+2\alpha_3,~\alpha_2+\alpha_3$
\item $\alpha_1+\alpha_2+2\alpha_3+2\alpha_4,~\alpha_1+\alpha_2+2\alpha_3,~\alpha_1+\alpha_2,~\alpha_2+2\alpha_3+\alpha_4,~\alpha_2+2\alpha_3,~\alpha_2+\alpha_3,~\alpha_2$
\item $\alpha_1+\alpha_2+2\alpha_3+2\alpha_4,~\alpha_1+\alpha_2+2\alpha_3,~\alpha_1+\alpha_2,~\alpha_2+2\alpha_3+\alpha_4,~\alpha_2+2\alpha_3,~\alpha_2$
\item $\alpha_1+\alpha_2+2\alpha_3+2\alpha_4,~\alpha_1+\alpha_2+2\alpha_3,~\alpha_1+\alpha_2,~\alpha_2+2\alpha_3+\alpha_4,~\alpha_2+\alpha_3+\alpha_4$
\item $\alpha_1+\alpha_2+2\alpha_3+2\alpha_4,~\alpha_1+\alpha_2+2\alpha_3,~\alpha_1+\alpha_2,~\alpha_2+2\alpha_3+\alpha_4,~\alpha_2+\alpha_3+\alpha_4,~\alpha_2+\alpha_3$
\item $\alpha_1+\alpha_2+2\alpha_3+2\alpha_4,~\alpha_1+\alpha_2+2\alpha_3,~\alpha_1+\alpha_2,~\alpha_2+2\alpha_3+\alpha_4,~\alpha_2+\alpha_3+\alpha_4,~\alpha_2+\alpha_3,~\alpha_2$
\item $\alpha_1+\alpha_2+2\alpha_3+2\alpha_4,~\alpha_1+\alpha_2+2\alpha_3,~\alpha_1+\alpha_2,~\alpha_2+2\alpha_3+\alpha_4,~\alpha_2+\alpha_3+\alpha_4,~\alpha_2$
\item $\alpha_1+\alpha_2+2\alpha_3+2\alpha_4,~\alpha_1+\alpha_2+2\alpha_3,~\alpha_1+\alpha_2,~\alpha_2+2\alpha_3+\alpha_4,~\alpha_2+\alpha_3$
\item $\alpha_1+\alpha_2+2\alpha_3+2\alpha_4,~\alpha_1+\alpha_2+2\alpha_3,~\alpha_1+\alpha_2,~\alpha_2+2\alpha_3+\alpha_4,~\alpha_2+\alpha_3,~\alpha_2$
\item $\alpha_1+\alpha_2+2\alpha_3+2\alpha_4,~\alpha_1+\alpha_2+2\alpha_3,~\alpha_1+\alpha_2,~\alpha_2+2\alpha_3+\alpha_4,~\alpha_2$
\item $\alpha_1+\alpha_2+2\alpha_3+2\alpha_4,~\alpha_1+\alpha_2+2\alpha_3,~\alpha_1+\alpha_2,~\alpha_2+2\alpha_3$
\item $\alpha_1+\alpha_2+2\alpha_3+2\alpha_4,~\alpha_1+\alpha_2+2\alpha_3,~\alpha_1+\alpha_2,~\alpha_2+2\alpha_3,~\alpha_2+\alpha_3+\alpha_4$
\item $\alpha_1+\alpha_2+2\alpha_3+2\alpha_4,~\alpha_1+\alpha_2+2\alpha_3,~\alpha_1+\alpha_2,~\alpha_2+2\alpha_3,~\alpha_2+\alpha_3+\alpha_4,~\alpha_2+\alpha_3$
\item $\alpha_1+\alpha_2+2\alpha_3+2\alpha_4,~\alpha_1+\alpha_2+2\alpha_3,~\alpha_1+\alpha_2,~\alpha_2+2\alpha_3,~\alpha_2+\alpha_3+\alpha_4,~\alpha_2+\alpha_3,~\alpha_2$
\item $\alpha_1+\alpha_2+2\alpha_3+2\alpha_4,~\alpha_1+\alpha_2+2\alpha_3,~\alpha_1+\alpha_2,~\alpha_2+2\alpha_3,~\alpha_2+\alpha_3+\alpha_4,~\alpha_2$
\item $\alpha_1+\alpha_2+2\alpha_3+2\alpha_4,~\alpha_1+\alpha_2+2\alpha_3,~\alpha_1+\alpha_2,~\alpha_2+2\alpha_3,~\alpha_2+\alpha_3$
\item $\alpha_1+\alpha_2+2\alpha_3+2\alpha_4,~\alpha_1+\alpha_2+2\alpha_3,~\alpha_1+\alpha_2,~\alpha_2+2\alpha_3,~\alpha_2+\alpha_3,~\alpha_2$
\item $\alpha_1+\alpha_2+2\alpha_3+2\alpha_4,~\alpha_1+\alpha_2+2\alpha_3,~\alpha_1+\alpha_2,~\alpha_2+2\alpha_3,~\alpha_2$
\item $\alpha_1+\alpha_2+2\alpha_3+2\alpha_4,~\alpha_1+\alpha_2+2\alpha_3,~\alpha_1+\alpha_2,~\alpha_2+\alpha_3+\alpha_4$
\item $\alpha_1+\alpha_2+2\alpha_3+2\alpha_4,~\alpha_1+\alpha_2+2\alpha_3,~\alpha_1+\alpha_2,~\alpha_2+\alpha_3+\alpha_4,~\alpha_2+\alpha_3$
\item $\alpha_1+\alpha_2+2\alpha_3+2\alpha_4,~\alpha_1+\alpha_2+2\alpha_3,~\alpha_1+\alpha_2,~\alpha_2+\alpha_3+\alpha_4,~\alpha_2+\alpha_3,~\alpha_2$
\item $\alpha_1+\alpha_2+2\alpha_3+2\alpha_4,~\alpha_1+\alpha_2+2\alpha_3,~\alpha_1+\alpha_2,~\alpha_2+\alpha_3+\alpha_4,~\alpha_2$
\item $\alpha_1+\alpha_2+2\alpha_3+2\alpha_4,~\alpha_1+\alpha_2+2\alpha_3,~\alpha_1+\alpha_2,~\alpha_2+\alpha_3$
\item $\alpha_1+\alpha_2+2\alpha_3+2\alpha_4,~\alpha_1+\alpha_2+2\alpha_3,~\alpha_1+\alpha_2,~\alpha_2+\alpha_3,~\alpha_2$
\item $\alpha_1+\alpha_2+2\alpha_3+2\alpha_4,~\alpha_1+\alpha_2+2\alpha_3,~\alpha_1+\alpha_2,~\alpha_2$
\item $\alpha_1+\alpha_2+2\alpha_3+2\alpha_4,~\alpha_1+\alpha_2+2\alpha_3,~\alpha_2+2\alpha_3+\alpha_4$
\item $\alpha_1+\alpha_2+2\alpha_3+2\alpha_4,~\alpha_1+\alpha_2+2\alpha_3,~\alpha_2+2\alpha_3+\alpha_4,~\alpha_2+2\alpha_3$
\item $\alpha_1+\alpha_2+2\alpha_3+2\alpha_4,~\alpha_1+\alpha_2+2\alpha_3,~\alpha_2+2\alpha_3+\alpha_4,~\alpha_2+2\alpha_3,~\alpha_2+\alpha_3+\alpha_4$
\item $\alpha_1+\alpha_2+2\alpha_3+2\alpha_4,~\alpha_1+\alpha_2+2\alpha_3,~\alpha_2+2\alpha_3+\alpha_4,~\alpha_2+2\alpha_3,~\alpha_2+\alpha_3+\alpha_4,~\alpha_2+\alpha_3$
\item $\alpha_1+\alpha_2+2\alpha_3+2\alpha_4,~\alpha_1+\alpha_2+2\alpha_3,~\alpha_2+2\alpha_3+\alpha_4,~\alpha_2+2\alpha_3,~\alpha_2+\alpha_3+\alpha_4,~\alpha_2+\alpha_3,~\alpha_2$
\item $\alpha_1+\alpha_2+2\alpha_3+2\alpha_4,~\alpha_1+\alpha_2+2\alpha_3,~\alpha_2+2\alpha_3+\alpha_4,~\alpha_2+2\alpha_3,~\alpha_2+\alpha_3+\alpha_4,~\alpha_2$
\item $\alpha_1+\alpha_2+2\alpha_3+2\alpha_4,~\alpha_1+\alpha_2+2\alpha_3,~\alpha_2+2\alpha_3+\alpha_4,~\alpha_2+2\alpha_3,~\alpha_2+\alpha_3$
\item $\alpha_1+\alpha_2+2\alpha_3+2\alpha_4,~\alpha_1+\alpha_2+2\alpha_3,~\alpha_2+2\alpha_3+\alpha_4,~\alpha_2+2\alpha_3,~\alpha_2+\alpha_3,~\alpha_2$
\item $\alpha_1+\alpha_2+2\alpha_3+2\alpha_4,~\alpha_1+\alpha_2+2\alpha_3,~\alpha_2+2\alpha_3+\alpha_4,~\alpha_2+2\alpha_3,~\alpha_2$
\item $\alpha_1+\alpha_2+2\alpha_3+2\alpha_4,~\alpha_1+\alpha_2+2\alpha_3,~\alpha_2+2\alpha_3+\alpha_4,~\alpha_2+\alpha_3+\alpha_4$
\item $\alpha_1+\alpha_2+2\alpha_3+2\alpha_4,~\alpha_1+\alpha_2+2\alpha_3,~\alpha_2+2\alpha_3+\alpha_4,~\alpha_2+\alpha_3+\alpha_4,~\alpha_2+\alpha_3$
\item $\alpha_1+\alpha_2+2\alpha_3+2\alpha_4,~\alpha_1+\alpha_2+2\alpha_3,~\alpha_2+2\alpha_3+\alpha_4,~\alpha_2+\alpha_3+\alpha_4,~\alpha_2+\alpha_3,~\alpha_2$
\item $\alpha_1+\alpha_2+2\alpha_3+2\alpha_4,~\alpha_1+\alpha_2+2\alpha_3,~\alpha_2+2\alpha_3+\alpha_4,~\alpha_2+\alpha_3+\alpha_4,~\alpha_2$
\item $\alpha_1+\alpha_2+2\alpha_3+2\alpha_4,~\alpha_1+\alpha_2+2\alpha_3,~\alpha_2+2\alpha_3+\alpha_4,~\alpha_2+\alpha_3$
\item $\alpha_1+\alpha_2+2\alpha_3+2\alpha_4,~\alpha_1+\alpha_2+2\alpha_3,~\alpha_2+2\alpha_3+\alpha_4,~\alpha_2+\alpha_3,~\alpha_2$
\item $\alpha_1+\alpha_2+2\alpha_3+2\alpha_4,~\alpha_1+\alpha_2+2\alpha_3,~\alpha_2+2\alpha_3+\alpha_4,~\alpha_2$
\item $\alpha_1+\alpha_2+2\alpha_3+2\alpha_4,~\alpha_1+\alpha_2+2\alpha_3,~\alpha_2+2\alpha_3$
\item $\alpha_1+\alpha_2+2\alpha_3+2\alpha_4,~\alpha_1+\alpha_2+2\alpha_3,~\alpha_2+2\alpha_3,~\alpha_2+\alpha_3+\alpha_4$
\item $\alpha_1+\alpha_2+2\alpha_3+2\alpha_4,~\alpha_1+\alpha_2+2\alpha_3,~\alpha_2+2\alpha_3,~\alpha_2+\alpha_3+\alpha_4,~\alpha_2+\alpha_3$
\item $\alpha_1+\alpha_2+2\alpha_3+2\alpha_4,~\alpha_1+\alpha_2+2\alpha_3,~\alpha_2+2\alpha_3,~\alpha_2+\alpha_3+\alpha_4,~\alpha_2+\alpha_3,~\alpha_2$
\item $\alpha_1+\alpha_2+2\alpha_3+2\alpha_4,~\alpha_1+\alpha_2+2\alpha_3,~\alpha_2+2\alpha_3,~\alpha_2+\alpha_3+\alpha_4,~\alpha_2$
\item $\alpha_1+\alpha_2+2\alpha_3+2\alpha_4,~\alpha_1+\alpha_2+2\alpha_3,~\alpha_2+2\alpha_3,~\alpha_2+\alpha_3$
\item $\alpha_1+\alpha_2+2\alpha_3+2\alpha_4,~\alpha_1+\alpha_2+2\alpha_3,~\alpha_2+2\alpha_3,~\alpha_2+\alpha_3,~\alpha_2$
\item $\alpha_1+\alpha_2+2\alpha_3+2\alpha_4,~\alpha_1+\alpha_2+2\alpha_3,~\alpha_2+2\alpha_3,~\alpha_2$
\item $\alpha_1+\alpha_2+2\alpha_3+2\alpha_4,~\alpha_1+\alpha_2+2\alpha_3,~\alpha_2+\alpha_3+\alpha_4$
\item $\alpha_1+\alpha_2+2\alpha_3+2\alpha_4,~\alpha_1+\alpha_2+2\alpha_3,~\alpha_2+\alpha_3+\alpha_4,~\alpha_2+\alpha_3$
\item $\alpha_1+\alpha_2+2\alpha_3+2\alpha_4,~\alpha_1+\alpha_2+2\alpha_3,~\alpha_2+\alpha_3+\alpha_4,~\alpha_2+\alpha_3,~\alpha_2$
\item $\alpha_1+\alpha_2+2\alpha_3+2\alpha_4,~\alpha_1+\alpha_2+2\alpha_3,~\alpha_2+\alpha_3+\alpha_4,~\alpha_2$
\item $\alpha_1+\alpha_2+2\alpha_3+2\alpha_4,~\alpha_1+\alpha_2+2\alpha_3,~\alpha_2+\alpha_3$
\item $\alpha_1+\alpha_2+2\alpha_3+2\alpha_4,~\alpha_1+\alpha_2+2\alpha_3,~\alpha_2+\alpha_3,~\alpha_2$
\item $\alpha_1+\alpha_2+2\alpha_3+2\alpha_4,~\alpha_1+\alpha_2+2\alpha_3,~\alpha_2$
\item $\alpha_1+\alpha_2+2\alpha_3+2\alpha_4,~\alpha_1+\alpha_2+\alpha_3$
\item $\alpha_1+\alpha_2+2\alpha_3+2\alpha_4,~\alpha_1+\alpha_2+\alpha_3,~\alpha_2+2\alpha_3+\alpha_4$
\item $\alpha_1+\alpha_2+2\alpha_3+2\alpha_4,~\alpha_1+\alpha_2+\alpha_3,~\alpha_2+2\alpha_3+\alpha_4,~\alpha_2+2\alpha_3$
\item $\alpha_1+\alpha_2+2\alpha_3+2\alpha_4,~\alpha_1+\alpha_2+\alpha_3,~\alpha_2+2\alpha_3+\alpha_4,~\alpha_2+2\alpha_3,~\alpha_2+\alpha_3+\alpha_4$
\item $\alpha_1+\alpha_2+2\alpha_3+2\alpha_4,~\alpha_1+\alpha_2+\alpha_3,~\alpha_2+2\alpha_3+\alpha_4,~\alpha_2+2\alpha_3,~\alpha_2+\alpha_3+\alpha_4,~\alpha_2+\alpha_3$
\item $\alpha_1+\alpha_2+2\alpha_3+2\alpha_4,~\alpha_1+\alpha_2+\alpha_3,~\alpha_2+2\alpha_3+\alpha_4,~\alpha_2+2\alpha_3,~\alpha_2+\alpha_3+\alpha_4,~\alpha_2+\alpha_3,~\alpha_2$
\item $\alpha_1+\alpha_2+2\alpha_3+2\alpha_4,~\alpha_1+\alpha_2+\alpha_3,~\alpha_2+2\alpha_3+\alpha_4,~\alpha_2+2\alpha_3,~\alpha_2+\alpha_3+\alpha_4,~\alpha_2$
\item $\alpha_1+\alpha_2+2\alpha_3+2\alpha_4,~\alpha_1+\alpha_2+\alpha_3,~\alpha_2+2\alpha_3+\alpha_4,~\alpha_2+2\alpha_3,~\alpha_2+\alpha_3$
\item $\alpha_1+\alpha_2+2\alpha_3+2\alpha_4,~\alpha_1+\alpha_2+\alpha_3,~\alpha_2+2\alpha_3+\alpha_4,~\alpha_2+2\alpha_3,~\alpha_2+\alpha_3,~\alpha_2$
\item $\alpha_1+\alpha_2+2\alpha_3+2\alpha_4,~\alpha_1+\alpha_2+\alpha_3,~\alpha_2+2\alpha_3+\alpha_4,~\alpha_2+2\alpha_3,~\alpha_2$
\item $\alpha_1+\alpha_2+2\alpha_3+2\alpha_4,~\alpha_1+\alpha_2+\alpha_3,~\alpha_2+2\alpha_3+\alpha_4,~\alpha_2+\alpha_3+\alpha_4$
\item $\alpha_1+\alpha_2+2\alpha_3+2\alpha_4,~\alpha_1+\alpha_2+\alpha_3,~\alpha_2+2\alpha_3+\alpha_4,~\alpha_2+\alpha_3+\alpha_4,~\alpha_2+\alpha_3$
\item $\alpha_1+\alpha_2+2\alpha_3+2\alpha_4,~\alpha_1+\alpha_2+\alpha_3,~\alpha_2+2\alpha_3+\alpha_4,~\alpha_2+\alpha_3+\alpha_4,~\alpha_2+\alpha_3,~\alpha_2$
\item $\alpha_1+\alpha_2+2\alpha_3+2\alpha_4,~\alpha_1+\alpha_2+\alpha_3,~\alpha_2+2\alpha_3+\alpha_4,~\alpha_2+\alpha_3+\alpha_4,~\alpha_2$
\item $\alpha_1+\alpha_2+2\alpha_3+2\alpha_4,~\alpha_1+\alpha_2+\alpha_3,~\alpha_2+2\alpha_3+\alpha_4,~\alpha_2+\alpha_3$
\item $\alpha_1+\alpha_2+2\alpha_3+2\alpha_4,~\alpha_1+\alpha_2+\alpha_3,~\alpha_2+2\alpha_3+\alpha_4,~\alpha_2+\alpha_3,~\alpha_2$
\item $\alpha_1+\alpha_2+2\alpha_3+2\alpha_4,~\alpha_1+\alpha_2+\alpha_3,~\alpha_2+2\alpha_3+\alpha_4,~\alpha_2$
\item $\alpha_1+\alpha_2+2\alpha_3+2\alpha_4,~\alpha_1+\alpha_2+\alpha_3,~\alpha_2+2\alpha_3$
\item $\alpha_1+\alpha_2+2\alpha_3+2\alpha_4,~\alpha_1+\alpha_2+\alpha_3,~\alpha_2+2\alpha_3,~\alpha_2+\alpha_3+\alpha_4$
\item $\alpha_1+\alpha_2+2\alpha_3+2\alpha_4,~\alpha_1+\alpha_2+\alpha_3,~\alpha_2+2\alpha_3,~\alpha_2+\alpha_3+\alpha_4,~\alpha_2+\alpha_3$
\item $\alpha_1+\alpha_2+2\alpha_3+2\alpha_4,~\alpha_1+\alpha_2+\alpha_3,~\alpha_2+2\alpha_3,~\alpha_2+\alpha_3+\alpha_4,~\alpha_2+\alpha_3,~\alpha_2$
\item $\alpha_1+\alpha_2+2\alpha_3+2\alpha_4,~\alpha_1+\alpha_2+\alpha_3,~\alpha_2+2\alpha_3,~\alpha_2+\alpha_3+\alpha_4,~\alpha_2$
\item $\alpha_1+\alpha_2+2\alpha_3+2\alpha_4,~\alpha_1+\alpha_2+\alpha_3,~\alpha_2+2\alpha_3,~\alpha_2+\alpha_3$
\item $\alpha_1+\alpha_2+2\alpha_3+2\alpha_4,~\alpha_1+\alpha_2+\alpha_3,~\alpha_2+2\alpha_3,~\alpha_2+\alpha_3,~\alpha_2$
\item $\alpha_1+\alpha_2+2\alpha_3+2\alpha_4,~\alpha_1+\alpha_2+\alpha_3,~\alpha_2+2\alpha_3,~\alpha_2$
\item $\alpha_1+\alpha_2+2\alpha_3+2\alpha_4,~\alpha_1+\alpha_2+\alpha_3,~\alpha_2+\alpha_3+\alpha_4$
\item $\alpha_1+\alpha_2+2\alpha_3+2\alpha_4,~\alpha_1+\alpha_2+\alpha_3,~\alpha_2+\alpha_3+\alpha_4,~\alpha_2+\alpha_3$
\item $\alpha_1+\alpha_2+2\alpha_3+2\alpha_4,~\alpha_1+\alpha_2+\alpha_3,~\alpha_2+\alpha_3+\alpha_4,~\alpha_2+\alpha_3,~\alpha_2$
\item $\alpha_1+\alpha_2+2\alpha_3+2\alpha_4,~\alpha_1+\alpha_2+\alpha_3,~\alpha_2+\alpha_3+\alpha_4,~\alpha_2$
\item $\alpha_1+\alpha_2+2\alpha_3+2\alpha_4,~\alpha_1+\alpha_2+\alpha_3,~\alpha_2+\alpha_3$
\item $\alpha_1+\alpha_2+2\alpha_3+2\alpha_4,~\alpha_1+\alpha_2+\alpha_3,~\alpha_2+\alpha_3,~\alpha_2$
\item $\alpha_1+\alpha_2+2\alpha_3+2\alpha_4,~\alpha_1+\alpha_2+\alpha_3,~\alpha_2$
\item $\alpha_1+\alpha_2+2\alpha_3+2\alpha_4,~\alpha_1+\alpha_2$
\item $\alpha_1+\alpha_2+2\alpha_3+2\alpha_4,~\alpha_1+\alpha_2,~\alpha_2+2\alpha_3+\alpha_4$
\item $\alpha_1+\alpha_2+2\alpha_3+2\alpha_4,~\alpha_1+\alpha_2,~\alpha_2+2\alpha_3+\alpha_4,~\alpha_2+2\alpha_3$
\item $\alpha_1+\alpha_2+2\alpha_3+2\alpha_4,~\alpha_1+\alpha_2,~\alpha_2+2\alpha_3+\alpha_4,~\alpha_2+2\alpha_3,~\alpha_2+\alpha_3$
\item $\alpha_1+\alpha_2+2\alpha_3+2\alpha_4,~\alpha_1+\alpha_2,~\alpha_2+2\alpha_3+\alpha_4,~\alpha_2+2\alpha_3,~\alpha_2+\alpha_3,~\alpha_2$
\item $\alpha_1+\alpha_2+2\alpha_3+2\alpha_4,~\alpha_1+\alpha_2,~\alpha_2+2\alpha_3+\alpha_4,~\alpha_2+2\alpha_3,~\alpha_2$
\item $\alpha_1+\alpha_2+2\alpha_3+2\alpha_4,~\alpha_1+\alpha_2,~\alpha_2+2\alpha_3+\alpha_4,~\alpha_2+\alpha_3$
\item $\alpha_1+\alpha_2+2\alpha_3+2\alpha_4,~\alpha_1+\alpha_2,~\alpha_2+2\alpha_3+\alpha_4,~\alpha_2+\alpha_3,~\alpha_2$
\item $\alpha_1+\alpha_2+2\alpha_3+2\alpha_4,~\alpha_1+\alpha_2,~\alpha_2+2\alpha_3+\alpha_4,~\alpha_2$
\item $\alpha_1+\alpha_2+2\alpha_3+2\alpha_4,~\alpha_1+\alpha_2,~\alpha_2+2\alpha_3$
\item $\alpha_1+\alpha_2+2\alpha_3+2\alpha_4,~\alpha_1+\alpha_2,~\alpha_2+2\alpha_3,~\alpha_2+\alpha_3+\alpha_4$
\item $\alpha_1+\alpha_2+2\alpha_3+2\alpha_4,~\alpha_1+\alpha_2,~\alpha_2+2\alpha_3,~\alpha_2+\alpha_3+\alpha_4,~\alpha_2$
\item $\alpha_1+\alpha_2+2\alpha_3+2\alpha_4,~\alpha_1+\alpha_2,~\alpha_2+2\alpha_3,~\alpha_2$
\item $\alpha_1+\alpha_2+2\alpha_3+2\alpha_4,~\alpha_1+\alpha_2,~\alpha_2+\alpha_3+\alpha_4$
\item $\alpha_1+\alpha_2+2\alpha_3+2\alpha_4,~\alpha_1+\alpha_2,~\alpha_2+\alpha_3+\alpha_4,~\alpha_2+\alpha_3$
\item $\alpha_1+\alpha_2+2\alpha_3+2\alpha_4,~\alpha_1+\alpha_2,~\alpha_2+\alpha_3+\alpha_4,~\alpha_2$
\item $\alpha_1+\alpha_2+2\alpha_3+2\alpha_4,~\alpha_1+\alpha_2,~\alpha_2+\alpha_3+\alpha_4,~\alpha_2,~\alpha_3$
\item $\alpha_1+\alpha_2+2\alpha_3+2\alpha_4,~\alpha_1+\alpha_2,~\alpha_2+\alpha_3+\alpha_4,~\alpha_3$
\item $\alpha_1+\alpha_2+2\alpha_3+2\alpha_4,~\alpha_1+\alpha_2,~\alpha_2+\alpha_3$
\item $\alpha_1+\alpha_2+2\alpha_3+2\alpha_4,~\alpha_1+\alpha_2,~\alpha_2$
\item $\alpha_1+\alpha_2+2\alpha_3+2\alpha_4,~\alpha_1+\alpha_2,~\alpha_2,~\alpha_3$
\item $\alpha_1+\alpha_2+2\alpha_3+2\alpha_4,~\alpha_1+\alpha_2,~\alpha_3$
\item $\alpha_1+\alpha_2+2\alpha_3+2\alpha_4,~\alpha_2+2\alpha_3+\alpha_4$
\item $\alpha_1+\alpha_2+2\alpha_3+2\alpha_4,~\alpha_2+2\alpha_3+\alpha_4,~\alpha_2+2\alpha_3$
\item $\alpha_1+\alpha_2+2\alpha_3+2\alpha_4,~\alpha_2+2\alpha_3+\alpha_4,~\alpha_2+2\alpha_3,~\alpha_2+\alpha_3$
\item $\alpha_1+\alpha_2+2\alpha_3+2\alpha_4,~\alpha_2+2\alpha_3+\alpha_4,~\alpha_2+2\alpha_3,~\alpha_2+\alpha_3,~\alpha_2$
\item $\alpha_1+\alpha_2+2\alpha_3+2\alpha_4,~\alpha_2+2\alpha_3+\alpha_4,~\alpha_2+2\alpha_3,~\alpha_2$
\item $\alpha_1+\alpha_2+2\alpha_3+2\alpha_4,~\alpha_2+2\alpha_3+\alpha_4,~\alpha_2+\alpha_3$
\item $\alpha_1+\alpha_2+2\alpha_3+2\alpha_4,~\alpha_2+2\alpha_3+\alpha_4,~\alpha_2+\alpha_3,~\alpha_2$
\item $\alpha_1+\alpha_2+2\alpha_3+2\alpha_4,~\alpha_2+2\alpha_3+\alpha_4,~\alpha_2$
\item $\alpha_1+\alpha_2+2\alpha_3+2\alpha_4,~\alpha_2+2\alpha_3$
\item $\alpha_1+\alpha_2+2\alpha_3+2\alpha_4,~\alpha_2+2\alpha_3,~\alpha_2+\alpha_3+\alpha_4$
\item $\alpha_1+\alpha_2+2\alpha_3+2\alpha_4,~\alpha_2+2\alpha_3,~\alpha_2+\alpha_3+\alpha_4,~\alpha_2$
\item $\alpha_1+\alpha_2+2\alpha_3+2\alpha_4,~\alpha_2+2\alpha_3,~\alpha_2$
\item $\alpha_1+\alpha_2+2\alpha_3+2\alpha_4,~\alpha_2+\alpha_3+\alpha_4$
\item $\alpha_1+\alpha_2+2\alpha_3+2\alpha_4,~\alpha_2+\alpha_3+\alpha_4,~\alpha_2+\alpha_3$
\item $\alpha_1+\alpha_2+2\alpha_3+2\alpha_4,~\alpha_2+\alpha_3+\alpha_4,~\alpha_2$
\item $\alpha_1+\alpha_2+2\alpha_3+2\alpha_4,~\alpha_2+\alpha_3+\alpha_4,~\alpha_2,~\alpha_3$
\item $\alpha_1+\alpha_2+2\alpha_3+2\alpha_4,~\alpha_2+\alpha_3+\alpha_4,~\alpha_3$
\item $\alpha_1+\alpha_2+2\alpha_3+2\alpha_4,~\alpha_2+\alpha_3$
\item $\alpha_1+\alpha_2+2\alpha_3+2\alpha_4,~\alpha_2$
\item $\alpha_1+\alpha_2+2\alpha_3+2\alpha_4,~\alpha_2,~\alpha_3$
\item $\alpha_1+\alpha_2+2\alpha_3+2\alpha_4,~\alpha_3$
\item $\alpha_1+2\alpha_2+2\alpha_3$
\item $\alpha_1+2\alpha_2+2\alpha_3,~\alpha_1+\alpha_2+2\alpha_3+2\alpha_4$
\item $\alpha_1+2\alpha_2+2\alpha_3,~\alpha_1+\alpha_2+2\alpha_3+2\alpha_4,~\alpha_2+2\alpha_3+\alpha_4$
\item $\alpha_1+2\alpha_2+2\alpha_3,~\alpha_1+\alpha_2+2\alpha_3+2\alpha_4,~\alpha_2+\alpha_3+\alpha_4$
\item $\alpha_1+2\alpha_2+2\alpha_3,~\alpha_1+\alpha_2+2\alpha_3+2\alpha_4,~\alpha_2+\alpha_3+\alpha_4,~\alpha_3$
\item $\alpha_1+2\alpha_2+2\alpha_3,~\alpha_1+\alpha_2+2\alpha_3+2\alpha_4,~\alpha_3$
\item $\alpha_1+2\alpha_2+2\alpha_3,~\alpha_1+\alpha_2+2\alpha_3+\alpha_4$
\item $\alpha_1+2\alpha_2+2\alpha_3,~\alpha_1+\alpha_2+2\alpha_3+\alpha_4,~\alpha_2+2\alpha_3+2\alpha_4$
\item $\alpha_1+2\alpha_2+2\alpha_3,~\alpha_1+\alpha_2+2\alpha_3+\alpha_4,~\alpha_2+2\alpha_3+2\alpha_4,~\alpha_4$
\item $\alpha_1+2\alpha_2+2\alpha_3,~\alpha_1+\alpha_2+2\alpha_3+\alpha_4,~\alpha_2+\alpha_3+\alpha_4$
\item $\alpha_1+2\alpha_2+2\alpha_3,~\alpha_1+\alpha_2+2\alpha_3+\alpha_4,~\alpha_3+\alpha_4$
\item $\alpha_1+2\alpha_2+2\alpha_3,~\alpha_1+\alpha_2+2\alpha_3+\alpha_4,~\alpha_3+\alpha_4,~\alpha_4$
\item $\alpha_1+2\alpha_2+2\alpha_3,~\alpha_1+\alpha_2+2\alpha_3+\alpha_4,~\alpha_4$
\item $\alpha_1+2\alpha_2+2\alpha_3,~\alpha_1+\alpha_2+\alpha_3+\alpha_4$
\item $\alpha_1+2\alpha_2+2\alpha_3,~\alpha_1+\alpha_2+\alpha_3+\alpha_4,~\alpha_2+2\alpha_3+2\alpha_4$
\item $\alpha_1+2\alpha_2+2\alpha_3,~\alpha_1+\alpha_2+\alpha_3+\alpha_4,~\alpha_2+2\alpha_3+2\alpha_4,~\alpha_3+\alpha_4$
\item $\alpha_1+2\alpha_2+2\alpha_3,~\alpha_1+\alpha_2+\alpha_3+\alpha_4,~\alpha_2+2\alpha_3+2\alpha_4,~\alpha_3+\alpha_4,~\alpha_3$
\item $\alpha_1+2\alpha_2+2\alpha_3,~\alpha_1+\alpha_2+\alpha_3+\alpha_4,~\alpha_2+2\alpha_3+2\alpha_4,~\alpha_3$
\item $\alpha_1+2\alpha_2+2\alpha_3,~\alpha_1+\alpha_2+\alpha_3+\alpha_4,~\alpha_2+2\alpha_3+\alpha_4$
\item $\alpha_1+2\alpha_2+2\alpha_3,~\alpha_1+\alpha_2+\alpha_3+\alpha_4,~\alpha_2+2\alpha_3+\alpha_4,~\alpha_3$
\item $\alpha_1+2\alpha_2+2\alpha_3,~\alpha_1+\alpha_2+\alpha_3+\alpha_4,~\alpha_3+\alpha_4$
\item $\alpha_1+2\alpha_2+2\alpha_3,~\alpha_1+\alpha_2+\alpha_3+\alpha_4,~\alpha_3$
\item $\alpha_1+2\alpha_2+2\alpha_3,~\alpha_1+\alpha_2+\alpha_3+\alpha_4,~\alpha_3,~\alpha_4$
\item $\alpha_1+2\alpha_2+2\alpha_3,~\alpha_1+\alpha_2+\alpha_3+\alpha_4,~\alpha_4$
\item $\alpha_1+2\alpha_2+2\alpha_3,~\alpha_1$
\item $\alpha_1+2\alpha_2+2\alpha_3,~\alpha_1,~\alpha_2+2\alpha_3+2\alpha_4$
\item $\alpha_1+2\alpha_2+2\alpha_3,~\alpha_1,~\alpha_2+2\alpha_3+2\alpha_4,~\alpha_3$
\item $\alpha_1+2\alpha_2+2\alpha_3,~\alpha_1,~\alpha_2+2\alpha_3+\alpha_4$
\item $\alpha_1+2\alpha_2+2\alpha_3,~\alpha_1,~\alpha_2+2\alpha_3+\alpha_4,~\alpha_3+\alpha_4$
\item $\alpha_1+2\alpha_2+2\alpha_3,~\alpha_1,~\alpha_2+2\alpha_3+\alpha_4,~\alpha_3+\alpha_4,~\alpha_4$
\item $\alpha_1+2\alpha_2+2\alpha_3,~\alpha_1,~\alpha_2+2\alpha_3+\alpha_4,~\alpha_4$
\item $\alpha_1+2\alpha_2+2\alpha_3,~\alpha_1,~\alpha_2+\alpha_3+\alpha_4$
\item $\alpha_1+2\alpha_2+2\alpha_3,~\alpha_1,~\alpha_2+\alpha_3+\alpha_4,~\alpha_3+\alpha_4$
\item $\alpha_1+2\alpha_2+2\alpha_3,~\alpha_1,~\alpha_2+\alpha_3+\alpha_4,~\alpha_3$
\item $\alpha_1+2\alpha_2+2\alpha_3,~\alpha_1,~\alpha_2+\alpha_3+\alpha_4,~\alpha_3,~\alpha_4$
\item $\alpha_1+2\alpha_2+2\alpha_3,~\alpha_1,~\alpha_2+\alpha_3+\alpha_4,~\alpha_4$
\item $\alpha_1+2\alpha_2+2\alpha_3,~\alpha_1,~\alpha_3+\alpha_4$
\item $\alpha_1+2\alpha_2+2\alpha_3,~\alpha_1,~\alpha_3$
\item $\alpha_1+2\alpha_2+2\alpha_3,~\alpha_1,~\alpha_3,~\alpha_4$
\item $\alpha_1+2\alpha_2+2\alpha_3,~\alpha_1,~\alpha_4$
\item $\alpha_1+2\alpha_2+2\alpha_3,~\alpha_2+2\alpha_3+2\alpha_4$
\item $\alpha_1+2\alpha_2+2\alpha_3,~\alpha_2+2\alpha_3+2\alpha_4,~\alpha_3$
\item $\alpha_1+2\alpha_2+2\alpha_3,~\alpha_2+2\alpha_3+\alpha_4$
\item $\alpha_1+2\alpha_2+2\alpha_3,~\alpha_2+2\alpha_3+\alpha_4,~\alpha_3+\alpha_4$
\item $\alpha_1+2\alpha_2+2\alpha_3,~\alpha_2+2\alpha_3+\alpha_4,~\alpha_3+\alpha_4,~\alpha_4$
\item $\alpha_1+2\alpha_2+2\alpha_3,~\alpha_2+2\alpha_3+\alpha_4,~\alpha_4$
\item $\alpha_1+2\alpha_2+2\alpha_3,~\alpha_2+\alpha_3+\alpha_4$
\item $\alpha_1+2\alpha_2+2\alpha_3,~\alpha_2+\alpha_3+\alpha_4,~\alpha_3+\alpha_4$
\item $\alpha_1+2\alpha_2+2\alpha_3,~\alpha_2+\alpha_3+\alpha_4,~\alpha_3$
\item $\alpha_1+2\alpha_2+2\alpha_3,~\alpha_2+\alpha_3+\alpha_4,~\alpha_3,~\alpha_4$
\item $\alpha_1+2\alpha_2+2\alpha_3,~\alpha_2+\alpha_3+\alpha_4,~\alpha_4$
\item $\alpha_1+2\alpha_2+2\alpha_3,~\alpha_3+\alpha_4$
\item $\alpha_1+2\alpha_2+2\alpha_3,~\alpha_3$
\item $\alpha_1+2\alpha_2+2\alpha_3,~\alpha_3,~\alpha_4$
\item $\alpha_1+2\alpha_2+2\alpha_3,~\alpha_4$
\item $\alpha_1+2\alpha_2+2\alpha_3+\alpha_4$
\item $\alpha_1+2\alpha_2+2\alpha_3+\alpha_4,~\alpha_1+\alpha_2+2\alpha_3+2\alpha_4$
\item $\alpha_1+2\alpha_2+2\alpha_3+\alpha_4,~\alpha_1+\alpha_2+2\alpha_3+2\alpha_4,~\alpha_1+\alpha_2+2\alpha_3$
\item $\alpha_1+2\alpha_2+2\alpha_3+\alpha_4,~\alpha_1+\alpha_2+2\alpha_3+2\alpha_4,~\alpha_1+\alpha_2+2\alpha_3,~\alpha_2+2\alpha_3$
\item $\alpha_1+2\alpha_2+2\alpha_3+\alpha_4,~\alpha_1+\alpha_2+2\alpha_3+2\alpha_4,~\alpha_1+\alpha_2+2\alpha_3,~\alpha_3+\alpha_4$
\item $\alpha_1+2\alpha_2+2\alpha_3+\alpha_4,~\alpha_1+\alpha_2+2\alpha_3+2\alpha_4,~\alpha_1+\alpha_2+2\alpha_3,~\alpha_3+\alpha_4,~\alpha_3$
\item $\alpha_1+2\alpha_2+2\alpha_3+\alpha_4,~\alpha_1+\alpha_2+2\alpha_3+2\alpha_4,~\alpha_1+\alpha_2+2\alpha_3,~\alpha_3$
\item $\alpha_1+2\alpha_2+2\alpha_3+\alpha_4,~\alpha_1+\alpha_2+2\alpha_3+2\alpha_4,~\alpha_2+2\alpha_3$
\item $\alpha_1+2\alpha_2+2\alpha_3+\alpha_4,~\alpha_1+\alpha_2+2\alpha_3+2\alpha_4,~\alpha_3+\alpha_4$
\item $\alpha_1+2\alpha_2+2\alpha_3+\alpha_4,~\alpha_1+\alpha_2+2\alpha_3+2\alpha_4,~\alpha_3+\alpha_4,~\alpha_3$
\item $\alpha_1+2\alpha_2+2\alpha_3+\alpha_4,~\alpha_1+\alpha_2+2\alpha_3+2\alpha_4,~\alpha_3$
\item $\alpha_1+2\alpha_2+2\alpha_3+\alpha_4,~\alpha_1+\alpha_2+2\alpha_3$
\item $\alpha_1+2\alpha_2+2\alpha_3+\alpha_4,~\alpha_1+\alpha_2+2\alpha_3,~\alpha_2+2\alpha_3+2\alpha_4$
\item $\alpha_1+2\alpha_2+2\alpha_3+\alpha_4,~\alpha_1+\alpha_2+2\alpha_3,~\alpha_3+\alpha_4$
\item $\alpha_1+2\alpha_2+2\alpha_3+\alpha_4,~\alpha_1+\alpha_2+2\alpha_3,~\alpha_3+\alpha_4,~\alpha_3$
\item $\alpha_1+2\alpha_2+2\alpha_3+\alpha_4,~\alpha_1+\alpha_2+2\alpha_3,~\alpha_3$
\item $\alpha_1+2\alpha_2+2\alpha_3+\alpha_4,~\alpha_1$
\item $\alpha_1+2\alpha_2+2\alpha_3+\alpha_4,~\alpha_1,~\alpha_2+2\alpha_3+2\alpha_4$
\item $\alpha_1+2\alpha_2+2\alpha_3+\alpha_4,~\alpha_1,~\alpha_2+2\alpha_3+2\alpha_4,~\alpha_2+2\alpha_3$
\item $\alpha_1+2\alpha_2+2\alpha_3+\alpha_4,~\alpha_1,~\alpha_2+2\alpha_3+2\alpha_4,~\alpha_2+2\alpha_3,~\alpha_3+\alpha_4$
\item $\alpha_1+2\alpha_2+2\alpha_3+\alpha_4,~\alpha_1,~\alpha_2+2\alpha_3+2\alpha_4,~\alpha_2+2\alpha_3,~\alpha_3+\alpha_4,~\alpha_3$
\item $\alpha_1+2\alpha_2+2\alpha_3+\alpha_4,~\alpha_1,~\alpha_2+2\alpha_3+2\alpha_4,~\alpha_2+2\alpha_3,~\alpha_3$
\item $\alpha_1+2\alpha_2+2\alpha_3+\alpha_4,~\alpha_1,~\alpha_2+2\alpha_3+2\alpha_4,~\alpha_3+\alpha_4$
\item $\alpha_1+2\alpha_2+2\alpha_3+\alpha_4,~\alpha_1,~\alpha_2+2\alpha_3+2\alpha_4,~\alpha_3+\alpha_4,~\alpha_3$
\item $\alpha_1+2\alpha_2+2\alpha_3+\alpha_4,~\alpha_1,~\alpha_2+2\alpha_3+2\alpha_4,~\alpha_3$
\item $\alpha_1+2\alpha_2+2\alpha_3+\alpha_4,~\alpha_1,~\alpha_2+2\alpha_3$
\item $\alpha_1+2\alpha_2+2\alpha_3+\alpha_4,~\alpha_1,~\alpha_2+2\alpha_3,~\alpha_3+\alpha_4$
\item $\alpha_1+2\alpha_2+2\alpha_3+\alpha_4,~\alpha_1,~\alpha_2+2\alpha_3,~\alpha_3+\alpha_4,~\alpha_3$
\item $\alpha_1+2\alpha_2+2\alpha_3+\alpha_4,~\alpha_1,~\alpha_2+2\alpha_3,~\alpha_3$
\item $\alpha_1+2\alpha_2+2\alpha_3+\alpha_4,~\alpha_1,~\alpha_3+\alpha_4$
\item $\alpha_1+2\alpha_2+2\alpha_3+\alpha_4,~\alpha_1,~\alpha_3+\alpha_4,~\alpha_3$
\item $\alpha_1+2\alpha_2+2\alpha_3+\alpha_4,~\alpha_1,~\alpha_3$
\item $\alpha_1+2\alpha_2+2\alpha_3+\alpha_4,~\alpha_2+2\alpha_3+2\alpha_4$
\item $\alpha_1+2\alpha_2+2\alpha_3+\alpha_4,~\alpha_2+2\alpha_3+2\alpha_4,~\alpha_2+2\alpha_3$
\item $\alpha_1+2\alpha_2+2\alpha_3+\alpha_4,~\alpha_2+2\alpha_3+2\alpha_4,~\alpha_2+2\alpha_3,~\alpha_3+\alpha_4$
\item $\alpha_1+2\alpha_2+2\alpha_3+\alpha_4,~\alpha_2+2\alpha_3+2\alpha_4,~\alpha_2+2\alpha_3,~\alpha_3+\alpha_4,~\alpha_3$
\item $\alpha_1+2\alpha_2+2\alpha_3+\alpha_4,~\alpha_2+2\alpha_3+2\alpha_4,~\alpha_2+2\alpha_3,~\alpha_3$
\item $\alpha_1+2\alpha_2+2\alpha_3+\alpha_4,~\alpha_2+2\alpha_3+2\alpha_4,~\alpha_3+\alpha_4$
\item $\alpha_1+2\alpha_2+2\alpha_3+\alpha_4,~\alpha_2+2\alpha_3+2\alpha_4,~\alpha_3+\alpha_4,~\alpha_3$
\item $\alpha_1+2\alpha_2+2\alpha_3+\alpha_4,~\alpha_2+2\alpha_3+2\alpha_4,~\alpha_3$
\item $\alpha_1+2\alpha_2+2\alpha_3+\alpha_4,~\alpha_2+2\alpha_3$
\item $\alpha_1+2\alpha_2+2\alpha_3+\alpha_4,~\alpha_2+2\alpha_3,~\alpha_3+\alpha_4$
\item $\alpha_1+2\alpha_2+2\alpha_3+\alpha_4,~\alpha_2+2\alpha_3,~\alpha_3+\alpha_4,~\alpha_3$
\item $\alpha_1+2\alpha_2+2\alpha_3+\alpha_4,~\alpha_2+2\alpha_3,~\alpha_3$
\item $\alpha_1+2\alpha_2+2\alpha_3+\alpha_4,~\alpha_3+\alpha_4$
\item $\alpha_1+2\alpha_2+2\alpha_3+\alpha_4,~\alpha_3+\alpha_4,~\alpha_3$
\item $\alpha_1+2\alpha_2+2\alpha_3+\alpha_4,~\alpha_3$
\item $\alpha_1+2\alpha_2+2\alpha_3+2\alpha_4$
\item $\alpha_1+2\alpha_2+2\alpha_3+2\alpha_4,~\alpha_1+2\alpha_2+2\alpha_3$
\item $\alpha_1+2\alpha_2+2\alpha_3+2\alpha_4,~\alpha_1+2\alpha_2+2\alpha_3,~\alpha_1+\alpha_2+2\alpha_3+\alpha_4$
\item $\alpha_1+2\alpha_2+2\alpha_3+2\alpha_4,~\alpha_1+2\alpha_2+2\alpha_3,~\alpha_1+\alpha_2+2\alpha_3+\alpha_4,~\alpha_1+\alpha_2+2\alpha_3$
\item $\alpha_1+2\alpha_2+2\alpha_3+2\alpha_4,~\alpha_1+2\alpha_2+2\alpha_3,~\alpha_1+\alpha_2+2\alpha_3+\alpha_4,~\alpha_1+\alpha_2+2\alpha_3,~\alpha_2+2\alpha_3$
\item $\alpha_1+2\alpha_2+2\alpha_3+2\alpha_4,~\alpha_1+2\alpha_2+2\alpha_3,~\alpha_1+\alpha_2+2\alpha_3+\alpha_4,~\alpha_1+\alpha_2+2\alpha_3,~\alpha_3+\alpha_4$
\item $\alpha_1+2\alpha_2+2\alpha_3+2\alpha_4,~\alpha_1+2\alpha_2+2\alpha_3,~\alpha_1+\alpha_2+2\alpha_3+\alpha_4,~\alpha_1+\alpha_2+2\alpha_3,~\alpha_3+\alpha_4,~\alpha_3$
\item $\alpha_1+2\alpha_2+2\alpha_3+2\alpha_4,~\alpha_1+2\alpha_2+2\alpha_3,~\alpha_1+\alpha_2+2\alpha_3+\alpha_4,~\alpha_1+\alpha_2+2\alpha_3,~\alpha_3$
\item $\alpha_1+2\alpha_2+2\alpha_3+2\alpha_4,~\alpha_1+2\alpha_2+2\alpha_3,~\alpha_1+\alpha_2+2\alpha_3+\alpha_4,~\alpha_2+2\alpha_3$
\item $\alpha_1+2\alpha_2+2\alpha_3+2\alpha_4,~\alpha_1+2\alpha_2+2\alpha_3,~\alpha_1+\alpha_2+2\alpha_3+\alpha_4,~\alpha_3+\alpha_4$
\item $\alpha_1+2\alpha_2+2\alpha_3+2\alpha_4,~\alpha_1+2\alpha_2+2\alpha_3,~\alpha_1+\alpha_2+2\alpha_3+\alpha_4,~\alpha_3+\alpha_4,~\alpha_3$
\item $\alpha_1+2\alpha_2+2\alpha_3+2\alpha_4,~\alpha_1+2\alpha_2+2\alpha_3,~\alpha_1+\alpha_2+2\alpha_3+\alpha_4,~\alpha_3$
\item $\alpha_1+2\alpha_2+2\alpha_3+2\alpha_4,~\alpha_1+2\alpha_2+2\alpha_3,~\alpha_1+\alpha_2+2\alpha_3$
\item $\alpha_1+2\alpha_2+2\alpha_3+2\alpha_4,~\alpha_1+2\alpha_2+2\alpha_3,~\alpha_1+\alpha_2+2\alpha_3,~\alpha_2+2\alpha_3+\alpha_4$
\item $\alpha_1+2\alpha_2+2\alpha_3+2\alpha_4,~\alpha_1+2\alpha_2+2\alpha_3,~\alpha_1+\alpha_2+2\alpha_3,~\alpha_3+\alpha_4$
\item $\alpha_1+2\alpha_2+2\alpha_3+2\alpha_4,~\alpha_1+2\alpha_2+2\alpha_3,~\alpha_1+\alpha_2+2\alpha_3,~\alpha_3+\alpha_4,~\alpha_3$
\item $\alpha_1+2\alpha_2+2\alpha_3+2\alpha_4,~\alpha_1+2\alpha_2+2\alpha_3,~\alpha_1+\alpha_2+2\alpha_3,~\alpha_3$
\item $\alpha_1+2\alpha_2+2\alpha_3+2\alpha_4,~\alpha_1+2\alpha_2+2\alpha_3,~\alpha_1$
\item $\alpha_1+2\alpha_2+2\alpha_3+2\alpha_4,~\alpha_1+2\alpha_2+2\alpha_3,~\alpha_1,~\alpha_2+2\alpha_3+\alpha_4$
\item $\alpha_1+2\alpha_2+2\alpha_3+2\alpha_4,~\alpha_1+2\alpha_2+2\alpha_3,~\alpha_1,~\alpha_2+2\alpha_3+\alpha_4,~\alpha_2+2\alpha_3$
\item $\alpha_1+2\alpha_2+2\alpha_3+2\alpha_4,~\alpha_1+2\alpha_2+2\alpha_3,~\alpha_1,~\alpha_2+2\alpha_3+\alpha_4,~\alpha_2+2\alpha_3,~\alpha_3+\alpha_4$
\item $\alpha_1+2\alpha_2+2\alpha_3+2\alpha_4,~\alpha_1+2\alpha_2+2\alpha_3,~\alpha_1,~\alpha_2+2\alpha_3+\alpha_4,~\alpha_2+2\alpha_3,~\alpha_3+\alpha_4,~\alpha_3$
\item $\alpha_1+2\alpha_2+2\alpha_3+2\alpha_4,~\alpha_1+2\alpha_2+2\alpha_3,~\alpha_1,~\alpha_2+2\alpha_3+\alpha_4,~\alpha_2+2\alpha_3,~\alpha_3$
\item $\alpha_1+2\alpha_2+2\alpha_3+2\alpha_4,~\alpha_1+2\alpha_2+2\alpha_3,~\alpha_1,~\alpha_2+2\alpha_3+\alpha_4,~\alpha_3+\alpha_4$
\item $\alpha_1+2\alpha_2+2\alpha_3+2\alpha_4,~\alpha_1+2\alpha_2+2\alpha_3,~\alpha_1,~\alpha_2+2\alpha_3+\alpha_4,~\alpha_3+\alpha_4,~\alpha_3$
\item $\alpha_1+2\alpha_2+2\alpha_3+2\alpha_4,~\alpha_1+2\alpha_2+2\alpha_3,~\alpha_1,~\alpha_2+2\alpha_3+\alpha_4,~\alpha_3$
\item $\alpha_1+2\alpha_2+2\alpha_3+2\alpha_4,~\alpha_1+2\alpha_2+2\alpha_3,~\alpha_1,~\alpha_2+2\alpha_3$
\item $\alpha_1+2\alpha_2+2\alpha_3+2\alpha_4,~\alpha_1+2\alpha_2+2\alpha_3,~\alpha_1,~\alpha_2+2\alpha_3,~\alpha_3+\alpha_4$
\item $\alpha_1+2\alpha_2+2\alpha_3+2\alpha_4,~\alpha_1+2\alpha_2+2\alpha_3,~\alpha_1,~\alpha_2+2\alpha_3,~\alpha_3+\alpha_4,~\alpha_3$
\item $\alpha_1+2\alpha_2+2\alpha_3+2\alpha_4,~\alpha_1+2\alpha_2+2\alpha_3,~\alpha_1,~\alpha_2+2\alpha_3,~\alpha_3$
\item $\alpha_1+2\alpha_2+2\alpha_3+2\alpha_4,~\alpha_1+2\alpha_2+2\alpha_3,~\alpha_1,~\alpha_3+\alpha_4$
\item $\alpha_1+2\alpha_2+2\alpha_3+2\alpha_4,~\alpha_1+2\alpha_2+2\alpha_3,~\alpha_1,~\alpha_3+\alpha_4,~\alpha_3$
\item $\alpha_1+2\alpha_2+2\alpha_3+2\alpha_4,~\alpha_1+2\alpha_2+2\alpha_3,~\alpha_1,~\alpha_3$
\item $\alpha_1+2\alpha_2+2\alpha_3+2\alpha_4,~\alpha_1+2\alpha_2+2\alpha_3,~\alpha_2+2\alpha_3+\alpha_4$
\item $\alpha_1+2\alpha_2+2\alpha_3+2\alpha_4,~\alpha_1+2\alpha_2+2\alpha_3,~\alpha_2+2\alpha_3+\alpha_4,~\alpha_2+2\alpha_3$
\item $\alpha_1+2\alpha_2+2\alpha_3+2\alpha_4,~\alpha_1+2\alpha_2+2\alpha_3,~\alpha_2+2\alpha_3+\alpha_4,~\alpha_2+2\alpha_3,~\alpha_3+\alpha_4$
\item $\alpha_1+2\alpha_2+2\alpha_3+2\alpha_4,~\alpha_1+2\alpha_2+2\alpha_3,~\alpha_2+2\alpha_3+\alpha_4,~\alpha_2+2\alpha_3,~\alpha_3+\alpha_4,~\alpha_3$
\item $\alpha_1+2\alpha_2+2\alpha_3+2\alpha_4,~\alpha_1+2\alpha_2+2\alpha_3,~\alpha_2+2\alpha_3+\alpha_4,~\alpha_2+2\alpha_3,~\alpha_3$
\item $\alpha_1+2\alpha_2+2\alpha_3+2\alpha_4,~\alpha_1+2\alpha_2+2\alpha_3,~\alpha_2+2\alpha_3+\alpha_4,~\alpha_3+\alpha_4$
\item $\alpha_1+2\alpha_2+2\alpha_3+2\alpha_4,~\alpha_1+2\alpha_2+2\alpha_3,~\alpha_2+2\alpha_3+\alpha_4,~\alpha_3+\alpha_4,~\alpha_3$
\item $\alpha_1+2\alpha_2+2\alpha_3+2\alpha_4,~\alpha_1+2\alpha_2+2\alpha_3,~\alpha_2+2\alpha_3+\alpha_4,~\alpha_3$
\item $\alpha_1+2\alpha_2+2\alpha_3+2\alpha_4,~\alpha_1+2\alpha_2+2\alpha_3,~\alpha_2+2\alpha_3$
\item $\alpha_1+2\alpha_2+2\alpha_3+2\alpha_4,~\alpha_1+2\alpha_2+2\alpha_3,~\alpha_2+2\alpha_3,~\alpha_3+\alpha_4$
\item $\alpha_1+2\alpha_2+2\alpha_3+2\alpha_4,~\alpha_1+2\alpha_2+2\alpha_3,~\alpha_2+2\alpha_3,~\alpha_3+\alpha_4,~\alpha_3$
\item $\alpha_1+2\alpha_2+2\alpha_3+2\alpha_4,~\alpha_1+2\alpha_2+2\alpha_3,~\alpha_2+2\alpha_3,~\alpha_3$
\item $\alpha_1+2\alpha_2+2\alpha_3+2\alpha_4,~\alpha_1+2\alpha_2+2\alpha_3,~\alpha_3+\alpha_4$
\item $\alpha_1+2\alpha_2+2\alpha_3+2\alpha_4,~\alpha_1+2\alpha_2+2\alpha_3,~\alpha_3+\alpha_4,~\alpha_3$
\item $\alpha_1+2\alpha_2+2\alpha_3+2\alpha_4,~\alpha_1+2\alpha_2+2\alpha_3,~\alpha_3$
\item $\alpha_1+2\alpha_2+2\alpha_3+2\alpha_4,~\alpha_1+\alpha_2+2\alpha_3+\alpha_4$
\item $\alpha_1+2\alpha_2+2\alpha_3+2\alpha_4,~\alpha_1+\alpha_2+2\alpha_3+\alpha_4,~\alpha_1+\alpha_2+2\alpha_3$
\item $\alpha_1+2\alpha_2+2\alpha_3+2\alpha_4,~\alpha_1+\alpha_2+2\alpha_3+\alpha_4,~\alpha_1+\alpha_2+2\alpha_3,~\alpha_2+2\alpha_3$
\item $\alpha_1+2\alpha_2+2\alpha_3+2\alpha_4,~\alpha_1+\alpha_2+2\alpha_3+\alpha_4,~\alpha_1+\alpha_2+2\alpha_3,~\alpha_2+\alpha_3$
\item $\alpha_1+2\alpha_2+2\alpha_3+2\alpha_4,~\alpha_1+\alpha_2+2\alpha_3+\alpha_4,~\alpha_1+\alpha_2+2\alpha_3,~\alpha_2+\alpha_3,~\alpha_3$
\item $\alpha_1+2\alpha_2+2\alpha_3+2\alpha_4,~\alpha_1+\alpha_2+2\alpha_3+\alpha_4,~\alpha_1+\alpha_2+2\alpha_3,~\alpha_3$
\item $\alpha_1+2\alpha_2+2\alpha_3+2\alpha_4,~\alpha_1+\alpha_2+2\alpha_3+\alpha_4,~\alpha_2+2\alpha_3$
\item $\alpha_1+2\alpha_2+2\alpha_3+2\alpha_4,~\alpha_1+\alpha_2+2\alpha_3+\alpha_4,~\alpha_2+\alpha_3$
\item $\alpha_1+2\alpha_2+2\alpha_3+2\alpha_4,~\alpha_1+\alpha_2+2\alpha_3+\alpha_4,~\alpha_2+\alpha_3,~\alpha_3$
\item $\alpha_1+2\alpha_2+2\alpha_3+2\alpha_4,~\alpha_1+\alpha_2+2\alpha_3+\alpha_4,~\alpha_3$
\item $\alpha_1+2\alpha_2+2\alpha_3+2\alpha_4,~\alpha_1+\alpha_2+2\alpha_3$
\item $\alpha_1+2\alpha_2+2\alpha_3+2\alpha_4,~\alpha_1+\alpha_2+2\alpha_3,~\alpha_2+2\alpha_3+\alpha_4$
\item $\alpha_1+2\alpha_2+2\alpha_3+2\alpha_4,~\alpha_1+\alpha_2+2\alpha_3,~\alpha_2+\alpha_3$
\item $\alpha_1+2\alpha_2+2\alpha_3+2\alpha_4,~\alpha_1+\alpha_2+2\alpha_3,~\alpha_2+\alpha_3,~\alpha_3+\alpha_4$
\item $\alpha_1+2\alpha_2+2\alpha_3+2\alpha_4,~\alpha_1+\alpha_2+2\alpha_3,~\alpha_3+\alpha_4$
\item $\alpha_1+2\alpha_2+2\alpha_3+2\alpha_4,~\alpha_1+\alpha_2+\alpha_3$
\item $\alpha_1+2\alpha_2+2\alpha_3+2\alpha_4,~\alpha_1+\alpha_2+\alpha_3,~\alpha_2+2\alpha_3+\alpha_4$
\item $\alpha_1+2\alpha_2+2\alpha_3+2\alpha_4,~\alpha_1+\alpha_2+\alpha_3,~\alpha_2+2\alpha_3+\alpha_4,~\alpha_2+2\alpha_3$
\item $\alpha_1+2\alpha_2+2\alpha_3+2\alpha_4,~\alpha_1+\alpha_2+\alpha_3,~\alpha_2+2\alpha_3+\alpha_4,~\alpha_2+2\alpha_3,~\alpha_3+\alpha_4$
\item $\alpha_1+2\alpha_2+2\alpha_3+2\alpha_4,~\alpha_1+\alpha_2+\alpha_3,~\alpha_2+2\alpha_3+\alpha_4,~\alpha_2+2\alpha_3,~\alpha_3+\alpha_4,~\alpha_3$
\item $\alpha_1+2\alpha_2+2\alpha_3+2\alpha_4,~\alpha_1+\alpha_2+\alpha_3,~\alpha_2+2\alpha_3+\alpha_4,~\alpha_2+2\alpha_3,~\alpha_3$
\item $\alpha_1+2\alpha_2+2\alpha_3+2\alpha_4,~\alpha_1+\alpha_2+\alpha_3,~\alpha_2+2\alpha_3+\alpha_4,~\alpha_3+\alpha_4$
\item $\alpha_1+2\alpha_2+2\alpha_3+2\alpha_4,~\alpha_1+\alpha_2+\alpha_3,~\alpha_2+2\alpha_3+\alpha_4,~\alpha_3+\alpha_4,~\alpha_3$
\item $\alpha_1+2\alpha_2+2\alpha_3+2\alpha_4,~\alpha_1+\alpha_2+\alpha_3,~\alpha_2+2\alpha_3+\alpha_4,~\alpha_3$
\item $\alpha_1+2\alpha_2+2\alpha_3+2\alpha_4,~\alpha_1+\alpha_2+\alpha_3,~\alpha_2+2\alpha_3$
\item $\alpha_1+2\alpha_2+2\alpha_3+2\alpha_4,~\alpha_1+\alpha_2+\alpha_3,~\alpha_2+2\alpha_3,~\alpha_3+\alpha_4$
\item $\alpha_1+2\alpha_2+2\alpha_3+2\alpha_4,~\alpha_1+\alpha_2+\alpha_3,~\alpha_2+2\alpha_3,~\alpha_3+\alpha_4,~\alpha_3$
\item $\alpha_1+2\alpha_2+2\alpha_3+2\alpha_4,~\alpha_1+\alpha_2+\alpha_3,~\alpha_2+2\alpha_3,~\alpha_3$
\item $\alpha_1+2\alpha_2+2\alpha_3+2\alpha_4,~\alpha_1+\alpha_2+\alpha_3,~\alpha_3+\alpha_4$
\item $\alpha_1+2\alpha_2+2\alpha_3+2\alpha_4,~\alpha_1+\alpha_2+\alpha_3,~\alpha_3+\alpha_4,~\alpha_3$
\item $\alpha_1+2\alpha_2+2\alpha_3+2\alpha_4,~\alpha_1+\alpha_2+\alpha_3,~\alpha_3$
\item $\alpha_1+2\alpha_2+2\alpha_3+2\alpha_4,~\alpha_1$
\item $\alpha_1+2\alpha_2+2\alpha_3+2\alpha_4,~\alpha_1,~\alpha_2+2\alpha_3+\alpha_4$
\item $\alpha_1+2\alpha_2+2\alpha_3+2\alpha_4,~\alpha_1,~\alpha_2+2\alpha_3+\alpha_4,~\alpha_2+2\alpha_3$
\item $\alpha_1+2\alpha_2+2\alpha_3+2\alpha_4,~\alpha_1,~\alpha_2+2\alpha_3+\alpha_4,~\alpha_2+2\alpha_3,~\alpha_3$
\item $\alpha_1+2\alpha_2+2\alpha_3+2\alpha_4,~\alpha_1,~\alpha_2+2\alpha_3+\alpha_4,~\alpha_3$
\item $\alpha_1+2\alpha_2+2\alpha_3+2\alpha_4,~\alpha_1,~\alpha_2+2\alpha_3$
\item $\alpha_1+2\alpha_2+2\alpha_3+2\alpha_4,~\alpha_1,~\alpha_2+2\alpha_3,~\alpha_3+\alpha_4$
\item $\alpha_1+2\alpha_2+2\alpha_3+2\alpha_4,~\alpha_1,~\alpha_2+\alpha_3$
\item $\alpha_1+2\alpha_2+2\alpha_3+2\alpha_4,~\alpha_1,~\alpha_2+\alpha_3,~\alpha_3+\alpha_4$
\item $\alpha_1+2\alpha_2+2\alpha_3+2\alpha_4,~\alpha_1,~\alpha_2+\alpha_3,~\alpha_3+\alpha_4,~\alpha_3$
\item $\alpha_1+2\alpha_2+2\alpha_3+2\alpha_4,~\alpha_1,~\alpha_2+\alpha_3,~\alpha_3$
\item $\alpha_1+2\alpha_2+2\alpha_3+2\alpha_4,~\alpha_1,~\alpha_3+\alpha_4$
\item $\alpha_1+2\alpha_2+2\alpha_3+2\alpha_4,~\alpha_1,~\alpha_3+\alpha_4,~\alpha_3$
\item $\alpha_1+2\alpha_2+2\alpha_3+2\alpha_4,~\alpha_1,~\alpha_3$
\item $\alpha_1+2\alpha_2+2\alpha_3+2\alpha_4,~\alpha_2+2\alpha_3+\alpha_4$
\item $\alpha_1+2\alpha_2+2\alpha_3+2\alpha_4,~\alpha_2+2\alpha_3+\alpha_4,~\alpha_2+2\alpha_3$
\item $\alpha_1+2\alpha_2+2\alpha_3+2\alpha_4,~\alpha_2+2\alpha_3+\alpha_4,~\alpha_2+2\alpha_3,~\alpha_3$
\item $\alpha_1+2\alpha_2+2\alpha_3+2\alpha_4,~\alpha_2+2\alpha_3+\alpha_4,~\alpha_3$
\item $\alpha_1+2\alpha_2+2\alpha_3+2\alpha_4,~\alpha_2+2\alpha_3$
\item $\alpha_1+2\alpha_2+2\alpha_3+2\alpha_4,~\alpha_2+2\alpha_3,~\alpha_3+\alpha_4$
\item $\alpha_1+2\alpha_2+2\alpha_3+2\alpha_4,~\alpha_2+\alpha_3$
\item $\alpha_1+2\alpha_2+2\alpha_3+2\alpha_4,~\alpha_2+\alpha_3,~\alpha_3+\alpha_4$
\item $\alpha_1+2\alpha_2+2\alpha_3+2\alpha_4,~\alpha_2+\alpha_3,~\alpha_3+\alpha_4,~\alpha_3$
\item $\alpha_1+2\alpha_2+2\alpha_3+2\alpha_4,~\alpha_2+\alpha_3,~\alpha_3$
\item $\alpha_1+2\alpha_2+2\alpha_3+2\alpha_4,~\alpha_3+\alpha_4$
\item $\alpha_1+2\alpha_2+2\alpha_3+2\alpha_4,~\alpha_3+\alpha_4,~\alpha_3$
\item $\alpha_1+2\alpha_2+2\alpha_3+2\alpha_4,~\alpha_3$
\item $\alpha_1+2\alpha_2+3\alpha_3+\alpha_4$
\item $\alpha_1+2\alpha_2+3\alpha_3+\alpha_4,~\alpha_1+2\alpha_2+2\alpha_3+2\alpha_4$
\item $\alpha_1+2\alpha_2+3\alpha_3+\alpha_4,~\alpha_1+2\alpha_2+2\alpha_3+2\alpha_4,~\alpha_1$
\item $\alpha_1+2\alpha_2+3\alpha_3+\alpha_4,~\alpha_1+2\alpha_2+2\alpha_3+2\alpha_4,~\alpha_1,~\alpha_4$
\item $\alpha_1+2\alpha_2+3\alpha_3+\alpha_4,~\alpha_1+2\alpha_2+2\alpha_3+2\alpha_4,~\alpha_4$
\item $\alpha_1+2\alpha_2+3\alpha_3+\alpha_4,~\alpha_1+\alpha_2+2\alpha_3+2\alpha_4$
\item $\alpha_1+2\alpha_2+3\alpha_3+\alpha_4,~\alpha_1+\alpha_2+2\alpha_3+2\alpha_4,~\alpha_1+\alpha_2$
\item $\alpha_1+2\alpha_2+3\alpha_3+\alpha_4,~\alpha_1+\alpha_2+2\alpha_3+2\alpha_4,~\alpha_1+\alpha_2,~\alpha_2$
\item $\alpha_1+2\alpha_2+3\alpha_3+\alpha_4,~\alpha_1+\alpha_2+2\alpha_3+2\alpha_4,~\alpha_1+\alpha_2,~\alpha_2,~\alpha_4$
\item $\alpha_1+2\alpha_2+3\alpha_3+\alpha_4,~\alpha_1+\alpha_2+2\alpha_3+2\alpha_4,~\alpha_1+\alpha_2,~\alpha_4$
\item $\alpha_1+2\alpha_2+3\alpha_3+\alpha_4,~\alpha_1+\alpha_2+2\alpha_3+2\alpha_4,~\alpha_2$
\item $\alpha_1+2\alpha_2+3\alpha_3+\alpha_4,~\alpha_1+\alpha_2+2\alpha_3+2\alpha_4,~\alpha_2,~\alpha_4$
\item $\alpha_1+2\alpha_2+3\alpha_3+\alpha_4,~\alpha_1+\alpha_2+2\alpha_3+2\alpha_4,~\alpha_4$
\item $\alpha_1+2\alpha_2+3\alpha_3+\alpha_4,~\alpha_1+\alpha_2$
\item $\alpha_1+2\alpha_2+3\alpha_3+\alpha_4,~\alpha_1+\alpha_2,~\alpha_2+2\alpha_3+2\alpha_4$
\item $\alpha_1+2\alpha_2+3\alpha_3+\alpha_4,~\alpha_1+\alpha_2,~\alpha_2+2\alpha_3+2\alpha_4,~\alpha_4$
\item $\alpha_1+2\alpha_2+3\alpha_3+\alpha_4,~\alpha_1+\alpha_2,~\alpha_4$
\item $\alpha_1+2\alpha_2+3\alpha_3+\alpha_4,~\alpha_1$
\item $\alpha_1+2\alpha_2+3\alpha_3+\alpha_4,~\alpha_1,~\alpha_2+2\alpha_3+2\alpha_4$
\item $\alpha_1+2\alpha_2+3\alpha_3+\alpha_4,~\alpha_1,~\alpha_2+2\alpha_3+2\alpha_4,~\alpha_2$
\item $\alpha_1+2\alpha_2+3\alpha_3+\alpha_4,~\alpha_1,~\alpha_2+2\alpha_3+2\alpha_4,~\alpha_2,~\alpha_4$
\item $\alpha_1+2\alpha_2+3\alpha_3+\alpha_4,~\alpha_1,~\alpha_2+2\alpha_3+2\alpha_4,~\alpha_4$
\item $\alpha_1+2\alpha_2+3\alpha_3+\alpha_4,~\alpha_1,~\alpha_2$
\item $\alpha_1+2\alpha_2+3\alpha_3+\alpha_4,~\alpha_1,~\alpha_2,~\alpha_4$
\item $\alpha_1+2\alpha_2+3\alpha_3+\alpha_4,~\alpha_1,~\alpha_4$
\item $\alpha_1+2\alpha_2+3\alpha_3+\alpha_4,~\alpha_2+2\alpha_3+2\alpha_4$
\item $\alpha_1+2\alpha_2+3\alpha_3+\alpha_4,~\alpha_2+2\alpha_3+2\alpha_4,~\alpha_2$
\item $\alpha_1+2\alpha_2+3\alpha_3+\alpha_4,~\alpha_2+2\alpha_3+2\alpha_4,~\alpha_2,~\alpha_4$
\item $\alpha_1+2\alpha_2+3\alpha_3+\alpha_4,~\alpha_2+2\alpha_3+2\alpha_4,~\alpha_4$
\item $\alpha_1+2\alpha_2+3\alpha_3+\alpha_4,~\alpha_2$
\item $\alpha_1+2\alpha_2+3\alpha_3+\alpha_4,~\alpha_2,~\alpha_4$
\item $\alpha_1+2\alpha_2+3\alpha_3+\alpha_4,~\alpha_4$
\item $\alpha_1+2\alpha_2+3\alpha_3+2\alpha_4$
\item $\alpha_1+2\alpha_2+3\alpha_3+2\alpha_4,~\alpha_1+2\alpha_2+2\alpha_3$
\item $\alpha_1+2\alpha_2+3\alpha_3+2\alpha_4,~\alpha_1+2\alpha_2+2\alpha_3,~\alpha_1$
\item $\alpha_1+2\alpha_2+3\alpha_3+2\alpha_4,~\alpha_1+\alpha_2+2\alpha_3$
\item $\alpha_1+2\alpha_2+3\alpha_3+2\alpha_4,~\alpha_1+\alpha_2+2\alpha_3,~\alpha_1+\alpha_2$
\item $\alpha_1+2\alpha_2+3\alpha_3+2\alpha_4,~\alpha_1+\alpha_2+2\alpha_3,~\alpha_1+\alpha_2,~\alpha_2$
\item $\alpha_1+2\alpha_2+3\alpha_3+2\alpha_4,~\alpha_1+\alpha_2+2\alpha_3,~\alpha_2$
\item $\alpha_1+2\alpha_2+3\alpha_3+2\alpha_4,~\alpha_1+\alpha_2$
\item $\alpha_1+2\alpha_2+3\alpha_3+2\alpha_4,~\alpha_1+\alpha_2,~\alpha_2+2\alpha_3$
\item $\alpha_1+2\alpha_2+3\alpha_3+2\alpha_4,~\alpha_1$
\item $\alpha_1+2\alpha_2+3\alpha_3+2\alpha_4,~\alpha_1,~\alpha_2+2\alpha_3$
\item $\alpha_1+2\alpha_2+3\alpha_3+2\alpha_4,~\alpha_1,~\alpha_2+2\alpha_3,~\alpha_2$
\item $\alpha_1+2\alpha_2+3\alpha_3+2\alpha_4,~\alpha_1,~\alpha_2$
\item $\alpha_1+2\alpha_2+3\alpha_3+2\alpha_4,~\alpha_2+2\alpha_3$
\item $\alpha_1+2\alpha_2+3\alpha_3+2\alpha_4,~\alpha_2+2\alpha_3,~\alpha_2$
\item $\alpha_1+2\alpha_2+3\alpha_3+2\alpha_4,~\alpha_2$
\item $\alpha_1+2\alpha_2+4\alpha_3+2\alpha_4$
\item $\alpha_1+2\alpha_2+4\alpha_3+2\alpha_4,~\alpha_1+2\alpha_2+2\alpha_3+2\alpha_4$
\item $\alpha_1+2\alpha_2+4\alpha_3+2\alpha_4,~\alpha_1+2\alpha_2+2\alpha_3+2\alpha_4,~\alpha_1+2\alpha_2+2\alpha_3$
\item $\alpha_1+2\alpha_2+4\alpha_3+2\alpha_4,~\alpha_1+2\alpha_2+2\alpha_3+2\alpha_4,~\alpha_1+2\alpha_2+2\alpha_3,~\alpha_1$
\item $\alpha_1+2\alpha_2+4\alpha_3+2\alpha_4,~\alpha_1+2\alpha_2+2\alpha_3+2\alpha_4,~\alpha_1+\alpha_2+\alpha_3$
\item $\alpha_1+2\alpha_2+4\alpha_3+2\alpha_4,~\alpha_1+2\alpha_2+2\alpha_3+2\alpha_4,~\alpha_1+\alpha_2+\alpha_3,~\alpha_1+\alpha_2$
\item $\alpha_1+2\alpha_2+4\alpha_3+2\alpha_4,~\alpha_1+2\alpha_2+2\alpha_3+2\alpha_4,~\alpha_1+\alpha_2+\alpha_3,~\alpha_1+\alpha_2,~\alpha_2$
\item $\alpha_1+2\alpha_2+4\alpha_3+2\alpha_4,~\alpha_1+2\alpha_2+2\alpha_3+2\alpha_4,~\alpha_1+\alpha_2+\alpha_3,~\alpha_2$
\item $\alpha_1+2\alpha_2+4\alpha_3+2\alpha_4,~\alpha_1+2\alpha_2+2\alpha_3+2\alpha_4,~\alpha_1+\alpha_2,~\alpha_2+\alpha_3$
\item $\alpha_1+2\alpha_2+4\alpha_3+2\alpha_4,~\alpha_1+2\alpha_2+2\alpha_3+2\alpha_4,~\alpha_1+\alpha_2$
\item $\alpha_1+2\alpha_2+4\alpha_3+2\alpha_4,~\alpha_1+2\alpha_2+2\alpha_3+2\alpha_4,~\alpha_1$
\item $\alpha_1+2\alpha_2+4\alpha_3+2\alpha_4,~\alpha_1+2\alpha_2+2\alpha_3+2\alpha_4,~\alpha_1,~\alpha_2+\alpha_3$
\item $\alpha_1+2\alpha_2+4\alpha_3+2\alpha_4,~\alpha_1+2\alpha_2+2\alpha_3+2\alpha_4,~\alpha_1,~\alpha_2+\alpha_3,~\alpha_2$
\item $\alpha_1+2\alpha_2+4\alpha_3+2\alpha_4,~\alpha_1+2\alpha_2+2\alpha_3+2\alpha_4,~\alpha_1,~\alpha_2$
\item $\alpha_1+2\alpha_2+4\alpha_3+2\alpha_4,~\alpha_1+2\alpha_2+2\alpha_3+2\alpha_4,~\alpha_2+\alpha_3$
\item $\alpha_1+2\alpha_2+4\alpha_3+2\alpha_4,~\alpha_1+2\alpha_2+2\alpha_3+2\alpha_4,~\alpha_2+\alpha_3,~\alpha_2$
\item $\alpha_1+2\alpha_2+4\alpha_3+2\alpha_4,~\alpha_1+2\alpha_2+2\alpha_3+2\alpha_4,~\alpha_2$
\item $\alpha_1+2\alpha_2+4\alpha_3+2\alpha_4,~\alpha_1+2\alpha_2+2\alpha_3+\alpha_4$
\item $\alpha_1+2\alpha_2+4\alpha_3+2\alpha_4,~\alpha_1+2\alpha_2+2\alpha_3+\alpha_4,~\alpha_1$
\item $\alpha_1+2\alpha_2+4\alpha_3+2\alpha_4,~\alpha_1+2\alpha_2+2\alpha_3$
\item $\alpha_1+2\alpha_2+4\alpha_3+2\alpha_4,~\alpha_1+2\alpha_2+2\alpha_3,~\alpha_1+\alpha_2+\alpha_3+\alpha_4$
\item $\alpha_1+2\alpha_2+4\alpha_3+2\alpha_4,~\alpha_1+2\alpha_2+2\alpha_3,~\alpha_1+\alpha_2+\alpha_3+\alpha_4,~\alpha_1+\alpha_2$
\item $\alpha_1+2\alpha_2+4\alpha_3+2\alpha_4,~\alpha_1+2\alpha_2+2\alpha_3,~\alpha_1+\alpha_2+\alpha_3+\alpha_4,~\alpha_1+\alpha_2,~\alpha_2$
\item $\alpha_1+2\alpha_2+4\alpha_3+2\alpha_4,~\alpha_1+2\alpha_2+2\alpha_3,~\alpha_1+\alpha_2+\alpha_3+\alpha_4,~\alpha_1+\alpha_2,~\alpha_2,~\alpha_4$
\item $\alpha_1+2\alpha_2+4\alpha_3+2\alpha_4,~\alpha_1+2\alpha_2+2\alpha_3,~\alpha_1+\alpha_2+\alpha_3+\alpha_4,~\alpha_1+\alpha_2,~\alpha_4$
\item $\alpha_1+2\alpha_2+4\alpha_3+2\alpha_4,~\alpha_1+2\alpha_2+2\alpha_3,~\alpha_1+\alpha_2+\alpha_3+\alpha_4,~\alpha_2$
\item $\alpha_1+2\alpha_2+4\alpha_3+2\alpha_4,~\alpha_1+2\alpha_2+2\alpha_3,~\alpha_1+\alpha_2+\alpha_3+\alpha_4,~\alpha_2,~\alpha_4$
\item $\alpha_1+2\alpha_2+4\alpha_3+2\alpha_4,~\alpha_1+2\alpha_2+2\alpha_3,~\alpha_1+\alpha_2+\alpha_3+\alpha_4,~\alpha_4$
\item $\alpha_1+2\alpha_2+4\alpha_3+2\alpha_4,~\alpha_1+2\alpha_2+2\alpha_3,~\alpha_1+\alpha_2$
\item $\alpha_1+2\alpha_2+4\alpha_3+2\alpha_4,~\alpha_1+2\alpha_2+2\alpha_3,~\alpha_1+\alpha_2,~\alpha_2+\alpha_3+\alpha_4$
\item $\alpha_1+2\alpha_2+4\alpha_3+2\alpha_4,~\alpha_1+2\alpha_2+2\alpha_3,~\alpha_1+\alpha_2,~\alpha_2+\alpha_3+\alpha_4,~\alpha_4$
\item $\alpha_1+2\alpha_2+4\alpha_3+2\alpha_4,~\alpha_1+2\alpha_2+2\alpha_3,~\alpha_1+\alpha_2,~\alpha_4$
\item $\alpha_1+2\alpha_2+4\alpha_3+2\alpha_4,~\alpha_1+2\alpha_2+2\alpha_3,~\alpha_1$
\item $\alpha_1+2\alpha_2+4\alpha_3+2\alpha_4,~\alpha_1+2\alpha_2+2\alpha_3,~\alpha_1,~\alpha_2+\alpha_3+\alpha_4$
\item $\alpha_1+2\alpha_2+4\alpha_3+2\alpha_4,~\alpha_1+2\alpha_2+2\alpha_3,~\alpha_1,~\alpha_2+\alpha_3+\alpha_4,~\alpha_2$
\item $\alpha_1+2\alpha_2+4\alpha_3+2\alpha_4,~\alpha_1+2\alpha_2+2\alpha_3,~\alpha_1,~\alpha_2+\alpha_3+\alpha_4,~\alpha_2,~\alpha_4$
\item $\alpha_1+2\alpha_2+4\alpha_3+2\alpha_4,~\alpha_1+2\alpha_2+2\alpha_3,~\alpha_1,~\alpha_2+\alpha_3+\alpha_4,~\alpha_4$
\item $\alpha_1+2\alpha_2+4\alpha_3+2\alpha_4,~\alpha_1+2\alpha_2+2\alpha_3,~\alpha_1,~\alpha_2$
\item $\alpha_1+2\alpha_2+4\alpha_3+2\alpha_4,~\alpha_1+2\alpha_2+2\alpha_3,~\alpha_1,~\alpha_2,~\alpha_4$
\item $\alpha_1+2\alpha_2+4\alpha_3+2\alpha_4,~\alpha_1+2\alpha_2+2\alpha_3,~\alpha_1,~\alpha_4$
\item $\alpha_1+2\alpha_2+4\alpha_3+2\alpha_4,~\alpha_1+2\alpha_2+2\alpha_3,~\alpha_2+\alpha_3+\alpha_4$
\item $\alpha_1+2\alpha_2+4\alpha_3+2\alpha_4,~\alpha_1+2\alpha_2+2\alpha_3,~\alpha_2+\alpha_3+\alpha_4,~\alpha_2$
\item $\alpha_1+2\alpha_2+4\alpha_3+2\alpha_4,~\alpha_1+2\alpha_2+2\alpha_3,~\alpha_2+\alpha_3+\alpha_4,~\alpha_2,~\alpha_4$
\item $\alpha_1+2\alpha_2+4\alpha_3+2\alpha_4,~\alpha_1+2\alpha_2+2\alpha_3,~\alpha_2+\alpha_3+\alpha_4,~\alpha_4$
\item $\alpha_1+2\alpha_2+4\alpha_3+2\alpha_4,~\alpha_1+2\alpha_2+2\alpha_3,~\alpha_2$
\item $\alpha_1+2\alpha_2+4\alpha_3+2\alpha_4,~\alpha_1+2\alpha_2+2\alpha_3,~\alpha_2,~\alpha_4$
\item $\alpha_1+2\alpha_2+4\alpha_3+2\alpha_4,~\alpha_1+2\alpha_2+2\alpha_3,~\alpha_4$
\item $\alpha_1+2\alpha_2+4\alpha_3+2\alpha_4,~\alpha_1+\alpha_2+\alpha_3+\alpha_4$
\item $\alpha_1+2\alpha_2+4\alpha_3+2\alpha_4,~\alpha_1+\alpha_2+\alpha_3+\alpha_4,~\alpha_1+\alpha_2$
\item $\alpha_1+2\alpha_2+4\alpha_3+2\alpha_4,~\alpha_1+\alpha_2+\alpha_3+\alpha_4,~\alpha_1+\alpha_2,~\alpha_2+\alpha_3$
\item $\alpha_1+2\alpha_2+4\alpha_3+2\alpha_4,~\alpha_1+\alpha_2+\alpha_3+\alpha_4,~\alpha_1+\alpha_2,~\alpha_2+\alpha_3,~\alpha_2$
\item $\alpha_1+2\alpha_2+4\alpha_3+2\alpha_4,~\alpha_1+\alpha_2+\alpha_3+\alpha_4,~\alpha_1+\alpha_2,~\alpha_2$
\item $\alpha_1+2\alpha_2+4\alpha_3+2\alpha_4,~\alpha_1+\alpha_2+\alpha_3+\alpha_4,~\alpha_2+\alpha_3$
\item $\alpha_1+2\alpha_2+4\alpha_3+2\alpha_4,~\alpha_1+\alpha_2+\alpha_3+\alpha_4,~\alpha_2+\alpha_3,~\alpha_2$
\item $\alpha_1+2\alpha_2+4\alpha_3+2\alpha_4,~\alpha_1+\alpha_2+\alpha_3+\alpha_4,~\alpha_2$
\item $\alpha_1+2\alpha_2+4\alpha_3+2\alpha_4,~\alpha_1+\alpha_2+\alpha_3$
\item $\alpha_1+2\alpha_2+4\alpha_3+2\alpha_4,~\alpha_1+\alpha_2+\alpha_3,~\alpha_1+\alpha_2$
\item $\alpha_1+2\alpha_2+4\alpha_3+2\alpha_4,~\alpha_1+\alpha_2+\alpha_3,~\alpha_1+\alpha_2,~\alpha_2+\alpha_3+\alpha_4$
\item $\alpha_1+2\alpha_2+4\alpha_3+2\alpha_4,~\alpha_1+\alpha_2+\alpha_3,~\alpha_1+\alpha_2,~\alpha_2$
\item $\alpha_1+2\alpha_2+4\alpha_3+2\alpha_4,~\alpha_1+\alpha_2+\alpha_3,~\alpha_1+\alpha_2,~\alpha_2,~\alpha_4$
\item $\alpha_1+2\alpha_2+4\alpha_3+2\alpha_4,~\alpha_1+\alpha_2+\alpha_3,~\alpha_1+\alpha_2,~\alpha_4$
\item $\alpha_1+2\alpha_2+4\alpha_3+2\alpha_4,~\alpha_1+\alpha_2+\alpha_3,~\alpha_2+\alpha_3+\alpha_4$
\item $\alpha_1+2\alpha_2+4\alpha_3+2\alpha_4,~\alpha_1+\alpha_2+\alpha_3,~\alpha_2+\alpha_3+\alpha_4,~\alpha_2$
\item $\alpha_1+2\alpha_2+4\alpha_3+2\alpha_4,~\alpha_1+\alpha_2+\alpha_3,~\alpha_2+\alpha_3+\alpha_4,~\alpha_2,~\alpha_4$
\item $\alpha_1+2\alpha_2+4\alpha_3+2\alpha_4,~\alpha_1+\alpha_2+\alpha_3,~\alpha_2+\alpha_3+\alpha_4,~\alpha_4$
\item $\alpha_1+2\alpha_2+4\alpha_3+2\alpha_4,~\alpha_1+\alpha_2+\alpha_3,~\alpha_2$
\item $\alpha_1+2\alpha_2+4\alpha_3+2\alpha_4,~\alpha_1+\alpha_2+\alpha_3,~\alpha_2,~\alpha_4$
\item $\alpha_1+2\alpha_2+4\alpha_3+2\alpha_4,~\alpha_1+\alpha_2+\alpha_3,~\alpha_4$
\item $\alpha_1+2\alpha_2+4\alpha_3+2\alpha_4,~\alpha_1+\alpha_2$
\item $\alpha_1+2\alpha_2+4\alpha_3+2\alpha_4,~\alpha_1+\alpha_2,~\alpha_2+\alpha_3+\alpha_4$
\item $\alpha_1+2\alpha_2+4\alpha_3+2\alpha_4,~\alpha_1+\alpha_2,~\alpha_2+\alpha_3$
\item $\alpha_1+2\alpha_2+4\alpha_3+2\alpha_4,~\alpha_1+\alpha_2,~\alpha_2+\alpha_3,~\alpha_4$
\item $\alpha_1+2\alpha_2+4\alpha_3+2\alpha_4,~\alpha_1+\alpha_2,~\alpha_4$
\item $\alpha_1+2\alpha_2+4\alpha_3+2\alpha_4,~\alpha_1$
\item $\alpha_1+2\alpha_2+4\alpha_3+2\alpha_4,~\alpha_1,~\alpha_2+\alpha_3+\alpha_4$
\item $\alpha_1+2\alpha_2+4\alpha_3+2\alpha_4,~\alpha_1,~\alpha_2+\alpha_3+\alpha_4,~\alpha_2$
\item $\alpha_1+2\alpha_2+4\alpha_3+2\alpha_4,~\alpha_1,~\alpha_2+\alpha_3$
\item $\alpha_1+2\alpha_2+4\alpha_3+2\alpha_4,~\alpha_1,~\alpha_2+\alpha_3,~\alpha_2$
\item $\alpha_1+2\alpha_2+4\alpha_3+2\alpha_4,~\alpha_1,~\alpha_2+\alpha_3,~\alpha_2,~\alpha_4$
\item $\alpha_1+2\alpha_2+4\alpha_3+2\alpha_4,~\alpha_1,~\alpha_2+\alpha_3,~\alpha_4$
\item $\alpha_1+2\alpha_2+4\alpha_3+2\alpha_4$, $\alpha_1$, $\alpha_2$
\item $\alpha_1+2\alpha_2+4\alpha_3+2\alpha_4$, $\alpha_1$, $\alpha_2$, $\alpha_4$
\item $\alpha_1+2\alpha_2+4\alpha_3+2\alpha_4,~\alpha_1,~\alpha_4$
\item $\alpha_1+2\alpha_2+4\alpha_3+2\alpha_4,~\alpha_2+\alpha_3+\alpha_4$
\item $\alpha_1+2\alpha_2+4\alpha_3+2\alpha_4,~\alpha_2+\alpha_3+\alpha_4,~\alpha_2$
\item $\alpha_1+2\alpha_2+4\alpha_3+2\alpha_4,~\alpha_2+\alpha_3$
\item $\alpha_1+2\alpha_2+4\alpha_3+2\alpha_4,~\alpha_2+\alpha_3,~\alpha_2$
\item $\alpha_1+2\alpha_2+4\alpha_3+2\alpha_4,~\alpha_2+\alpha_3,~\alpha_2,~\alpha_4$
\item $\alpha_1+2\alpha_2+4\alpha_3+2\alpha_4,~\alpha_2+\alpha_3,~\alpha_4$
\item $\alpha_1+2\alpha_2+4\alpha_3+2\alpha_4,~\alpha_2$
\item $\alpha_1+2\alpha_2+4\alpha_3+2\alpha_4,~\alpha_2,~\alpha_4$
\item $\alpha_1+2\alpha_2+4\alpha_3+2\alpha_4,~\alpha_4$
\item $\alpha_1+3\alpha_2+4\alpha_3+2\alpha_4$
\item $\alpha_1+3\alpha_2+4\alpha_3+2\alpha_4,~\alpha_1+\alpha_2+2\alpha_3+2\alpha_4$
\item $\alpha_1+3\alpha_2+4\alpha_3+2\alpha_4,~\alpha_1+\alpha_2+2\alpha_3+2\alpha_4,~\alpha_1+\alpha_2+2\alpha_3$
\item $\alpha_1+3\alpha_2+4\alpha_3+2\alpha_4,~\alpha_1+\alpha_2+2\alpha_3+2\alpha_4,~\alpha_1+\alpha_2+2\alpha_3,~\alpha_1+\alpha_2$
\item $\alpha_1+3\alpha_2+4\alpha_3+2\alpha_4,~\alpha_1+\alpha_2+2\alpha_3+2\alpha_4,~\alpha_1+\alpha_2+2\alpha_3,~\alpha_1+\alpha_2,~\alpha_1$
\item $\alpha_1+3\alpha_2+4\alpha_3+2\alpha_4,~\alpha_1+\alpha_2+2\alpha_3+2\alpha_4,~\alpha_1+\alpha_2+2\alpha_3,~\alpha_1$
\item $\alpha_1+3\alpha_2+4\alpha_3+2\alpha_4,~\alpha_1+\alpha_2+2\alpha_3+2\alpha_4,~\alpha_1+\alpha_2+\alpha_3$
\item $\alpha_1+3\alpha_2+4\alpha_3+2\alpha_4,~\alpha_1+\alpha_2+2\alpha_3+2\alpha_4,~\alpha_1+\alpha_2+\alpha_3,~\alpha_1$
\item $\alpha_1+3\alpha_2+4\alpha_3+2\alpha_4,~\alpha_1+\alpha_2+2\alpha_3+2\alpha_4,~\alpha_1+\alpha_2$
\item $\alpha_1+3\alpha_2+4\alpha_3+2\alpha_4,~\alpha_1+\alpha_2+2\alpha_3+2\alpha_4,~\alpha_1+\alpha_2,~\alpha_1$
\item $\alpha_1+3\alpha_2+4\alpha_3+2\alpha_4,~\alpha_1+\alpha_2+2\alpha_3+2\alpha_4,~\alpha_1+\alpha_2,~\alpha_1,~\alpha_3$
\item $\alpha_1+3\alpha_2+4\alpha_3+2\alpha_4,~\alpha_1+\alpha_2+2\alpha_3+2\alpha_4,~\alpha_1+\alpha_2,~\alpha_3$
\item $\alpha_1+3\alpha_2+4\alpha_3+2\alpha_4,~\alpha_1+\alpha_2+2\alpha_3+2\alpha_4,~\alpha_1$
\item $\alpha_1+3\alpha_2+4\alpha_3+2\alpha_4,~\alpha_1+\alpha_2+2\alpha_3+2\alpha_4,~\alpha_1,~\alpha_3$
\item $\alpha_1+3\alpha_2+4\alpha_3+2\alpha_4,~\alpha_1+\alpha_2+2\alpha_3+2\alpha_4,~\alpha_3$
\item $\alpha_1+3\alpha_2+4\alpha_3+2\alpha_4,~\alpha_1+\alpha_2+2\alpha_3+\alpha_4$
\item $\alpha_1+3\alpha_2+4\alpha_3+2\alpha_4,~\alpha_1+\alpha_2+2\alpha_3+\alpha_4,~\alpha_1+\alpha_2$
\item $\alpha_1+3\alpha_2+4\alpha_3+2\alpha_4,~\alpha_1+\alpha_2+2\alpha_3+\alpha_4,~\alpha_1+\alpha_2,~\alpha_1$
\item $\alpha_1+3\alpha_2+4\alpha_3+2\alpha_4,~\alpha_1+\alpha_2+2\alpha_3+\alpha_4,~\alpha_1$
\item $\alpha_1+3\alpha_2+4\alpha_3+2\alpha_4,~\alpha_1+\alpha_2+2\alpha_3$
\item $\alpha_1+3\alpha_2+4\alpha_3+2\alpha_4,~\alpha_1+\alpha_2+2\alpha_3,~\alpha_1+\alpha_2+\alpha_3+\alpha_4$
\item $\alpha_1+3\alpha_2+4\alpha_3+2\alpha_4,~\alpha_1+\alpha_2+2\alpha_3,~\alpha_1+\alpha_2+\alpha_3+\alpha_4,~\alpha_1$
\item $\alpha_1+3\alpha_2+4\alpha_3+2\alpha_4,~\alpha_1+\alpha_2+2\alpha_3,~\alpha_1+\alpha_2+\alpha_3+\alpha_4,~\alpha_1,~\alpha_4$
\item $\alpha_1+3\alpha_2+4\alpha_3+2\alpha_4,~\alpha_1+\alpha_2+2\alpha_3,~\alpha_1+\alpha_2+\alpha_3+\alpha_4,~\alpha_4$
\item $\alpha_1+3\alpha_2+4\alpha_3+2\alpha_4,~\alpha_1+\alpha_2+2\alpha_3,~\alpha_1+\alpha_2$
\item $\alpha_1+3\alpha_2+4\alpha_3+2\alpha_4,~\alpha_1+\alpha_2+2\alpha_3,~\alpha_1+\alpha_2,~\alpha_1$
\item $\alpha_1+3\alpha_2+4\alpha_3+2\alpha_4,~\alpha_1+\alpha_2+2\alpha_3,~\alpha_1+\alpha_2,~\alpha_1,~\alpha_3+\alpha_4$
\item $\alpha_1+3\alpha_2+4\alpha_3+2\alpha_4,~\alpha_1+\alpha_2+2\alpha_3,~\alpha_1+\alpha_2,~\alpha_1,~\alpha_3+\alpha_4,~\alpha_4$
\item $\alpha_1+3\alpha_2+4\alpha_3+2\alpha_4,~\alpha_1+\alpha_2+2\alpha_3,~\alpha_1+\alpha_2,~\alpha_1,~\alpha_4$
\item $\alpha_1+3\alpha_2+4\alpha_3+2\alpha_4,~\alpha_1+\alpha_2+2\alpha_3,~\alpha_1+\alpha_2,~\alpha_3+\alpha_4$
\item $\alpha_1+3\alpha_2+4\alpha_3+2\alpha_4,~\alpha_1+\alpha_2+2\alpha_3,~\alpha_1+\alpha_2,~\alpha_3+\alpha_4,~\alpha_4$
\item $\alpha_1+3\alpha_2+4\alpha_3+2\alpha_4,~\alpha_1+\alpha_2+2\alpha_3,~\alpha_1+\alpha_2,~\alpha_4$
\item $\alpha_1+3\alpha_2+4\alpha_3+2\alpha_4,~\alpha_1+\alpha_2+2\alpha_3,~\alpha_1$
\item $\alpha_1+3\alpha_2+4\alpha_3+2\alpha_4,~\alpha_1+\alpha_2+2\alpha_3,~\alpha_1,~\alpha_3+\alpha_4$
\item $\alpha_1+3\alpha_2+4\alpha_3+2\alpha_4,~\alpha_1+\alpha_2+2\alpha_3,~\alpha_1,~\alpha_3+\alpha_4,~\alpha_4$
\item $\alpha_1+3\alpha_2+4\alpha_3+2\alpha_4,~\alpha_1+\alpha_2+2\alpha_3,~\alpha_1,~\alpha_4$
\item $\alpha_1+3\alpha_2+4\alpha_3+2\alpha_4,~\alpha_1+\alpha_2+2\alpha_3,~\alpha_3+\alpha_4$
\item $\alpha_1+3\alpha_2+4\alpha_3+2\alpha_4,~\alpha_1+\alpha_2+2\alpha_3,~\alpha_3+\alpha_4,~\alpha_4$
\item $\alpha_1+3\alpha_2+4\alpha_3+2\alpha_4,~\alpha_1+\alpha_2+2\alpha_3,~\alpha_4$
\item $\alpha_1+3\alpha_2+4\alpha_3+2\alpha_4,~\alpha_1+\alpha_2+\alpha_3+\alpha_4$
\item $\alpha_1+3\alpha_2+4\alpha_3+2\alpha_4,~\alpha_1+\alpha_2+\alpha_3+\alpha_4,~\alpha_1$
\item $\alpha_1+3\alpha_2+4\alpha_3+2\alpha_4,~\alpha_1+\alpha_2+\alpha_3+\alpha_4,~\alpha_1,~\alpha_3$
\item $\alpha_1+3\alpha_2+4\alpha_3+2\alpha_4,~\alpha_1+\alpha_2+\alpha_3+\alpha_4,~\alpha_3$
\item $\alpha_1+3\alpha_2+4\alpha_3+2\alpha_4,~\alpha_1+\alpha_2+\alpha_3$
\item $\alpha_1+3\alpha_2+4\alpha_3+2\alpha_4,~\alpha_1+\alpha_2+\alpha_3,~\alpha_1$
\item $\alpha_1+3\alpha_2+4\alpha_3+2\alpha_4,~\alpha_1+\alpha_2+\alpha_3,~\alpha_1,~\alpha_3+\alpha_4$
\item $\alpha_1+3\alpha_2+4\alpha_3+2\alpha_4,~\alpha_1+\alpha_2+\alpha_3,~\alpha_1,~\alpha_3+\alpha_4,~\alpha_4$
\item $\alpha_1+3\alpha_2+4\alpha_3+2\alpha_4,~\alpha_1+\alpha_2+\alpha_3,~\alpha_1,~\alpha_4$
\item $\alpha_1+3\alpha_2+4\alpha_3+2\alpha_4,~\alpha_1+\alpha_2+\alpha_3,~\alpha_3+\alpha_4$
\item $\alpha_1+3\alpha_2+4\alpha_3+2\alpha_4,~\alpha_1+\alpha_2+\alpha_3,~\alpha_3+\alpha_4,~\alpha_4$
\item $\alpha_1+3\alpha_2+4\alpha_3+2\alpha_4,~\alpha_1+\alpha_2+\alpha_3,~\alpha_4$
\item $\alpha_1+3\alpha_2+4\alpha_3+2\alpha_4,~\alpha_1+\alpha_2$
\item $\alpha_1+3\alpha_2+4\alpha_3+2\alpha_4,~\alpha_1+\alpha_2,~\alpha_1$
\item $\alpha_1+3\alpha_2+4\alpha_3+2\alpha_4,~\alpha_1+\alpha_2,~\alpha_1,~\alpha_3+\alpha_4$
\item $\alpha_1+3\alpha_2+4\alpha_3+2\alpha_4,~\alpha_1+\alpha_2,~\alpha_1,~\alpha_3$
\item $\alpha_1+3\alpha_2+4\alpha_3+2\alpha_4,~\alpha_1+\alpha_2,~\alpha_1,~\alpha_3,~\alpha_4$
\item $\alpha_1+3\alpha_2+4\alpha_3+2\alpha_4,~\alpha_1+\alpha_2,~\alpha_1,~\alpha_4$
\item $\alpha_1+3\alpha_2+4\alpha_3+2\alpha_4,~\alpha_1+\alpha_2,~\alpha_3+\alpha_4$
\item $\alpha_1+3\alpha_2+4\alpha_3+2\alpha_4,~\alpha_1+\alpha_2,~\alpha_3$
\item $\alpha_1+3\alpha_2+4\alpha_3+2\alpha_4,~\alpha_1+\alpha_2,~\alpha_3,~\alpha_4$
\item $\alpha_1+3\alpha_2+4\alpha_3+2\alpha_4,~\alpha_1+\alpha_2,~\alpha_4$
\item $\alpha_1+3\alpha_2+4\alpha_3+2\alpha_4,~\alpha_1$
\item $\alpha_1+3\alpha_2+4\alpha_3+2\alpha_4,~\alpha_1,~\alpha_3+\alpha_4$
\item $\alpha_1+3\alpha_2+4\alpha_3+2\alpha_4,~\alpha_1,~\alpha_3$
\item $\alpha_1+3\alpha_2+4\alpha_3+2\alpha_4$, $\alpha_1$, $\alpha_3$, $\alpha_4$
\item $\alpha_1+3\alpha_2+4\alpha_3+2\alpha_4,~\alpha_1,~\alpha_4$
\item $\alpha_1+3\alpha_2+4\alpha_3+2\alpha_4,~\alpha_3+\alpha_4$
\item $\alpha_1+3\alpha_2+4\alpha_3+2\alpha_4,~\alpha_3$
\item $\alpha_1+3\alpha_2+4\alpha_3+2\alpha_4,~\alpha_3,~\alpha_4$
\item $\alpha_1+3\alpha_2+4\alpha_3+2\alpha_4,~\alpha_4$
\item $2\alpha_1+3\alpha_2+4\alpha_3+2\alpha_4$
\item $2\alpha_1+3\alpha_2+4\alpha_3+2\alpha_4,~\alpha_2+2\alpha_3+2\alpha_4$
\item $2\alpha_1+3\alpha_2+4\alpha_3+2\alpha_4,~\alpha_2+2\alpha_3+2\alpha_4,~\alpha_2+2\alpha_3$
\item $2\alpha_1+3\alpha_2+4\alpha_3+2\alpha_4,~\alpha_2+2\alpha_3+2\alpha_4,~\alpha_2+2\alpha_3,~\alpha_2$
\item $2\alpha_1+3\alpha_2+4\alpha_3+2\alpha_4,~\alpha_2+2\alpha_3+2\alpha_4,~\alpha_2+\alpha_3$
\item $2\alpha_1+3\alpha_2+4\alpha_3+2\alpha_4,~\alpha_2+2\alpha_3+2\alpha_4,~\alpha_2$
\item $2\alpha_1+3\alpha_2+4\alpha_3+2\alpha_4,~\alpha_2+2\alpha_3+2\alpha_4,~\alpha_2,~\alpha_3$
\item $2\alpha_1+3\alpha_2+4\alpha_3+2\alpha_4,~\alpha_2+2\alpha_3+2\alpha_4,~\alpha_3$
\item $2\alpha_1+3\alpha_2+4\alpha_3+2\alpha_4,~\alpha_2+2\alpha_3+\alpha_4$
\item $2\alpha_1+3\alpha_2+4\alpha_3+2\alpha_4,~\alpha_2+2\alpha_3+\alpha_4,~\alpha_2$
\item $2\alpha_1+3\alpha_2+4\alpha_3+2\alpha_4,~\alpha_2+2\alpha_3$
\item $2\alpha_1+3\alpha_2+4\alpha_3+2\alpha_4,~\alpha_2+2\alpha_3,~\alpha_2+\alpha_3+\alpha_4$
\item $2\alpha_1+3\alpha_2+4\alpha_3+2\alpha_4,~\alpha_2+2\alpha_3,~\alpha_2+\alpha_3+\alpha_4,~\alpha_4$
\item $2\alpha_1+3\alpha_2+4\alpha_3+2\alpha_4,~\alpha_2+2\alpha_3,~\alpha_2$
\item $2\alpha_1+3\alpha_2+4\alpha_3+2\alpha_4,~\alpha_2+2\alpha_3,~\alpha_2,~\alpha_3+\alpha_4$
\item $2\alpha_1+3\alpha_2+4\alpha_3+2\alpha_4,~\alpha_2+2\alpha_3,~\alpha_2,~\alpha_3+\alpha_4,~\alpha_4$
\item $2\alpha_1+3\alpha_2+4\alpha_3+2\alpha_4,~\alpha_2+2\alpha_3,~\alpha_2,~\alpha_4$
\item $2\alpha_1+3\alpha_2+4\alpha_3+2\alpha_4,~\alpha_2+2\alpha_3,~\alpha_3+\alpha_4$
\item $2\alpha_1+3\alpha_2+4\alpha_3+2\alpha_4,~\alpha_2+2\alpha_3,~\alpha_3+\alpha_4,~\alpha_4$
\item $2\alpha_1+3\alpha_2+4\alpha_3+2\alpha_4,~\alpha_2+2\alpha_3,~\alpha_4$
\item $2\alpha_1+3\alpha_2+4\alpha_3+2\alpha_4,~\alpha_2+\alpha_3+\alpha_4$
\item $2\alpha_1+3\alpha_2+4\alpha_3+2\alpha_4,~\alpha_2+\alpha_3+\alpha_4,~\alpha_3$
\item $2\alpha_1+3\alpha_2+4\alpha_3+2\alpha_4,~\alpha_2+\alpha_3$
\item $2\alpha_1+3\alpha_2+4\alpha_3+2\alpha_4,~\alpha_2+\alpha_3,~\alpha_3+\alpha_4$
\item $2\alpha_1+3\alpha_2+4\alpha_3+2\alpha_4,~\alpha_2+\alpha_3,~\alpha_3+\alpha_4,~\alpha_4$
\item $2\alpha_1+3\alpha_2+4\alpha_3+2\alpha_4,~\alpha_2+\alpha_3,~\alpha_4$
\item $2\alpha_1+3\alpha_2+4\alpha_3+2\alpha_4,~\alpha_2$
\item $2\alpha_1+3\alpha_2+4\alpha_3+2\alpha_4,~\alpha_2,~\alpha_3+\alpha_4$
\item $2\alpha_1+3\alpha_2+4\alpha_3+2\alpha_4,~\alpha_2,~\alpha_3$
\item $2\alpha_1+3\alpha_2+4\alpha_3+2\alpha_4, \alpha_2, \alpha_3, \alpha_4$
\item $2\alpha_1+3\alpha_2+4\alpha_3+2\alpha_4,~\alpha_2,~\alpha_4$
\item $2\alpha_1+3\alpha_2+4\alpha_3+2\alpha_4,~\alpha_3+\alpha_4$
\item $2\alpha_1+3\alpha_2+4\alpha_3+2\alpha_4,~\alpha_3$
\item $2\alpha_1+3\alpha_2+4\alpha_3+2\alpha_4,~\alpha_3,~\alpha_4$
\item $2\alpha_1+3\alpha_2+4\alpha_3+2\alpha_4,~\alpha_4$
\end{enumerate}
\end{multicols}

\scriptsize\begin{verbatim}
import java.util.ArrayList;
import java.util.Scanner;
public class Carriers {
public static void main(String[] args) {new Carr();}}
class Carr{
    Carr(){
        a = new int[25][25];
        m = new int[25][25];
        a[23][10]=24;
        a[22][4]=23; a[22][11]=24;
        a[21][2]=22; a[21][5]=23; a[21][12]=24;
        a[20][1]=21; a[20][3]=22; a[20][6]=23; a[20][13]=24;
        a[19][2]=21; a[19][7]=23; a[19][14]=24;
        a[18][1]=19; a[18][2]=20; a[18][3]=21; a[18][8]=23; a[18][15]=24;
        a[17][1]=18; a[17][3]=20; a[17][9]=23; a[17][16]=24;
        a[16][4]=19; a[16][5]=21; a[16][7]=22;
        a[15][1]=16; a[15][4]=18; a[15][5]=20; a[15][6]=21; a[15][8]=22;
        a[14][1]=15; a[14][4]=17; a[14][6]=20; a[14][9]=22;
        a[13][2]=15; a[13][3]=16; a[13][5]=18; a[13][6]=19; a[13][7]=20; a[13][8]=21;
        a[12][1]=13; a[12][2]=14; a[12][3]=15; a[12][5]=17; a[12][6]=18; a[12][8]=20; a[12][9]=21;
        a[11][2]=12; a[11][3]=13; a[11][7]=17; a[11][8]=18; a[11][9]=19;
        a[10][4]=11; a[10][5]=12; a[10][6]=13; a[10][7]=14; a[10][8]=15; a[10][9]=16;
        a[8][1]=9;
        a[7][1]=8;
        a[6][2]=8; a[6][3]=9;
        a[5][1]=6; a[5][2]=7; a[5][3]=8;
        a[4][2]=5; a[4][3]=6;
        a[2][1]=3;

        m[23][10]=-1;
        m[22][4]=-1; m[22][11]=-1;
        m[21][2]=-2; m[21][5]=-2; m[21][12]=-2;
        m[20][1]=-1; m[20][3]=-2; m[20][6]=-2; m[20][13]=-2;
        m[19][2]=-1; m[19][7]=1; m[19][14]=1;
        m[18][1]=-2; m[18][2]=-1; m[18][3]=-1; m[18][8]=2; m[18][15]=2;
        m[17][1]=1; m[17][3]=1; m[17][9]=1; m[17][16]=1;
        m[16][4]=-1; m[16][5]=1; m[16][7]=1;
        m[15][1]=-2; m[15][4]=-1; m[15][5]=1; m[15][6]=1; m[15][8]=2;
        m[14][1]=1; m[14][4]=-1; m[14][6]=-1; m[14][9]=1;
        m[13][2]=-1; m[13][3]=2; m[13][5]=-1; m[13][6]=2; m[13][7]=-1; m[13][8]=1;
        m[12][1]=1; m[12][2]=-2; m[12][3]=-1; m[12][5]=-2; m[12][6]=-1; m[12][8]=1; m[12][9]=1;
        m[11][2]=1; m[11][3]=1; m[11][7]=-1; m[11][8]=-1; m[11][9]=-1;
        m[10][4]=1; m[10][5]=1; m[10][6]=1; m[10][7]=1; m[10][8]=1; m[10][9]=1;
        m[8][1]=-2;
        m[7][1]=1;
        m[6][2]=-1; m[6][3]=2;
        m[5][1]=1; m[5][2]=-2; m[5][3]=-1;
        m[4][2]=1; m[4][3]=1;
        m[2][1]=1;

        b = new String[25];
        b[1]="a_4"; b[2]="a_3"; b[3]="a_3+a_4"; b[4]="a_2"; b[5]="a_2+a_3";

        b[6]="a_2+a_3+a_4"; b[7]="a_2+2a_3"; b[8]="a_2+2a_3+a_4"; b[9]="a_2+2a_3+2a_4";

        b[10]="a_1"; b[11]="a_1+a_2"; b[12]="a_1+a_2+a_3"; b[13]="a_1+a_2+a_3+a_4";

        b[14]="a_1+a_2+2a_3"; b[15]="a_1+a_2+2a_3+a_4"; b[16]="a_1+a_2+2a_3+2a_4";

        b[17]="a_1+2a_2+2a_3"; b[18]="a_1+2a_2+2a_3+a_4"; b[19]="a_1+2a_2+2a_3+2a_4";

        b[20]="a_1+2a_2+3a_3+a_4"; b[21]="a_1+2a_2+3a_3+2a_4"; b[22]="a_1+2a_2+4a_3+2a_4";

        b[23]="a_1+3a_2+4a_3+2a_4"; b[24]="2a_1+3a_2+4a_3+2a_4";

        c = new double[25][25];
        lambda = new double[25];
        for(int i=1; i<=24; i++) { lambda[i] = 1;}
        roots = new ArrayList<>();
        roots.add(24);
        n = 0;
        depend = new ArrayList<>();
        Scanner in = new Scanner(System.in);
        System.out.println("0 - sufficient condition\n" + "else - necessary condition");
        x = in.nextInt();
        in.close();

        while(true) {
            last = roots.get(roots.size()-1);
            for(int i=24; i>=last; i--) {
                for(int j=1; j<=24; j++) {
                if(roots.contains(a[i][j])||roots.contains(a[j][i]))

                { c[i][j]=(m[i][j]-m[j][i])*lambda[a[i][j]+a[j][i]];}
                else { c[i][j]=0; }}}
            depend.clear();
            q = true;
            while(q) {
                q = false;
                for(int i=24; i>last; i--) {
                    p = true; dep = 0;
                    for(int j=1; j<=24; j++) {
                        if(c[i][j]!=0 && !depend.contains(j)) {
                            if(dep==0) { dep = j; }
                            else { p = false; break; }}}
                    if(dep!=0 && p) {
                        depend.add(dep); q = true; }}}
            rank = true;
            if(x==0) {
                for(int j=1; j<=24; j++) {
                    if(c[last][j]!=0 && !depend.contains(j)) { rank = false; break; }}}
            else {
                rowold = last;
                colold = 0;
                while(true) {
                    count = 0;
                    for(int j=1; j<=24; j++) {
                        if(c[rowold][j]!=0 && !depend.contains(j) && j!=colold) {
                            rank = false; colnew = j;
                            for(int i=24; i>last; i--) {
                                if(c[i][j]!=0 && i!=rowold) {
                                    rank = true; count++; rownew = i; }}
                            if (!rank) { break; }}}
                    if (!rank || count!=1) { break; }
                    else { rowold = rownew; colold = colnew; }}}
            if(rank) {
                n++;
                for(int h: roots) { System.out.print(b[h]+" "); }
                System.out.println();
                if(last!=1) { roots.add(last-1); }}
            else {
                if(last!=1) {
                    roots.remove((Integer) last);
                    roots.add(last-1); }}
            if(last==1) {
                if(roots.size()!=1) {
                    roots.remove((Integer) 1);
                    last = roots.get(roots.size()-1);
                    roots.remove((Integer) last);
                    roots.add(last-1);}
                else { break; }}}
        System.out.println("empty");
        n++;
        System.out.print(n); }
    int x; int[][]a; int[][]m; String[]b; double[][]c; double[]lambda;

    ArrayList<Integer> roots; ArrayList<Integer> depend;

    boolean q; boolean p; int dep; int last; boolean rank; int n; int count; int rowold; int colold;

    int rownew; int colnew;
}
\end{verbatim}

\end{document}